\documentclass[pdflatex,sn-mathphys-num, Numbered]{sn-jnl}

\usepackage{graphicx}
\usepackage{multirow}
\usepackage{amsmath,amssymb,amsfonts}
\usepackage{amsthm}
\usepackage{mathrsfs}
\usepackage[title]{appendix}
\usepackage{xcolor}
\usepackage{textcomp}
\usepackage{manyfoot}
\usepackage{booktabs}
\usepackage{algorithm}
\usepackage{algorithmicx}
\usepackage{algpseudocode}
\usepackage{listings}
\usepackage{verbatim}
\usepackage{commath}
\usepackage{makecell}
\usepackage{todonotes}
\allowdisplaybreaks

\newcommand{\R}{\mathbb R}

\newcommand{\E}{\mathbb E}

\newcommand{\bangle}[1]{\langle #1\rangle}
\newcommand{\calD}{\mathcal{D}_\tau}
\newcommand{\calH}{\mathcal{H}_\tau}

\newcommand{\calT}{\mathcal{T}}
\newcommand{\rme}{\mathrm{e}}
\newcommand{\rmL}{\mathrm{L}}
\newcommand{\osc}{\mathrm{osc}}
\newcommand{\ess}{\mathrm{ess}}

\newcommand{\ZDCT}{\mathrm{DC}}

\renewcommand{\leq}{\leqslant}
\renewcommand{\geq}{\geqslant}
\renewcommand{\le}{\leqslant}
\renewcommand{\ge}{\geqslant}

\newcommand{\rmc}{\mathrm{c}}
\usepackage{array, makecell}

\theoremstyle{thmstyleone}
\newtheorem{theorem}{Theorem}
\newtheorem{proposition}[theorem]{Proposition}
\newtheorem{lemma}{Lemma}

\theoremstyle{thmstyletwo}
\newtheorem{example}{Example}
\newtheorem{remark}{Remark}

\theoremstyle{thmstylethree}
\newtheorem{definition}{Definition}
\newtheorem{assumption}{Assumption}
\begin{document}

\title[Explicit convergence rates of underdamped Langevin dynamics under weighted and weak Poincar\'e--Lions inequalities]{Explicit convergence rates of underdamped Langevin dynamics under weighted and weak Poincar\'e--Lions inequalities}

\author[1]{\fnm{Giovanni} \sur{Brigati}}\email{giovanni.brigati@ist.ac.at}

\author[2,3]{\fnm{Gabriel} \sur{Stoltz}}\email{gabriel.stoltz@enpc.fr}

\author[4]{\fnm{Andi Q.} \sur{Wang}}\email{andi.wang@warwick.ac.uk}

\author*[5]{\fnm{Lihan} \sur{Wang}}\email{lihanw@andrew.cmu.edu}

\affil[1]{ \orgname{Institute of Science and Technology Austria}, \orgaddress{\street{Am Campus 1}, \city{Klosterneuburg}, \postcode{3400},  \country{Austria}}}

\affil[2]{\orgdiv{CERMICS}, \orgname{Ecole des Ponts, IP Paris}, \orgaddress{\postcode{77455}, \city{Marne-La-Vall\'ee}, \country{France}}}

\affil[3]{\orgdiv{Matherials team-project}, \orgname{Inria Paris}, \orgaddress{\city{Paris}}, \country{France}}

\affil[4]{\orgdiv{Department of Statistics}, \orgname{University of Warwick}, \orgaddress{\city{Coventry}, \postcode{CV4 7AL}, \country{United Kingdom}}}

\affil[5]{\orgdiv{Department of Mathematical Sciences}, \orgname{Carnegie Mellon University}, \orgaddress{\city{Pittsburgh}, \postcode{15213}, \state{PA}, \country{United States}}}

\abstract{We study the long-time behavior of the underdamped Langevin dynamics, in the case of so-called \emph{weak confinement}. Indeed, any $\rmL^\infty$ distribution (in position and velocity) relaxes to equilibrium over time, and we quantify the convergence rate. In our situation, the spatial equilibrium distribution does not satisfy a Poincar\'e inequality. Instead, we assume a weighted Poincar\'e inequality, which allows for fat-tail or sub-exponential potential energies. We provide constructive and fully explicit estimates in $\rmL^2$-norm for $\rmL^\infty$ initial data. A key-ingredient is a new space-time weighted Poincar\'e--Lions inequality, entailing, in turn, a weak Poincar\'e--Lions inequality.}


\maketitle

\section{Introduction and main results}

The \emph{underdamped Langevin dynamics} on~$\R^d \times \R^d$ evolves the position~$X \in \R^d$ and the velocity (in fact, momentum)~$V \in \R^d$ as  
\begin{equation}
  \label{eq:langevin}
  \left\{
  \begin{aligned}
    \dif X_t &= \nabla_v \psi(V_t)\dif t, \\
    \dif V_t &=  -\nabla_x \phi(X_t)\dif t - \gamma \nabla_v \psi(V_t)\dif t + \sqrt{2\gamma}\dif B_t, 
  \end{aligned}\right.
\end{equation}
where $\phi(x)$ is a so-called \emph{weakly confining} potential energy, $\psi(v)$ is the density of kinetic energy, that may or may not be weak confining, $\gamma>0$ is the friction coefficient, and $(B_t)_{t \geq 0}$ is a $d$-dimensional standard Brownian motion. Confinement properties of $\phi,\psi$ are related to their growth at infinity, and we will give precise definitions below. 

\smallskip

As $t\to\infty$, under suitable conditions on $\phi$ and $\psi$, the joint distribution of $(X_t,V_t)$ is expected to converge to the equilibrium distribution $\Theta(\dif x \dif v):= \mu(\dif x) \otimes \nu(\dif v) $, where $\mu, \nu$ are the following probability measures, absolutely continuous with respect to the Lebesgue measure:
\begin{equation}\label{eq:gibbseq}
  \mu(\dif x) := \frac{1}{Z_{\mathsf x}}\mathrm{e}^{-\phi(x)} \dif x, \quad \nu(\dif v) := \frac{1}{Z_{\mathsf v}}\mathrm{e}^{-\psi(v)} \dif v, \quad Z_{\mathsf x} = \int_{\R^d} \rme^{-\phi} , \qquad Z_{\mathsf v} = \int_{\R^d} \rme^{-\psi},
\end{equation}
with~$\dif x, \dif v$ Lebesgue measures on~$\R^d$. For any observable function $h_0(x,v) \in \rmL^2(\Theta)$, define~$h(t,x,v) := \E_{x,v}\left[ h_0(X_t,V_t) \right]$, where the expectation is taken over realizations of~\eqref{eq:langevin} starting from~$(X_0,V_0)=(x,v)$. This function then satisfies the backward Kolmogorov equation, also known as the \emph{kinetic Ornstein--Uhlenbeck equation}:
   \begin{equation}\label{VOU}
    \partial_t h + \calT h = -\gamma\nabla_v^\star \, \nabla_v h,
\end{equation}
with initial condition $h_0$, where 
\[
\begin{aligned}
&\nabla_x^\star := - (\nabla_x  - \nabla_x \phi) \, \cdot  , 
 \quad \nabla_v^\star := - (\nabla_v - \nabla_v \psi) \, \cdot,
\end{aligned}
\]
are the adjoint operators in $\mathrm{L}^2(\Theta)$ of $\nabla_x$ and~$\nabla_v,$ respectively, and
\[
\calT := - \nabla_v^\star \nabla_x + \nabla_x^\star \nabla_v = \nabla_x \phi \cdot \nabla_v - \nabla_v \psi \cdot \nabla_x
\]
is antisymmetric with respect to $\rmL^2(\Theta)$.
We finally observe that the \emph{distribution function} $f=f(t,x,v)$ of the system described by \eqref{eq:langevin}, given by $f = h \, \Theta$, solves the \emph{kinetic Fokker--Planck equation}
\begin{equation}
\label{eq:kfp}
    \partial_t f + \mathcal{T} f = \gamma \nabla_v \cdot (\nabla_v f + \nabla_v\psi\, f), 
\end{equation}
with initial condition $f(0,x,v) = f_0 := h_0 \, \Theta.$

\smallskip

Due to the long-time convergence of the law of~$(X_t,V_t)$ as~$t\to \infty$, one also expects~$h(t,\cdot,\cdot)$ to converge to the constant $\int_{\R^d \times \R^d} h_0\dif \Theta$. Without loss of generality, we assume henceforth $\int_{\R^d \times \R^d} h_0\dif \Theta =0$, since the equation~\eqref{VOU} is linear and satisfied by constant functions. As \eqref{VOU} is also mass-preserving, we have $\int_{\R^d \times \R^d} h(t,\cdot,\cdot) \, \dif \Theta = 0$ for all $t \geq 0$.

\smallskip

Inspired by the discussion in \cite{bouin2023mathrm}, we quantify the convergence rate in $\mathrm{L}^2(\Theta)$ for initial data $h_0 \in \mathrm{L}^\infty(\Theta)$, and for various choices of the potentials $\phi,\psi$, according to their confining properties. While formulated under general assumptions, the results of this paper --- collected in Theorems~\ref{thm:weightedpiconv} to~\ref{thm:weakPIconv} below --- can be applied to a few well-known benchmark examples, as given in Table \ref{tab:weakpirates} below. We adopted the same style of presentation as \cite{bouin2023mathrm}. Our estimates are fully constructive, and rely on functional inequalities, involving weak norms of the solution to \eqref{VOU}.  

\begin{table}[htpb]
  \centering
  \begin{tabular}{|c|c|c|c|}
    \hline
    Potential & \makecell{$\psi(v) = \bangle{v}^{\delta}$ \\ $\delta\geq 1$} & \makecell{$\psi(v) = \bangle{v}^{\delta}$ \\ $\delta\in (0,1)$} & \makecell{$\psi(v) =$ \\ $ (d+q)\log\bangle{v}$} \\
    \hline
    \makecell{$\phi(x) = \bangle{x}^\alpha$ \\ $\alpha\geq 1$} & $\exp(-\lambda t)$ & \makecell{$\exp(-ct^{\delta/(2-\delta)})$\\ Example~\ref{ex:v_subex_x_std}} & \makecell{$t^{-q/2}$ \\ Example~\ref{ex:v_poly_x_std}}\\
    \hline
    \makecell{$\phi(x) = \bangle{x}^\alpha$ \\ $\alpha< 1$} & \makecell{$\exp(-ct^{\alpha/(2-\alpha)})$ \\ Example~\ref{ex1}} & \makecell{$\exp(-ct^{\alpha\delta/(2\alpha+2\delta-3\alpha\delta)})$ \\ Example~\ref{ex6}} & \makecell{$t^{-q/2-}$ \\ Example~\ref{example:v_poly_x_subexp}} \\
    \hline
    \makecell{$\phi(x) =$  \\ $(d+p)\log \bangle{x}$} & \makecell{$t^{-p/2}$ \\ Example~\ref{ex:poly_gaus}} & \makecell{$t^{-p/2-}$ \\ Example~\ref{example:v_subexp_x_subexp}} & \makecell{$t^{-pq/(4+2p+2q)}$ \\ Example~\ref{exam:vel_poly_x_poly}}\\
    \hline 
  \end{tabular}\caption{Explicit convergence rates for solutions of~\eqref{VOU} for various benchmark potentials and~$\rmL^\infty$ initial data. The results in the second row are either well-known (the exponentially convergent case) or can be found in Appendix \ref{app:stdspace}.}
  \label{tab:weakpirates}
\end{table}

The original contributions of this paper can be summarized as follows.

\begin{enumerate}
    \item We establish precise stretched-exponential and algebraic convergence rates, which improve or complement the state of the art, see Remark~\ref{rem:cl} for a precise comparison. 
    \item We adapt the hypocoercivity framework of \cite{Albritton2019} to the weak-confinement case. To this extent, we prove new weighted and weak Poincar\'e--Lions inequalities, in Theorem~\ref{thm:wtpi} and Lemma~\ref{lem:weak_Poincare_solutions_VOU}, which depend on the construction of a solution to a suitable divergence equation \cite{amrouche2015lemma}. 
\end{enumerate}

\paragraph{Scientific motivations}
Kinetic equations describe the time evolution of a distribution function $f=f(t,x,v)$ in position $x$ and velocity $v$. The distribution $f$ models a system of interacting/colliding particles, at an intermediate level between the microscopic and the macroscopic scales. In the Boltzmann--Grad limit \cite{grad1949kinetic}, which considers only short-range, strong interactions between particles and neglect re-collisions, we obtain the Boltzmann equation for $f=f(t,x,v)$~\cite{boltzmann1872weitere}
\[
\partial_t f + \calT f = \mathcal Q(f,f), 
\]
where $\calT$ is the transport operator (i.e., free motion of particles), and $\mathcal Q(f,f)$ encodes the collision between particles. The term  $\mathcal Q(f,f)$ pushes $f$ towards the velocity equilibrium $\nu(\dif v) := \mathrm{e}^{-\psi(v)} \, \dif v$, which is left unchanged by $\mathcal Q$, and, together with the transport term $\calT$, pushes $f$ towards global equilibrium $\Theta$. However, due to nonlinearlity of $\mathcal{Q}$, proving this convergence result for Boltzmann's equation is a major challenge.

In our paper, we analyze the long-time convergence to global equilibrium for a simpler model, namely \emph{underdamped Langevin dynamics}. In this case, the distribution function $f$ satisfies a kinetic Fokker--Planck equation, which replaces the collision kernel $\mathcal Q(f,f)$ by the interaction with the velocity equilibrium $\mathcal Q(f,\nu)$, and describes the motion of particles under Newton's second law of motion. The interplay between the transport and interactions is still highly nontrivial, and its treatment requires techniques from \emph{hypocoercivity} (see below). We specialize to the case of \emph{weakly confining potentials $\phi$}, where the control on the spatial marginal requires additional care with respect to the classical setting. 

Our paper provides an explicit quantification of the convergence rate, for weakly-confining, underdamped Langevin dynamics. Our approach is fully constructive, and relies on functional inequalities involving weak norms of~$f$. Beside the main result, we develop the technical toolbox needed in the proof, with new weighted and weak Poincar\'e-like inequalities adapted to the problem under investigation.

On top of the physical motivation, our convergence estimates are relevant for the study of algorithms used in computational statistics (also known as molecular dynamics) and machine learning, where Langevin dynamics are the backbone of many popular methods for sampling from a distribution~$\mu$; see for instance~\cite{LRS10,LM15,LS16,AT17} for various reference works in molecular dynamics where Langevin dynamics play an important role. Compared to the overdamped dynamics
\[
\dif X_t = -\nabla_x \phi(X_t) \dif t + \sqrt{2}\dif W_t,
\]
the kinetic dynamics~\eqref{eq:langevin} allows to construct better numerical schemes~\cite{leimkuhler2013rational, shen2019randomized} with larger step-sizes, which reduces the computational cost \cite{cheng2018underdamped, dalalyan2020sampling}. Intuitively, the addition of momentum is analogous to Nesterov acceleration in optimization~\cite{ma2021there} and enables the particle to more effectively explore the state space; see, for instance, \cite{Cao2023}. In such applications, it is necessary to study weak confinement in the case of \textit{heavy-tailed} targets, as arise, say when performing regression with heavy-tailed noise. For similar analysis in discrete-time algorithms, see~\cite{Andrieu2023}. 

\paragraph{Hypocoercivity}
The kinetic Ornstein--Uhlenbeck equation \eqref{VOU} is characterized by a degenerate dissipative term $-\gamma \nabla_v^\star \nabla_v h$ in the $v$-variable. A simple computation shows that
\begin{equation}\label{eq:ee}
\frac{\dif}{\dif t} \|h(t,\cdot,\cdot)\|^2_{\mathrm{L}^2(\Theta)} = - 2\gamma \|\nabla_v h(t,\cdot,\cdot)\|^2_{\mathrm{L}^2(\Theta)} \leq 0,
\end{equation}
which proves that the $\mathrm{L}^2$-distance from equilibrium decreases along the flow of \eqref{VOU}, together with the fact that the dissipation term $\|\nabla_v h(t,\cdot,\cdot)\|^2_{\mathrm{L}^2(\Theta)}$ is integrable along time. The right-hand side of \eqref{eq:ee} vanishes every time that $h(t,x,v) = \rho(t,x)$, i.e., when $h$ is at local equilibrium. This simple energy estimate is not sufficient to prove that~$h \to h_\infty=0$ for $t \to \infty$, as the spatial variables cannot be controlled in terms of the velocity derivatives only. On the other hand, the transport operator $\calT$ connects the variable $x$ to the variable $v$, allowing a passage of information from the latter to the former. This idea is shared by all the strategies which have been developed to treat degenerate-dissipative equations as \eqref{VOU}.

\smallskip

The first contribution in this direction came with A.~Kolmogorov~\cite{kolmogoroff1934zufallige}, where the fundamental solution of the equation $\partial_t f + v\cdot \nabla_x = \Delta_v f$ on $\mathbb R^{1+2d}$ is derived. In spite of this equation being degenerate-diffusive, regularization properties are recovered also in the $x$-variable.  

\smallskip

L.~H\"ormander~\cite{hormander1967hypoelliptic} provided a general theory for degenerate-elliptic operators $\mathcal{L}$, which are \emph{hypoelliptic}, i.e., such that the solution $u$ to  $\mathcal L u =f$ exists and is unique, and it is regular when~$f$ is. 
Indeed, L.~H\"ormander considered the case of operators  $\mathcal{L} = X_0 + \sum^N_{i=1} X_i^\dagger X_i$, where $X_i$ is a first order differential operator (and~$\dagger$ denotes the flat~$\rmL^2$ adjoint), as the generator of \eqref{VOU} is. For such operators, hypoellipticity is equivalent to the bracket condition
\[
\mathrm{Lie}(X_i)_{i=0}^N = \mathbb R^{1+2d}.
\]
We observe that 
\[
[\calT,\nabla_v] = \calT \nabla_v - \nabla_v \calT = (\nabla_{vv}^2 \psi) \, \nabla_x. 
\]
Hence, provided $(\nabla^2_{vv} \psi)$ is non-degenerate, equation \eqref{VOU} admits a unique, regular solution for all $h_0 \in \mathrm{L}^2(\Theta)$.
The commutator $[\calT,\nabla_v]$ plays a key-role in controlling the $x$-variable for \eqref{VOU}, see~\cite{eckmann_spectral_2003, herau_isotropic_2004, helffer2005hypoelliptic}. Early tools to study the long-time behavior of~\eqref{VOU} included also Lyapunov techniques~\cite{mattingly_ergodicity_2002, wu2001large, bellet2006ergodic}.

\smallskip

While hypoellipticity concerns regularity, the more recent theory of \emph{hypocoercivity} aims at establishing convergence rates, as $t\to \infty$, for hypoelliptic kinetic equations. After~\cite{talay2002stochastic, mouhot2006quantitative}, C.~Villani~\cite{villani_hypocoercivity_2009} developed the $\mathrm{H}^1$-hypocoercivity framework. Relying on commutators, C.~Villani constructed a functional~$\mathcal E$ equivalent to the standard squared $\mathrm{H}^1$-norm, and such that $\frac{\dif}{\dif t} \mathcal E(h(t,\cdot,\cdot)) < 0$ along the flow of $\partial_t h = \mathcal{L} h$. If $\mathcal{L}$ is the generator of \eqref{VOU}, and the measure $\Theta$ satisfies a Poincar\'e inequality, it holds true that $\mathcal{E}(h(t,\cdot,\cdot)) \leq \mathrm{e}^{-\lambda t} \mathcal{E}(h_0)$, so that  
\[
\forall t>0, \qquad  \|h(t,\cdot,\cdot)\|^2_{\mathrm{H}^1(\Theta)} \leq C \, \mathrm{e}^{-\lambda t} \|h_0\|^2_{\mathrm{H}^1(\Theta)}, 
\]
for a \emph{rate} $\lambda>0$ and a pre-factor $C>1$ independent of~$h$.

\smallskip

Later, inspired by the pioneering contribution  \cite{herau_hypocoercivity_2006}, an $\rmL^2$-approach to hypocoercivity was developed by J.~Dolbeault, C.~Mouhot, and C.~Schmeiser in \cite{dolbeault_hypocoercivity_2009, dolbeault2015hypocoercivity}, which applies to linear kinetic equations in the form 
\begin{equation}
\partial_t h + \calT h = \gamma \, Q(h,\nu), \qquad h=h(t,x,v),
\end{equation}
where $Q$ is an interaction kernel with a fixed background $\nu$ (in velocity variable only). 
Two motivations were the study of equations with non-regularizing kernels, and the consistency of the decay estimates for $h$ with the overdamped regime ($\gamma \to \infty)$ of the equation, see also~\cite{grothaus_hypocoercivity_2014, dolbeault2013exponential}.
The \emph{DMS} technique is based on a twisted $\mathrm{L}^2$-functional $\mathcal H = \frac{1}{2} \, \|\cdot\|_{\mathrm{L}^2(\Theta)}^2 + \epsilon (\cdot, A \cdot)_{\mathrm{L}^2}$, for a parameter $\epsilon>0$ chosen such that $\mathcal H$ is equivalent to~$\|\cdot\|_{\mathrm{L}^2(\Theta)}^2$. Let us indicate with $\Pi_v$ the velocity average 
\[
(\Pi_v h)(t,x) := \mathbb{E}_\nu (h(t,x,v)).
\]
The operator $A = (1+(\calT \Pi_v)^\star (\calT \Pi_v))^{-1} (\calT \Pi_v)^\star$ is built in accordance with the shape of the overdamped limit. An interpretation based on a Schur complement, together with a thorough review of the literature, can be found in~\cite{bernard2022hypocoercivity}. Regarding \eqref{VOU}, if $\Theta$ satisfies a Poincar\'e inequality, it holds true that $\frac{\dif}{\dif t} \mathcal{H}(h(t,\cdot,\cdot)) \leq - \lambda \mathcal{H}(h(t,\cdot,\cdot))$, for a constant $\lambda>0$. Such an entropy-entropy production estimate improves on \eqref{eq:ee}.
Hence, for a pre-factor $C>1$,   
\begin{equation}\label{eq:hypodms}
\|h(t,\cdot,\cdot)\|^2_{\mathrm{L}^2(\Theta)} \leq C \, \mathrm{e}^{-\lambda t} \, \|h_0\|^2_{\mathrm{L}^2(\Theta)}.
\end{equation}
The scaling of the convergence rate of \eqref{VOU} can be studied in both limiting regimes~$\gamma \to 0$ and $\gamma\to\infty$~\cite{grothaus_hypocoercivity_2014, dolbeault2013exponential,IOS19}. 
Other techniques for studying \eqref{VOU} include $\Gamma_2$ calculus \cite{baudoin_bakry-emery_2013}, matrix inequalities~\cite{arnold2024exponential}, and coupling \cite{eberle_couplings_2019}.
\paragraph{The variational approach}
In \cite{Albritton2019}, a new \emph{variational} $\mathrm L^2$ hypocoercive technique, involving weak norms of solutions to \eqref{VOU} and functional inequalities, is introduced. One important observation is that ---due to hypoellipticity--- the right hand side of \eqref{eq:ee} may vanish only on a Lebesgue-measure zero subset of~$[0,\infty)$. Hence, fixing a time-averaging parameter $\tau>0$, we consider the functional
\[
\forall t>0, \qquad \mathcal{H}_\tau(h,t) = \frac1\tau \int_t^{t+\tau} \|h(s,\cdot,\cdot)\|^2_{\mathrm{L}^2(\Theta)} \, \dif s,
\]
for all functions $h \in \mathrm{L}^2_{\mathrm{loc}}(0,\infty,\mathrm{L}^2(\Theta)).$ 
As in \eqref{eq:ee}, we compute, along the flow of \eqref{VOU},
\begin{equation}\label{eq:ee2}
\forall t>0, \qquad \frac{\dif}{\dif t} \mathcal{H}_\tau(h,t) = -\frac{2 \gamma}{\tau} \int_t^{t+\tau} \|\nabla_v h(s,\cdot,\cdot)\|^2_{\mathrm{L}^2(\Theta)} \, \dif s.
\end{equation}
In addition, directly from the equation, see \cite{Albritton2019,brigati2023time,Cao2023}, we find that 
\[
\|(\partial_t + \calT)h\|_{\mathrm{L}^2(0,\infty; \mathrm{L}^2(\mu; \mathrm{H}^{-1}(\nu)))} \leq \|\nabla_v h\|^2_{\mathrm{L}^2(0,\infty;\mathrm{L}^2(\Theta))}.
\]
This way, well-posedness of \eqref{VOU} is achieved in the space 
\[
\mathrm{H}^1_{\mathrm{kin}} := \{ h \in \mathrm{L}^2(0,\tau;\mathrm{H}^1(\nu)) \, : \, (\partial_t + \calT)h \in \mathrm{L}^2(0,\tau; \mathrm{L}^2(\mu; \mathrm{H}^{-1}(\nu)))\},
\]
for any $\tau>0$, as shown in \cite{degond1986global}, and later in \cite{Albritton2019}. Let us briefly sketch the hypocoercive scheme of~\cite{Albritton2019}, in the simple case where $\Theta(\dif x,\dif v) = \dif x \otimes (2\pi)^{-d/2} \, \mathrm{e}^{-|v|^2/2} \,\dif v$, and the $x$-variable is confined in a unit torus $\mathbb T^d$, with periodic boundary conditions.
\begin{enumerate}
    \item Let $h$ be a solution to \eqref{VOU} on the interval $[0,\tau]$, for $\tau>0$. Let $U_\tau$ be the normalized Lebesgue measure on $[0,\tau]$:
    \[
    U_\tau(\dif t) = \frac{1}{\tau} \mathbf{1}_{[0,\tau]} \dif t.
    \]
    Then, the macroscopic and microscopic parts of $h$ are split apart:
    \[\|h\|^2_{\mathrm{L}^2(U_\tau \otimes \Theta)} = \|h-\Pi_v h \|^2_{\mathrm{L}^2(U_\tau \otimes \Theta)} + \|\Pi_v h\|^2_{\mathrm{L}^2(U_\tau \otimes \Theta)}. 
    \]
    \item The Poincar\'e inequality in the measure $\nu$ implies 
    \[
    \|h-\Pi_v h\|^2_{\mathrm{L}^2(U_\tau \otimes \Theta)} \leq \|\nabla_v h\|^2_{\mathrm{L}^2(U_\tau \otimes \Theta)}.
    \]
    \item For the macroscopic term $\|\Pi_v h\|^2_{\mathrm{L}^2(U_\tau \otimes \Theta)}$, the Poincar\'e--Lions inequality \cite{amrouche2015lemma} on $[0,\tau] \times \mathbb T^d$ yields
    \[
    \|\Pi_v h\|^2_{\mathrm{L}^2(U_\tau \otimes \Theta)} \lesssim \|\nabla_{t,x} \Pi_v h\|^2_{\mathrm{H}^{-1}([0,\tau]\times \mathbb T^d)}.
    \]
    Then, the velocity variable is re-introduced via an averaging lemma, which reads 
    \[
    \|\nabla_{t,x} \Pi_v h\|^2_{\mathrm{H}^{-1}([0,\tau]\times \mathbb T^d)} \lesssim \|\nabla_v h\|^2_{\mathrm{L}^2(U_\tau \otimes \Theta)} + \|(\partial_t +\calT) h\|^2_{\mathrm{L}^2(U_\tau \otimes \mu; \mathrm{H}^{-1}(\nu))}. 
    \]
    \item Combining the previous passages gives a \emph{twisted Poincar\'e inequality} for solutions to \eqref{VOU}.
    \[\|h\|^2_{\mathrm{L}^2(U_\tau \otimes \Theta)} \lesssim \|\nabla_v h\|^2_{\mathrm{L}^2(U_\tau \otimes \Theta)}.
    \]
    This way, the entropy-entropy inequality estimate \eqref{eq:ee2} leads to 
    \[
    \forall t>0, \qquad \frac{\dif}{\dif t} \mathcal{H}_\tau(h,t) \leq - \lambda \,\mathcal{H}_\tau(h,t), 
    \]
    for some positive constant $\lambda$.
    Finally,
    \[
    \mathcal{H}_\tau(h,t) \leq \mathrm{e}^{-\lambda t} \|h_0\|^2_{\mathrm{L}^2(\Theta)}.
    \]
\end{enumerate}
We note that the pre-factor $C>1$ ---typical of classical hypocoercive estimates--- is not present, as the reference $\mathrm{L}^2$-norm is never changed. One advantage is that the parameter $\tau$ can be adjusted quite explicitly, optimizing~$\lambda$ and~$C$ in~\eqref{eq:hypodms} may be complicated~\cite{achleitner2017optimal}. 

\paragraph{Constructive estimates and second-order lifts}

The scheme in \cite{Albritton2019} comes without explicit estimates for the rates. In \cite{brigati2023time}, a fully constructive approach is studied for the model-case described above, where particles are confined in the periodic box $\mathbb T$.
In~\cite{Cao2023}, the first constructive estimates are established for \eqref{VOU} in presence of a spatial potential $\phi$, such that $\mu$ satisfies a \emph{Poincar\'e inequality}: \begin{equation}\label{eq:pimu}
  \forall g\in \mathrm{H}^1(\mu), \qquad \textrm{Var}_\mu (g) \le \frac{1}{C_{\mathrm P,\mu}}\mu(|\nabla g|^2),
\end{equation}
for some constant~$C_{\mathrm P,\mu}>0$, which essentially amounts to (super-)linear growth of $\phi$ at infinity (see~\cite[Proposition~4.4.2]{bakry_analysis_2014}). From the aforementioned works, if $\nu$ is the standard Gaussian measure, it is known that the convergence to equilibrium is exponential. In \cite{Cao2023} it is shown that, if $\phi$ is convex and $C_{\mathrm{P},\mu} \ll 1$, then, by optimizing~$\gamma$, solutions to~\eqref{VOU} converge exponentially with rate~$\lambda = \mathrm{O}(\sqrt{C_{\mathrm{P},\mu}})$. This was previously known only in the Gaussian setting, and shows the optimality of the underdamped Langevin dynamics \eqref{eq:langevin} as a second-order lift~\cite{eberle2024non} of reversible \textit{overdamped Langevin dynamics}:
\begin{equation}
  \dif X_t = -\nabla \phi(X_t)\dif t+ \sqrt{2} \dif B_t.
  \label{eq:overdamped}
\end{equation}
In general, with the covariance matrix
\begin{equation}
  \label{eq:scrm}
  \mathscr{M} := \int_{\R^d} \nabla_v \psi \otimes \nabla_v \psi \dif \nu \in \R^{d \times d},
\end{equation}
Langevin dynamics is the second-order lift of the twisted overdamped dynamics
\[
\dif X_t = -\mathscr{M}\nabla \phi(x_t)\, \dif t + \sqrt{2}\mathscr{M}^{1/2} \, \dif B_t.
\]
We refer to~\cite{brigati2024hypocoercivity} for a more precise discussion of the interplay between hypocoercivity and lifts. 

\smallskip

The framework of \cite{Cao2023} was further developed in  \cite{Brigati2023}, where general and possibly weak confining kinetic energies are treated, and the $\mathrm{H}^1(\mu) \to \rmL^2(\mu)$ compact embedding assumption used in~\cite{Cao2023} is removed. This largely enables the further development of the framework in this paper. Other related works include \cite{lu2022explicit}, where linear Boltzmann-type operator could be tackled, \cite{dietert2022quantitative} which allows the diffusion to be degenerate in various parts of the domain, and \cite{li2024quantum} which considers Lindblad dynamics that arise in quantum systems.  

\paragraph{Main goal: explicit (Sub)Exponential convergence.}

If a Poincar\'e inequality does not hold for either $\mu$ or $\nu$, then convergence of solutions to \eqref{VOU} is no longer exponential \cite{hu2019subexponential, cao2020kinetic, grothaus2019weak, Brigati2023}. In particular, in~\cite{bouin2023mathrm}, the $\mathrm{L}^2$-hypocoercivity approach of \cite{dolbeault2015hypocoercivity} is adapted to the weak-confinement setting for \eqref{VOU}, deriving algebraic convergence rates whenever $h_0\in\mathrm{L}^\infty$. The work \cite{Dietert2023} indicates that the framework of~\cite{Albritton2019} is also applicable to obtain convergence results in the weak-confinement setting, although the so-obtained bounds are known to be implicit, in particular concerning dimension dependence. The goal of this work is to fully harness the strength of the framework of~\cite{Albritton2019, Cao2023, Brigati2023} and obtain explicit convergence rates in the weak-confinement setting.
We treat various cases of potentials $\phi,\psi$, following up on the analysis started in \cite{bouin2023mathrm}.

\paragraph{Assumptions on the potential and kinetic energies.}
Our main assumption is the following \emph{weighted Poincar\'e inequality} on $\mu$. This kind of inequality has been studied in \cite{Muckenhoupt1972,blanchet2007hardy,bobkov2009weighted}. 

\medskip

\begin{assumption}\label{a:wp}
  The potential energy function~$\phi$ is such that~$\rme^{-\phi} \in \rmL^1(\dif x)$.  There exist a function $W\in C^1(\R^d)$ such that $W(x) \geq 1$, 
  \begin{equation}
    \label{int}
    \int_{\R^d} W^\sigma \dif \mu < \infty
  \end{equation}
  for some exponent~$\sigma>0$, and
  \begin{equation}\label{eq:condW}
    \sup_{\R^d} \frac{|\nabla_x W|}{W} \le \theta_W
 \end{equation}
  for some constant $\theta_W \in \mathbb{R}_+$.
  Moreover, there exists a constant $P_W > 0$ such that, for any function $f\in \rmL^2(\mu_W)$ satisfying~$\nabla_x f \in \rmL^2(\mu)$, it holds
  \begin{equation}\label{eq:wp}
    \int_{\R^d} \left(f - \int_{\R^d} f\dif \mu_W \right)^2 \, \dif \mu_W \leq P_W \int_{\R^d} |\nabla_x f|^2 \, \dif \mu,
  \end{equation}
  where
  \[
  \mu_W := \frac{1}{Z_W} W^{-2} \dif \mu, \qquad Z_W = \int_{\R^d} W^{-2} \dif \mu \in (0,1).
  \]
  \end{assumption}

  \medskip

Let us mention two representative potentials for which Assumption~\ref{a:wp} holds true with appropriate choices of function~$W$:
\begin{itemize}
\item if $\phi(x) = \bangle{x}^\alpha$ for some $\alpha \in (0,1)$ (with~$\bangle{x} := \sqrt{1+|x|^2}$), then $W(x) = \bangle{x}^{1-\alpha}$ (see~\cite{bouin2021hypocoercivity}), and \eqref{int} holds for any $\sigma\in(0,\infty)$;
\item if $\phi(x) = (p+d)\log \bangle{x}$ for some $p>0$, then $W(x) = \bangle{x}$ (see~\cite{blanchet2007hardy, bobkov2009weighted}), \eqref{int} holds for~$\sigma \in (0,p)$ and the constant~$P_W$ in~\eqref{eq:wp} is~$2Z_W^{-1}/p$ (see~\cite{bobkov2009weighted}).
\end{itemize}
In both cases the bound~\eqref{eq:condW} holds with $\theta_W=1$. In general we expect~$\theta_W$ to be dimension-free. 

\smallskip

We would like to comment here that our weighted Poincar\'e inequality is known as the ``reversed weighted Poincar\'e inequality'' in \cite{bobkov2009weighted} or ``converse weighted Poincar\'e inequality'' in \cite{cattiaux2010functional}. The weighted Poincar\'e inequality in \cite{bobkov2009weighted, cattiaux2010functional} has the form \[ \int_{\R^d} \Big( f - \int_{\R^d} f\dif \mu\Big)^2 \dif \mu \lesssim \int_{\R^d} |\nabla_x f|^2 W^2 \dif \mu,\] and is stronger than \eqref{eq:wp}. Additionally, we scaled the left-hand side of \eqref{eq:wp} so that $\mu_W$ is now a probability measure. 

\smallskip

We also assume pointwise bounds on the first and second derivatives of~$\phi$, a condition which holds for the above benchmark examples.

\medskip

\begin{assumption}\label{a:hessian}
  The potential $\phi$ belongs to~$C^2(\R^d)$. Moreover, there exist constants $M,L>0$ such that
  \begin{equation}
    \label{eq:lowerbdphi}
    \forall x\in \R^d, \qquad |\nabla_x \phi(x)| \le L, \qquad -M \mathrm{Id} \le \nabla^2_{xx} \phi(x) \le M \mathrm{Id}.
  \end{equation}
\end{assumption}

\medskip

\begin{remark}
The bound on the Hessian, which is essentially the ``Ricci curvature'' of $\mu$, is standard and allows for a simpler presentation of the results compared to works like~\cite{bernard2022hypocoercivity, villani_hypocoercivity_2009}. We believe that our assumed bound on the first derivative is only technical, as it is solely used to prove the~$\rmL^2 \rightarrow \mathrm{H}^2$ elliptic regularity for solutions of \eqref{ellipt} below (more precisely, we use it to show that the right hand side of~\eqref{eq:difellipt} belongs to the correct functional space). Another possible assumption to guarantee that the solution of~\eqref{ellipt} is in~$\mathrm{H}^2$ is that $W|\nabla_{xx}^2 \phi| \lesssim 1$, which is also satisfied for many interesting benchmark examples. 
\end{remark}

\medskip

As for $\psi(v)$, we assume the following moment bounds to prove the crucial averaging lemma (see Lemma~\ref{lm:avg}). 

\medskip

\begin{assumption}\label{ass2}
The function~$\psi\in C^2(\R^d)$ and $\nu(\dif v) = \mathrm{e}^{-\psi(v)} \dif v\in \mathcal{P}_{AC}(\dif v)$. Moreover,
\begin{equation}
\label{eq:integrability_conditions_on_psi}
  \int_{\R^d} |\nabla_v \psi(v)|^4 \, \nu(\dif v) + \int_{\R^d} |\nabla_{vv}^2 \psi(v)|^2 \, \nu(\dif v) < +\infty.
\end{equation}
\end{assumption}

\medskip

By Assumption~\ref{ass2}, the covariance matrix $\mathscr{M}$ defined in \eqref{eq:scrm} is well-defined and positive definite (see for instance the argument in~\cite[Lemma~A.1]{stoltz2018langevin}).

\paragraph{Convergence result for a weighted Poincar\'e inequality.}
We are now in position to state our first convergence result, which is directly based on the weighted Poincar\'e inequality~\eqref{eq:wp} (see Section~\ref{sec:weightedpiproof} for the proof).

\medskip

\begin{theorem}
  \label{thm:weightedpiconv}
  Suppose that the potential $\phi$ and the weight function~$W$ satisfy Assumptions~\ref{a:wp} and~\ref{a:hessian}, and that the kinetic energy~$\psi$ satisfies Assumption~\ref{ass2}. Let $h(t)$ be the solution of \eqref{VOU}, for an initial condition $h_0 \in \rmL^\infty(\Theta)$ such that $\iint_{\R^d \times \R^d} h_0(x,v) \Theta(\dif x \dif v)=0$. 
  Then, 
  \begin{enumerate}[(i)]
  \item If $\nu$ satisfies the Poincar\'e inequality 
    \begin{equation}\label{eq:pinu}
    \mathrm{Var}_\nu(f) \le \frac{1}{C_{\mathrm{P},\nu}}\nu(|\nabla_v f|^2),
    \end{equation}
    then there exists an explicitly computable constant $C>0$, depending on $\phi, \psi,W,\sigma,\gamma,\|h_0\|_{\mathrm{L}^\infty(\Theta)}$, such that
    \begin{equation*}
      \|h(t)\|_{\rmL^2(\Theta)} \le C(1+t)^{-\sigma/4}.
    \end{equation*}\label{enu:thm_i}
  \item If there exists some weight function~$\mathcal{G} \in C^1(\R^d)$ with~$\mathcal{G}(v) \geq 1$, and some $\delta>0$ such that $\int_{\R^d} \mathcal{G}^\delta\dif \nu<\infty$, and $\nu$ satisfies the weighted Poincar\'e inequality
    \begin{equation}\label{eq:wtpinu}
      \int \mathcal{G}^{-2}(v)(f(v)-\nu(f))^2 \, \nu(\dif v) \le P_v \int |\nabla_v f|^2 \dif \nu,
    \end{equation}
    then there exists an explicitly computable constant $C>0$ depending on $\phi, \psi,W,\mathcal{G},\sigma,\gamma,\delta,\|h_0\|_{\mathrm{L}^\infty(\Theta)}$, such that
    \begin{equation*}
      \|h(t)\|_{\rmL^2(\Theta)} \le C(1+t)^{- \frac{\sigma \delta}{4(\sigma+\delta+2)}}.
    \end{equation*}\label{enu:thm_ii}
  \end{enumerate}
\end{theorem}

Note that item~\eqref{enu:thm_i} in the above theorem can be recovered from item~\eqref{enu:thm_ii} if one can choose~$\mathcal{G}=1$ and let $\delta\to\infty$. We would like to comment that the inequality \eqref{eq:wtpinu} is slightly different from the spatial version~\eqref{eq:wp}, as we are not renormalizing $\mathcal{G}^{-2} \nu(\dif v)$ to make it a probability measure, and the average subtracted is with respect to~$\nu$ and not the probability measure proportional to~$\mathcal{G}^{-2} \nu(\dif v)$. These is for simplicity in our estimates; see~\cite[Appendix A]{bouin2021hypocoercivity} for the equivalence of these two inequalities. 

\smallskip

Since~\cite{Albritton2019}, and especially after ~\cite{Cao2023, Brigati2023, dietert2022quantitative, eberle2024non}, the Poincar\'e--Lions inequality on the space-time strip~$[0,\tau] \times (\R^d,\mu)$, which states that the~$\rmL^2$ norm of functions can be controlled by the~$\mathrm{H}^{-1}$ norm of some of its derivatives, is fundamental to establish convergence rates as in Theorem~\ref{thm:weightedpiconv}. We know from~\cite{Dietert2023} that such a Poincar\'e--Lions inequality is still the crucial step in the weak confinement setting, although both the statement and the proof have to be modified. The Poincar\'e--Lions inequality is next combined with an averaging lemma in order to derive some form of Poincar\'e inequality, see Lemma~\ref{lm:avg} in Section~\ref{sec:Space-time-velocity_averaging}. 

\smallskip

We present here the weighted Poincar\'e--Lions inequality needed in our work, with constants explicitly computed (henceforth complementing~\cite{Dietert2023}). The proof of this result can be read in Section~\ref{sec:weightedpi}.

\medskip

\begin{theorem}[Weighted Poincar\'e-Lions inequality]\label{thm:wtpi}
  Suppose that the potential~$\phi$ and the weight function~$W$ satisfy Assumptions~\ref{a:wp} and~\ref{a:hessian}. Introduce the Radon--Nikodym derivative
  \[
  \zeta^2 = \frac{\dif \mu}{\dif \mu_W} = Z_W W^2, 
  \]
  and 
  \[
  \overline{g} = \iint_{[0,\tau] \times\R^d} g(t,x) \, U_\tau(\dif t) \, \mu_W(\dif x).
  \]
  Then, for any $g \in \rmL^2(U_\tau\otimes\mu_W)$,
  \begin{equation}
    \label{eq:wtpl}
    \|g-\overline{g}\|_{\rmL^2(U_\tau\otimes \mu_W)} \leq \mathrm{C}_{\mathrm{Lions}} \left( \|\zeta^{-1}\partial_t g\|_{\mathrm{H}^{-1}(U_\tau\otimes\mu)} + \|\nabla_x g\|_{\mathrm{H}^{-1}(U_\tau\otimes\mu)} \right), 
  \end{equation}
  with 
  \[
  \mathrm{C}_{\mathrm{Lions}} := \sqrt{ \left(1+\frac{2}{Z_W}\right) C_1^2 + (1 +2 \theta_W^2) \, C_0^2 },
  \]
  where $C_0,C_1$ are the exact constants appearing in~\eqref{C_0} and~\eqref{C_1}.
 \end{theorem}
 \medskip

As it is the case in the aforementioned works, the Poincar\'e--Lions inequality we need in this paper is achieved after studying the $\mathrm{L}^2 \to \mathrm{H}^1_0$-regularity of a divergence equation ---which is indeed an equivalent problem, see \cite{amrouche2015lemma}.
In the present work, we analyze the following, in Proposition~\ref{prop:divergence_equation},
  \begin{equation}
    \label{div.}
    \left\{ \begin{aligned}
      - \partial_t F_0 + \zeta^2\nabla_x^\star \nabla_x F_1 & = g, \\
      F_0(t=0,\cdot) = F_0(t=\tau,\cdot) & = 0, \\
      \nabla_x F_1(t=0,\cdot) = \nabla_x F_1(t=\tau,\cdot) & = 0,
    \end{aligned} \right.
  \end{equation}
where $g \in \mathrm{L}^2(U_\tau \otimes \mu)$ is given, and $F_0,F_1$ are the unknowns.

\medskip

\begin{remark}
  It is possible to prove similar results by slightly different proofs, which yield better estimates in terms of scaling of the parameters. When~$\mu$ satisfies a standard Poincar\'e inequality, an improved result was first announced in~\cite{eberle2024non}, with a proof later given in \cite{Eberle2024spacetime}.
\end{remark}

\medskip

\begin{remark}\label{rem:rem2}
  A careful inspection of the proof of Theorem~\ref{thm:weightedpiconv} shows that it is possible to remove the assumption~$h_0\in \rmL^\infty(\Theta)$ and obtain a similar result if one establishes the following moment bounds for solutions of~\eqref{VOU}:
  \begin{align*}
    \sup_{t \geq 0} \int_{\R^d \times \R^d} \left(W(x)^\sigma + \mathcal{G}(v)^\delta \right) h(t,x,v)^2 \, \Theta(\dif x \dif v) < \infty,
  \end{align*}
  provided that the initial condition satisfies 
  \[
  \int_{\mathbb R^d \times \mathbb R^d} \left(W(x)^\sigma + \mathcal{G}(v)^\delta \right) h_0^2(x,v) \, \Theta(\dif x \, \dif v) < \infty. 
  \]
  To the best of our knowledge, these moment estimates are currently open, except when $h_0 \in \mathrm{L}^\infty(\Theta)$, which follows easily from maximum principle. 
\label{rmk:remove_Linfty}
\end{remark}

\medskip

In the next paragraph we discuss an improvement of our strategy, which genuinely and heavily uses the hypothesis $h_0 \in \mathrm{L}^\infty(\Theta)$, instead of simply invoking it in order to ensure uniform moment bounds along the flow of \eqref{VOU}.

\paragraph{Convergence result for a weak Poincar\'e inequality.}
When $\phi$ and $\psi$ are such that the associated probability measures satisfy weighted Poincaré inequalities, Theorem~\ref{thm:weightedpiconv} states that the solution~$h(t)$ of \eqref{VOU} with $h_0\in \rmL^\infty(\Theta)$ converges to equilibrium with an algebraic rate. However, we know from~\cite{grothaus2019weak} that such solutions should converge with a stretched exponential rate, which cannot be easily captured by a simple modification of the proof of Theorem~\ref{thm:weightedpiconv}. This is the motivation for the results which we now present, which allow to obtain stretched exponential rates by combining the hypocoercive framework of~\cite{Albritton2019, Cao2023, Brigati2023} with weak Poincar\'e inequalities~\cite{rockner2001weak}. As shown in \cite{cattiaux2010functional}, such inequalities are a consequence of weighted Poincar\'e inequalities.

\medskip

\begin{definition}
  \label{def:WPI}
  The probability measure~$\mu$ satisfies a weak Poincar\'e inequality if, for any~$f\in \mathrm{H}^1(\mu)\cap \rmL^\infty(\mu)$ and any~$s>0$,
  \begin{equation}
    \int_{\R^d} (f - \mu(f))^2 \dif \mu \leq s \int_{\R^d}|\nabla_x f|^2 \dif \mu + \beta(s) \Phi(f),
    \label{eq:WPI}
  \end{equation}
  where~$\beta:(0,\infty) \to [0,\infty)$ is a decreasing function such that~$\beta(s)\to 0$ as $s\to \infty$, and $\Phi(f):=\|f\|^2_\osc=(\ess_\mu \sup f - \ess_\mu \inf f)^2$ is the squared oscillation.
\end{definition}

\medskip

Let us make precise the corresponding functions~$\beta$ for the benchmark cases under consideration (see Section~\ref{sec:strong_vel_confinement} for a more precise discussion):
\begin{itemize}
\item if $\phi(x) = \bangle{x}^\alpha$ for $\alpha \in (0,1)$, then $\beta(s) = \exp(-cs^{\frac{\alpha}{2(1-\alpha)}})$ for some constant $c>0$;
\item if $\phi(x) = (p+d)\log\bangle{x}$ for some $p>0$, then $\beta(s) = cs^{-\frac{p}{2}}$.
\end{itemize}

Weak Poincar\'e inequalities provide an alternative approach for establishing long-time convergence rates of overdamped Langevin dynamics \eqref{eq:overdamped} with $\rmL^\infty$ initial conditions. These inequalities were formally introduced in~\cite{rockner2001weak}, with ideas that can be traced back to~\cite{Liggett1991}, and have been well-studied for reversible processes (see for instance~\cite{Barthe2005, cattiaux2010functional}), also finding recent applications in the study of discrete-time Markov chains~\cite{Andrieu2023}. For nonreversible processes, weak Poincar\'e inequalitites have also been used previously; see~\cite{hu2019subexponential, grothaus2019weak, Andrieu2021} for underdamped Langevin and piecewise-deterministic Markov processes. In particular, improving the convergence rates of~\cite{grothaus2019weak}, which builds upon the hypocoercivity framework of~\cite{dolbeault2015hypocoercivity}, is one of the motivations behind the present work.

\begin{table}[htpb]
  \centering
  \begin{tabular}{|c|c|c|c|}
    \hline
    potential $\phi(x)$ & \makecell{$W(x)$  for weighted \\ Poincar\'e inequality \eqref{eq:wp}} & \makecell{$\beta(s)$ for weak \\ Poincar\'e Inequality \eqref{eq:WPI}} & convergence rate \\ \hline $\bangle{x}^\alpha$ for $\alpha \in (0,1)$ & \makecell{$\bangle{x}^{1-\alpha}$ \\ \cite[Proposition 3.6]{cattiaux2010functional}} & \makecell{$\exp(-cs^{\alpha/[2(1-\alpha)]})$ \\ \cite[Proposition 4.11]{cattiaux2010functional}}& $\exp(-ct^{\alpha/(2-\alpha)})$ \\ \hline $(p+d)\log \bangle{x}$ for $p>0$ & \makecell{$\bangle{x}$ \\ \cite[Proposition 3.2]{cattiaux2010functional}} & \makecell{$cs^{-p/2}$ \\ \cite[Proposition 4.9]{cattiaux2010functional}}& $ct^{-p/2}$ \\ \hline
  \end{tabular}
  \caption{Choice of weight function~$W(x)$ for the weighted Poincar\'e inequality~\eqref{eq:wp}, $\beta(s)$ for the weak Poincar\'e inequality~\eqref{eq:WPI}, and associated $\rmL^2$ convergence rate for both overdamped~\eqref{eq:overdamped} and underdamped Langevin dynamics~\eqref{VOU} for $\nu$ satisfying the Poincar\'e inequality~\eqref{eq:pinu} and an~$\rmL^\infty$ initial data, for the two benchmark examples of potential~$\phi$. Here the constant $c$ may differ from case to case.} \label{table:rates}
\end{table}

One key advantage of working with weak Poincar\'e inequalities is that they allow for a concrete characterization of the dependence of convergence rates on the tail of the potential~$\phi$; see Table \ref{table:rates} for convergence rates of \eqref{eq:overdamped} for benchmark examples of~$\phi$. The convergence rates in Table~\ref{table:rates} for overdamped Langevin can be obtained by combining the functional inequalities derived by \cite{cattiaux2010functional} and a refined version of \cite[Theorem~2.1]{rockner2001weak} as in our Section~\ref{sec:weak_Poincare_general_case}.

\smallskip

The rates in the last column of Table~\ref{table:rates} match with the lower bound for convergence estimates in total variation obtained in~\cite[Theorems~3.1 and~3.6]{brevsar2024subexponential}, which we believe are also optimal in~$\rmL^2$. In Section~\ref{sec:weakpi}, we show that identical arguments can also be used to study the underdamped Langevin dynamics~\eqref{VOU} despite the loss of coercivity, thanks to the weighted Poincar\'e--Lions inequality~\eqref{eq:wtpl} and the space-time-velocity Poincar\'e-type inequality~\eqref{eq:avg} shown later on in Lemma~\ref{lm:avg}. In particular, when~$\nu$ satisfies the Poincar\'e inequality~\eqref{eq:pinu}, our convergence rates for~\eqref{VOU} in Theorem~\ref{thm:weakPIconv} match with those of~\eqref{eq:overdamped} for the benchmark examples (see the second column of Table~\ref{tab:weakpirates}), therefore improving the estimates of~\cite{grothaus2019weak, cao2020kinetic}.

\begin{table}[htpb]
  \centering
  \begin{tabular}{|c|c|c|c|}
    \hline
    Potential & \makecell{$\psi(v) = \bangle{v}^{\delta}$ \\ $\delta\geq 1$} & \makecell{$\psi(v) = \bangle{v}^{\delta}$ \\ $\delta\in (0,1)$} & \makecell{$\psi(v) =$ \\ $ (d+q)\log\bangle{v}$} \\
    \hline  \makecell{$\phi(x) = \bangle{x}^\alpha$ \\ $\alpha\geq 1$} & \makecell{$\exp(-\lambda t)$ \\ Example~($A_1$)} & \makecell{$\exp(-ct^{\delta/(4-3\delta)})$ \\ Example~($A_1$)} & \makecell{$t^{-1/\theta(q)-}$ \\ Example~($A_2$)} \\
    \hline
    \makecell{$\phi(x) = \bangle{x}^\alpha$ \\ $\alpha< 1$} & \makecell{$\exp(-ct^{\alpha /(8-7\alpha)})$ \\ Example~($A_1$)} & \makecell{$\exp(-ct^{\alpha\delta/(4\alpha+8\delta-11\alpha\delta)})$ \\ Example~($A_1$)} & \makecell{$t^{-1/\theta(q)-}$ \\ Example~($A_2$)} \\
    \hline
    \makecell{$\phi(x) =$  \\ $(d+p)\log \bangle{x}$} & \makecell{$t^{-1/2\theta(p)-}$ \\ Example~($B_1$)} & \makecell{$t^{-1/2\theta(p)-}$ \\ Example~($B_1$)} & \makecell{$t^{-1/(2\theta(p)+\theta(q)+2\theta(p)\theta(q))-}$ \\ Example~($B_2$)}\\
    \hline 
  \end{tabular}\caption{For convenience, we report the rates obtained in~\cite[Example~1.3]{grothaus2019weak}, for direct comparison with our rates in Table~\ref{tab:weakpirates}. Here, $\theta(q):=\frac{d+q+2}{q}\wedge \frac{4q+4+2d}{(q^2-4-2d-2q)^+} \ge \tfrac{2}{q}$.} 
  \label{tab:GW_rates}
\end{table}

We now state the second main convergence result of our work.

\begin{theorem}\label{thm:weakPIconv}
  Suppose that Assumptions~\ref{a:wp}, \ref{a:hessian} and \ref{ass2} hold true, and that~$\nu$ satisfies either the Poincar\'e inequality~\eqref{eq:pinu} or the weak Poincar\'e inequality~\eqref{eq:WPI}. Denote by $h(t) : \mathbb R^d\times \mathbb R^d\to \mathbb R$ the solution at time~$t$ of \eqref{VOU}, for any initial condition $h_0\in \rmL^\infty(\Theta)$. Then there exists a positive nonincreasing function~$\mathsf F$ such that $\mathsf F(t) \to 0$ as~$t \to +\infty$ and 
  \begin{equation*}
    \forall t \geq 0, \qquad \|h(t)\|^2_{\rmL^2(\Theta)} \le \Phi(h_0) \mathsf F(t).
  \end{equation*}
  More precisely, $ \mathsf F$ is defined in \eqref{eq:sfF_def} and is related to the inverse of the function $F_{\mathfrak{a}}$ introduced in Definition~\ref{def:WPI_bits} of Section~\ref{sec:weak_Poincare_general_case}.
\end{theorem}

\medskip

\begin{remark}
  The statements of Theorems~\ref{thm:weightedpiconv} and~\ref{thm:weakPIconv} are similar, with a few key differences however, as the requirement on the initial condition is stronger for Theorem~\ref{thm:weakPIconv}, but the functional inequality satisfied by~$\nu$ is weaker. Indeed, Theorem~\ref{thm:weightedpiconv} in principle does not require $h_0\in \mathrm{L}^\infty(\Theta)$, see Remark~\ref{rmk:remove_Linfty}; while Theorem~\ref{thm:weakPIconv} crucially relies on the fact that~$h_0 \in \rmL^\infty(\Theta)$. On the other hand, Theorem~\ref{thm:weightedpiconv} requires~$\nu$ to satisfy a weighted Poincar\'e inequality, while Theorem~\ref{thm:weakPIconv} requires a weak Poincar\'e in~$\nu$, which is less demanding as a weighted Poincar\'e inequality implies a weak Poincar\'e inequality~\cite{cattiaux2010functional}.
\end{remark}

\medskip
 
While the statement of Theorem~\ref{thm:weakPIconv} looks abstract and not quantitative, we specify the convergence rates for our benchmark examples in Section~\ref{subsec:specific_cases}. Our results are summarized in Table~\ref{tab:weakpirates}. The examples we consider are inspired by those in~\cite{grothaus2019weak}, the rates we obtain being sharper than those of~\cite{grothaus2019weak} in all weak confining cases; for comparisons, see Table~\ref{tab:GW_rates}.

\medskip

\begin{remark}[Comparison with the existing literature]\label{rem:cl}
    The results in~\cite{bouin2023mathrm} provide algebraic decay rates instead of stretched exponential ones for the situations where we obtain the latter rates. At the same time, the initial conditions in~\cite{bouin2023mathrm} need not be in~$\rmL^\infty$, which the authors use only to get uniform moment bounds as it is our case for Theorem~\ref{thm:weightedpiconv}, see the discussion in  Remark~\ref{rem:rem2}.  The work~\cite{cao2020kinetic} establishes slightly worse decay rates than the ones we obtain in Example~\ref{ex1} with however weaker assumptions on~$h_0$; and it does not cover any other example that we consider here.
\end{remark}

\medskip

 When~$\nu$ is Gaussian, building upon the weak Poincar\'e inequalities obtained in~\cite{cattiaux2010functional} and the hypocoercive approach of \cite{villani_hypocoercivity_2009}, the work~\cite{hu2019subexponential} obtains convergence rates identical to the ones we obtain both for Examples~\ref{ex1} and~\ref{ex:poly_gaus}. Our results can be adapted to piecewise deterministic Markov processes, by combining our approach with previously established arguments from works such as~\cite{lu2022explicit, Andrieu2022}, but we do not consider such extensions for simplicity. 

 \smallskip

Similarly to the argument in~\cite{cattiaux2010functional} which allow to derive a weak Poincar\'e inequality from a weighted Poincar\'e inequality, one could also derive on~$[0,\tau] \times (\R^d,\mu)$ a weak Poincar\'e--Lions inequality from the weighted Poincar\'e--Lions inequality~\eqref{eq:wtpl}:
\begin{equation*}
  \left\|g - \overline{g}\right\|_{\rmL^2(U_\tau\otimes\mu)}^2 \le s\left( \|\zeta^{-1}\partial_t g\|^2_{\mathrm{H}^{-1}(U_\tau\otimes\mu)} + \|\nabla_x g\|^2_{\mathrm{H}^{-1}(U_\tau\otimes\mu)} \right) + \widetilde{\beta}(s)\Phi(f),
\end{equation*}
with~$\widetilde{\beta}$ a decreasing function such that~$\widetilde{\beta}(s) \to 0$ as $s\to \infty$ (similarly to~\eqref{eq:WPI}). While we do not explicitly state the above inequality in the proof, the proof of Theorem~\ref{thm:weakPIconv} essentially follows from the combination of the above weak Poincar\'e--Lions inequality and the averaging Lemma~\ref{lm:avg}, together with a Bihari--LaSalle argument. We however write a shorter proof as the algebra can be simplified to some extent by direct estimates.

\section{Weighted Poincar\'e--Lions inequality and averaging Lemma}
\label{sec:weightedpi}

We introduce in this section important technical tools that are used to prove both Theorems~\ref{thm:weightedpiconv} and~\ref{thm:weakPIconv}. As in \cite{Cao2023, Brigati2023, eberle2024non}, the crucial step for obtaining convergence rates is to establish the weighted Poincar\'e--Lions inequality given by Theorem~\ref{thm:wtpi}. This is based on constructing a solution to a divergence equation in the space-time strip~$[0,\tau]\times \R^d$, which is the statement of Proposition~\ref{prop:divergence_equation} (see Section~\ref{sec:div_eq}). This key proposition is then proved in Section~\ref{sec:proof_wtpi}. We conclude this section by stating and proving, in Section~\ref{sec:Space-time-velocity_averaging}, an important intermediate result, namely a space-time-velocity averaging lemma.

\subsection{Solution to divergence equation and proof of Theorem~\ref{thm:wtpi}}
\label{sec:div_eq}

Since the operator~$\nabla_x^\star \nabla_x$ does not admit a spectral gap on~$L^2(\mu)$ (as this would be equivalent to~$\mu$ satisfying a Poincar\'e inequality), one cannot directly follow the approach of~\cite{Brigati2023} and instead needs to consider a modified divergence equation where the weight is accounted for. 

\medskip

\begin{proposition}
  \label{prop:divergence_equation}
  Suppose that the potential~$\phi$ and the weight function~$W$ satisfy Assumptions~\ref{a:wp} and~\ref{a:hessian}, and consider a function~$g \in \rmL^2(U_\tau \otimes \mu_W)$ such that~$\overline{g} = 0$. Then there exist functions $F_0, F_1$ satisfying the equation
  \begin{equation}
    \label{div}
    \left\{ \begin{aligned}
      - \partial_t F_0 + \zeta^2\nabla_x^\star \nabla_x F_1 & = g, \\
      F_0(t=0,\cdot) = F_0(t=\tau,\cdot) & = 0, \\
      \nabla_x F_1(t=0,\cdot) = \nabla_x F_1(t=\tau,\cdot) & = 0,
    \end{aligned} \right.
  \end{equation}
  with estimates
  \begin{equation}
    \label{eq:divestL2}
    \|F_0\|_{\rmL^2(U_\tau \otimes \mu_W)} + \|\nabla_x F_1\|_{\rmL^2(U_\tau \otimes \mu)} \leq C\left(\tau +  \sqrt{\frac{P_W}{1-R_{W,\tau}}} \right) \|g\|_{\rmL^2(U_\tau \otimes \mu_W)},
  \end{equation}
  and 
  \begin{multline} \label{eq:divestH1}
    \|\partial_t F_0\|_{\rmL^2(U_\tau \otimes \mu_W)} + \|\nabla_x F_0\|_{\rmL^2(U_\tau \otimes \mu)} + \|\nabla_x \partial_t F_1\|_{\rmL^2(U_\tau \otimes \mu)} +  \|\nabla_{xx}^2 F_1\|_{\rmL^2(U_\tau \otimes \mu)} \\
    \leq C\left[\sqrt{M}\tau + \frac{1}{\sqrt{1-R_{W,\tau}}}\left(Z_W^{-1/2}+ \sqrt{MP_W} + \frac{1}{1-\rme^{-\tau/\sqrt{P_W}}}\right)\right]\|g\|_{\rmL^2(U_\tau \otimes \mu_W)},
  \end{multline}
  where $C$ denote universal constants which can be made explicit (see~\eqref{C_0}-\eqref{C_1}), and
  \begin{equation}\label{eq:Rwtau}
    R_{W,\tau} = \frac{2 \tau \rme^{-\tau/\sqrt{P_W}}}{\sqrt{P_W}(1-\rme^{-2\tau/\sqrt{P_W}})} \in (0,1).
  \end{equation}
\end{proposition}

The solutions to~\eqref{div} are of course not unique. Proposition~\ref{prop:divergence_equation} allows to prove Theorem~\ref{thm:wtpi} as follows.

\begin{proof}[Proof of Theorem~\ref{thm:wtpi}]
  The result is proved by $\rmL^2(U_\tau\otimes\mu_W)$-duality. Let us denote by~$C_0$ and~$C_1$ the constants respectively in \eqref{eq:divestL2} and \eqref{eq:divestH1}, which are made fully explicit in~\eqref{C_0}-\eqref{C_1}. The estimates~\eqref{eq:divestL2} and~\eqref{eq:divestH1} together with \eqref{eq:condW} imply that, for any $g\in \rmL^2(U_\tau\otimes\mu_W)$ with $\overline{g}=0$, the solutions~$F_0,F_1$ in~\eqref{div} satisfy $\zeta^{-1} F_0 \in \mathrm{H}^1(U_\tau\otimes\mu)$ and~$\nabla_x F_1 \in \mathrm{H}^1(U_\tau\otimes\mu)^d$, with 
  \[
  \begin{aligned}
    \|\zeta^{-1} F_0\|_{\mathrm{H}^{1}(U_\tau\otimes\mu)}^2 
    & \leq  \|F_0\|_{\rmL^2(U_\tau\otimes\mu_W)}^2 + \|\partial_t F_0\|_{\rmL^2(U_\tau\otimes\mu_W)}^2 \\
    &\quad + \left( \sup_{\R^d} \frac{| \nabla_x \zeta |}{\zeta} \|F_0\|_{\rmL^2(U_\tau\otimes\mu_W)} + \|\zeta^{-1} \nabla_x F_0\|_{\rmL^2(U_\tau\otimes\mu)} \right)^2 \\
    & \leq \|F_0\|_{\rmL^2(U_\tau \otimes \mu_W)}^2 + \|\partial_t F_0\|_{\rmL^2(U_\tau\otimes\mu_W)}^2  + \frac{2}{Z_W}\|\nabla_x F_0\|_{\rmL^2(U_\tau\otimes\mu)}^2 \\ &\quad+ 2 \left( \sup_{\R^d} \frac{\left|\nabla_x W\right|^2}{W^2}\right) \|F_0\|_{\rmL^2(U_\tau\otimes\mu_W)}^2 \\
    & \leq \left[ \left(1+\frac{2}{Z_W}\right) C_1^2 + (1 +2 \theta_W^2) \, C_0^2 \right] \|g\|_{\rmL^2(U_\tau \otimes \mu_W)}^2;
  \end{aligned}
  \]
  and 
  \[
  \|\nabla_x F_1\|^2_{\mathrm H^1(U_\tau \otimes \mu)} \leq (C_0^2+C_1^2) \, \|g\|_{\rmL^2(U_\tau \otimes \mu_W)}^2.
  \]
  Therefore, using integration by parts and~\eqref{C_0} as well as the above estimate,
  \begin{align*}
    & \|g\|_{\rmL^2(U_\tau\otimes\mu_W)}^2 \stackrel{\eqref{div}}{=} \iint_{[0,\tau] \times\R^d} g(- \partial_t F_0 + \zeta^2\nabla_x^\star \nabla_x F_1) \dif U_\tau\dif \mu_W \\
    & \qquad = \iint_{[0,\tau] \times\R^d} (\zeta^{-1} \partial_t  g \, \zeta^{-1} F_0 + \nabla_x g \cdot \nabla_x F_1) \dif U_\tau\dif \mu \\
    & \qquad \le \left\|\zeta^{-1}\partial_t g\right\|_{\mathrm{H}^{-1}(U_\tau\otimes\mu)} \left\|\zeta^{-1}F_0 \right\|_{\mathrm{H}^{1}(U_\tau\otimes\mu)} + \|\nabla_x g\|_{\mathrm{H}^{-1}(U_\tau\otimes\mu)}\|\nabla_x F_1\|_{\mathrm{H}^{1}(U_\tau\otimes\mu)} \\
    & \qquad \leq \mathrm{C}_{\mathrm{Lions}}  \|g\|_{\rmL^2(U_\tau\otimes\mu_W)}\left(\left\|\zeta^{-1}\partial_t g\right\|_{\mathrm{H}^{-1}(U_\tau\otimes\mu)} + \|\nabla_x g\|_{\mathrm{H}^{-1}(U_\tau\otimes\mu)} \right),
  \end{align*}
  where we introduced 
  \[
  \mathrm{C}_{\mathrm{Lions}} = \sqrt{ \left(1+\frac{2}{Z_W}\right) C_1^2 + (1 +2 \theta_W^2) \, C_0^2 }.
  \]
  This concludes the proof of~\eqref{eq:wtpl}.
\end{proof}

\subsection{Proof of Proposition~\ref{prop:divergence_equation}}
\label{sec:proof_wtpi}

As laid out in \cite{Cao2023, Brigati2023}, we divide the proof of Proposition~\ref{prop:divergence_equation} into five steps.

\subsubsection*{Step 1: Tensorized weighted space-time Poincar\'e inequality}

Consider any function $z(t,x)$ such that $\partial_t z, \nabla_x z \in \rmL^2(U_\tau \otimes \mu)$ with~$\overline{z}=0$, and denote by
\[
Z(x) = \int_0^\tau z(t,x) \, U_\tau(\dif t).
\]
In view of the standard Poincar\'e--Wirtinger inequality for $U_\tau$ and~\eqref{eq:wp} (noting that~$\mu_W(Z) = 0$),
\begin{align}
  \|z  \|^2_{\rmL^2(U_\tau \otimes \mu_W)} &=  
  \|z-Z\|^2_{\rmL^2(U_\tau \otimes \mu_W)} + \|Z \|^2_{\rmL^2(\mu_W)} \notag \\
  & \leq \frac{\tau^2}{\pi^2}\, \|\partial_t z \|^2_{\rmL^2(U_\tau \otimes \mu_W)} + P_W \, \|\nabla_x Z \|^2_{\mathrm{L}^2(\mu)} \notag\\
  &\leq \frac{\tau^2}{\pi^2} \,\|\partial_t  z\|^2_{\mathrm{L}^2(U_\tau \otimes \mu_W)} + P_W \, \|\nabla_x z\|^2_{\mathrm{L}^2(U_\tau \otimes \mu)} \notag\\
  & \leq \max \left(\frac{\tau^2}{\pi^2} ,  P_W \right) \left( \|\partial_t  z\|^2_{\mathrm{L}^2(U_\tau \otimes \mu_W)} +\|\nabla_x z\|^2_{\mathrm{L}^2(U_\tau \otimes \mu)}  \right), 
  \label{eq:tensorwpi}
\end{align} 
where we used a Cauchy--Schwarz inequality to obtain the second inequality.

\subsubsection*{Step 2: An elliptic estimate} 

We would like to find a solution to the equation~$-\partial_t F_0 + \zeta^2\nabla_x^\star \nabla_x F_1 = g$ for~$g \in \mathrm{L}^2(U_\tau \otimes \mu_W)$ with~$\overline{g} = 0$. A natural way is to look for~$F$ of the form~$F_0=\partial_t u$ and~$F_1 = u$ for some function~$u$ (although the boundary conditions on~$\nabla_x F_1$ will not necessarily be satisfied this way; we take care of this issue later on). The key observation here is that the operator $-\partial^2_{tt} + \zeta^2 \nabla_x^\star \nabla_x$ with Neumann boundary conditions is self-adjoint and positive definite over the space of functions~$u \in \mathrm{L}^2(U_\tau \otimes \mu_W)$ with~$\overline{u}=0$, so that one can solve the equation using a weak formulation and Lax--Milgram's theorem.

\smallskip

More precisely, we claim that we can can find a solution to the equation
\begin{equation}
    \label{ellipt}
    \begin{cases}
    -\partial^2_{tt} u + \zeta^2\nabla_x^\star \nabla_x  u = g, \\
   \partial_t u(0,\cdot) = \partial_t u(\tau,\cdot) = 0, 
   \end{cases}
\end{equation}
which satisfies the estimates
\begin{equation}\label{eq:elliptestL2}
    \|\partial_t u\|_{\rmL^2(U_\tau\otimes\mu_W)}^2 + \|\nabla_x u\|_{\rmL^2(U_\tau\otimes\mu)}^2 \le \max \left(\frac{\tau^2}{\pi^2}, P_W\right) \|g\|_{\rmL^2(U_\tau\otimes\mu_W)}^2,
\end{equation}
and 
\begin{equation}
  \label{eq:elliptestH1}
  \begin{split}
    \|\partial_{tt}^2 u\|_{\rmL^2(U_\tau\otimes\mu_W)}^2 +  \|\nabla_x \partial_{t} u\|_{\rmL^2(U_\tau\otimes\mu)}^2 +\|\nabla^2_{xx} u\|_{\rmL^2(U_\tau\otimes\mu)}^2\\
    \le \left[ 1+4Z_W^{-1} +M\max \left(\frac{\tau^2}{\pi^2} ,  P_W \right) \right] \|g\|_{\rmL^2(U_\tau\otimes\mu_W)}^2.
  \end{split}
\end{equation}
The equation \eqref{ellipt} has the following weak formulation: define the Hilbert space
\begin{equation}\label{eq:spaceV}
    V := \left \{u \in \rmL^2(U_\tau\otimes\mu_W) \, \middle| \, \overline{u}=0, \ \partial_t u \in \rmL^2(U_\tau\otimes \mu_W), \ \nabla_x u \in \rmL^2(U_\tau\otimes \mu) \right\},
\end{equation}
with associated norm
\begin{equation}\label{eq:Vnorm}
  \|u\|_V^2 = \|\partial_t u\|^2_{\rmL^2(U_\tau\otimes\mu_W)} + \|\nabla_x u\|^2_{\rmL^2(U_\tau\otimes\mu)}.
\end{equation}
The tensorized weighted Poincar\'e inequality~\eqref{eq:tensorwpi} indeed ensures that~$V \subseteq \rmL^2(U_\tau\otimes\mu_W)$. 
A function~$u\in V$ is a weak solution of~\eqref{ellipt} if and only if
\begin{equation}
  \label{eq:wkfmellipt}
  \begin{aligned}
    \forall v \in V, \qquad \iint_{[0,\tau] \times \R^d} \partial_t u \, \partial_t v \dif U_\tau \dif \mu_W & + \iint_{[0,\tau] \times \R^d} \nabla_x u \cdot \nabla_x v \dif U_\tau \dif \mu \\
    & = \iint_{[0,\tau] \times \R^d} gv \dif U_\tau \dif \mu_W.
  \end{aligned}
\end{equation}
Note that the left hand side of the above equality is a continuous bilinear form on~$V \times V$, which is coercive by definition of the norm~\eqref{eq:Vnorm}; while the right hand side is a linear form which is continuous in view of the tensorized weighted Poincar\'e inequality~\eqref{eq:tensorwpi}. The Lax--Milgram theorem therefore ensures that~\eqref{eq:wkfmellipt} has a unique solution~$u\in V$. Moreover, from the weak formulation~\eqref{eq:wkfmellipt},
\begin{align*}
  \|\partial_t u\|_{\rmL^2(U_\tau\otimes \mu_W)}^2 + \|\nabla_x u\|_{\rmL^2(U_\tau\otimes \mu)}^2 & = \iint_{[0,\tau]\times\R^d} ug \dif U_\tau \dif\mu_W \\
  & \leq \|u\|_{\rmL^2(U_\tau \otimes \mu_W)}\|g\|_{\rmL^2(U_\tau \otimes \mu_W)},
\end{align*}
so that the weighted space-time Poincar\'e inequality~\eqref{eq:tensorwpi} yields the estimate~\eqref{eq:elliptestL2}.

\smallskip
 
We now establish the claimed~$\mathrm{H}^2$ estimates~\eqref{eq:elliptestH1}. We first prove that $\partial_{tt}^2 u \in \rmL^2(U_\tau \otimes \mu_W)$ and $\nabla_x \partial_t u \in \rmL^2(U_\tau \otimes \mu)$. Note that $w_0=\partial_t u$ satisfies the following PDE with Dirichlet boundary conditions:
\begin{equation}\label{eq:dttellipt}
     \left\{ \begin{aligned}
    -\partial^2_{tt} w_0 + \zeta^2\nabla_x^\star \nabla_x  w_0 & = \partial_t g, \\
   w_0(0,\cdot) = w_0(\tau,\cdot) & = 0. 
   \end{aligned} \right.
\end{equation}
Its weak formulation requires to introduce a slightly different Hilbert space than~\eqref{eq:spaceV}, namely
\begin{equation*}
  \widetilde{V} := \left\{ w \in \rmL^2(U_\tau\otimes \mu_W) \ \middle| \ w(0,\cdot) = w(\tau,\cdot)=0, \ \partial_t w \in \rmL^2(U_\tau\otimes \mu_W), \ \nabla_x w \in \rmL^2(U_\tau\otimes \mu) \right\},
\end{equation*}
still endowed with the norm~\eqref{eq:Vnorm}. Note that the Poincar\'e inequality for $\mathrm{H}^1$ functions of time vanishing at~$t=0$ and~$t=\tau$, integrated in the~$x$ variable, ensures that functions~$w \in \widetilde{V}$ satisfy
\begin{equation}
\label{eq:bounds_L2_w_Poincare_space_time}
\iint_{[0,\tau] \times \R^d} w^2 \dif U_\tau \dif \mu_W \leq \frac{\tau^2}{\pi^2} \iint_{[0,\tau] \times \R^d} \left(\partial_t w\right)^2 \dif U_\tau \dif \mu_W \leq \frac{\tau^2}{\pi^2} \|w\|_V^2.
\end{equation}
Therefore, one can again apply Lax--Milgram's theorem to the weak formulation associated with~\eqref{eq:dttellipt} and obtain that there exists a unique~$w_0 \in \widetilde{V}$ solution of this weak formulation for~$g\in \rmL^2(U_\tau \otimes \mu_W) \subset \widetilde{V}^\star$. The latter inclusion follows from~\eqref{eq:bounds_L2_w_Poincare_space_time} as
\[
\left| \iint_{[0,\tau] \times \R^d} g v \dif U_\tau \dif \mu_W \right| \leq \|g\|_{\rmL^2(U_\tau \otimes \mu_W)}\|v\|_{\rmL^2(U_\tau \otimes \mu_W)} \leq \frac{\tau}{\pi}\|g\|_{\rmL^2(U_\tau \otimes \mu_W)}\|v\|_V.
\]
The comparison of~\eqref{ellipt} and~\eqref{eq:dttellipt} shows that~$w_0=\partial_t u$. Testing \eqref{eq:dttellipt} against~$\partial_t u$ itself leads to
\begin{align*}
  \|\partial_t u\|^2_{V} = \iint_{[0,\tau] \times \R^d} \Big(|\zeta^{-1} \partial_{tt}^2 u|^2 + |\nabla_x\partial_t u|^2\Big) \dif U_\tau \dif\mu = \langle \partial_t g, \partial_t u\rangle_{\widetilde{V}^\star,\widetilde V} \leq \|\partial_t g\|_{\widetilde V^\star} \|\partial_t u\|_{V}.
\end{align*}
Now,
\begin{align*}
\|\partial_t g\|_{\widetilde V^\star} = \sup_{\|z\|_{\widetilde V} \leq 1} \, \langle \partial_t g, z\rangle_{\widetilde V^\star,\widetilde V} = \sup_{\|z\|_{\widetilde V} \leq 1} \, ( g, \partial_t z )_{\mathrm{L}^2(U_\tau \otimes \mu_W)} \leq \|g\|_{\mathrm{L}^2(U_\tau \otimes \mu_W)},
\end{align*}
so that
\begin{equation}
  \label{eq:est2}
  \|\partial_t u\|^2_{V} = \|\partial^2_{tt} u\|^2_{\mathrm{L}^2(U_\tau \otimes \mu_W)} + \|\nabla_x \partial_t u\|^2_{\mathrm{L}^2(U_\tau \otimes \mu)} \leq \|g\|^2_{\mathrm{L}^2(U_\tau \otimes \mu_W)}.
\end{equation}

\smallskip

It remains to show that $\nabla_{xx}^2 u \in \rmL^2(U_\tau\otimes\mu)$. Note first that~$\zeta^2 \nabla_x^\star\nabla_x u = g + \partial_{tt}^2 u \in \rmL^2(U_\tau\otimes\mu_W)$. Therefore, by the triangle inequality and \eqref{eq:est2},
\begin{equation}\label{eq:boundLu}
  \|\zeta^2 \nabla_x^\star\nabla_x u\|_{\rmL^2(U_\tau\otimes\mu_W)}^2 \le   4\|g\|_{\rmL^2(U_\tau\otimes\mu_W)}^2.
\end{equation}
Formally differentiating~\eqref{ellipt} with respect to~$\partial_{x_i}$ for~$1 \leq i \leq d$, we find that $w_i = \partial_{x_i} u$ satisfies the equation
\begin{equation}\label{eq:difellipt}
     \begin{cases}
    -\partial^2_{tt} w_i + \zeta^2\nabla_x^\star \nabla_x  w_i = \partial_{x_i}g -2\zeta \partial_{x_i} \zeta \nabla_x^\star \nabla_x u-\zeta^2 \nabla_x \partial_{x_i} \phi\cdot \nabla_x u =:I_i, \\
   \partial_t w_i(0,\cdot) = \partial_t w_i(\tau,\cdot) = 0.
   \end{cases}
\end{equation}
The latter equation has a unique weak solution in~$V$ provided~$I_i \in V^\star$ with~$\overline{I_i}=0$. In this case, since~$w_i$ and~$\partial_{x_i} u - \overline{\partial_{x_i} u}$ solve the same PDE and both have a vanishing integral with respect to~$U_\tau \otimes \mu_W$, they are equal in the sense of distributions, and finally~$\partial_{x_i} u - \overline{\partial_{x_i} u} \in V$, so that~$\partial_{x_i,x_j}^2 u \in \rmL^2(U_\tau\otimes\mu)$ for all~$1 \leq i,j \leq d$. In order to prove that~$I_i \in V^\star$, we consider an arbitrary smooth and compactly supported test function~$\varphi \in C_\rmc^\infty([0,\tau]\otimes \R^d)$ (which may not vanish at the boundaries~$t=0$ or~$t=\tau$), and use an integration by parts to write
\begin{equation}\label{eq:testphig}
  \begin{split}
    & \iint_{[0,\tau]\times\R^d} \varphi\partial_{x_i}g \dif U_\tau\dif \mu_W = -\iint_{[0,\tau]\times\R^d} g\partial_{x_i}\varphi  \dif U_\tau\dif \mu_W\\
    & \qquad + 2\iint_{[0,\tau]\times\R^d} \varphi g \zeta^{-1}\partial_{x_i}\zeta \dif U_\tau\dif\mu_W + \iint_{[0,\tau]\times\R^d} \varphi g \partial_{x_i}\phi \dif U_\tau\dif \mu_W.
    \end{split}
\end{equation}
The first term on the right hand side of~\eqref{eq:testphig} is bounded by $\|g\|_{\rmL^2(U_\tau\otimes\mu_W)}\|\partial_{x_i}\varphi\|_{\rmL^2(U_\tau\otimes\mu_W)}$. The second term can be paired with the second term of~$I_i$, the resulting quantity being bounded as (using the equation~\eqref{ellipt} satisfied by~$u$)
\begin{align*}
  \left| 2\iint_{[0,\tau]\times\R^d}  \varphi  \zeta^{-1}\partial_{x_i} \zeta (g-\zeta^2 \nabla_x^\star\nabla_x u) \dif U_\tau\dif\mu_W\right|
  = 2 \left| \iint_{[0,\tau]\times\R^d}  \varphi \zeta^{-1}\partial_{x_i} \zeta \partial_{tt}^2 u \dif U_\tau\dif\mu_W\right| & \\
  \qquad = 2 \left| \iint_{[0,\tau]\times\R^d}  \partial_t\varphi \zeta^{-1}\partial_{x_i} \zeta \partial_{t} u \dif U_\tau\dif\mu_W \right| \stackrel{\eqref{eq:condW}}{\le} 2 \theta_W\|\partial_t\varphi\|_{\rmL^2(U_\tau\otimes\mu_W)}\|\partial_t u\|_{\rmL^2(U_\tau\otimes\mu_W)}. &
\end{align*}
Finally, the last term of \eqref{eq:testphig} can be paired with the last term of $I_i$, multiplied by~$\varphi$ and integrated, and then bounded as 
\begin{align*}
  & \left|\iint_{[0,\tau]\times\R^d} \varphi \left(g \partial_{x_i}\phi - \zeta^2 \nabla_x \partial_{x_i}\phi\cdot\nabla_x u\right) \dif U_\tau\dif \mu_W \right| \\
  & = \left|\iint_{[0,\tau]\times\R^d} \varphi (-\partial_{tt}^2 u)(\partial_{x_i}\phi) \dif U_\tau\dif \mu_W +  \iint_{[0,\tau]\times\R^d} \varphi \left(\nabla_x^\star\nabla_x u (\partial_{x_i}\phi) -  \nabla_x \partial_{x_i}\phi\cdot\nabla_x u\right) \dif U_\tau\dif \mu \right|\\
  & = \left| \iint_{[0,\tau]\times\R^d} (\partial_t\varphi)(\partial_{t} u)(\partial_{x_i}\phi) \dif U_\tau\dif \mu_W +  \iint_{[0,\tau]\times\R^d} (\partial_{x_i}\phi) \nabla_x\varphi  \cdot \nabla_x u   \dif U_\tau\dif \mu \right|\\
  & \le \sup_{x\in \R^d}|\nabla_x \phi| \left(\|\partial_t\varphi\|_{\rmL^2(U_\tau\otimes\mu_W)}\|\partial_t u\|_{\rmL^2(U_\tau\otimes\mu_W)} + \|\nabla_x\varphi\|_{\rmL^2(U_\tau\otimes\mu)}\|\nabla_x u\|_{\rmL^2(U_\tau\otimes\mu)} \right).
\end{align*}
Combining the latter estimates, using \eqref{eq:elliptestL2} and the pointwise bound \eqref{eq:lowerbdphi} on $|\nabla_x \phi|$, we can conclude that
\[
\left| \iint_{[0,\tau]\times\R^d} I_i\varphi \dif U_\tau\dif \mu_W \right| \lesssim \|g\|_{\rmL^2(U_\tau\otimes\mu_W)}\|\varphi\|_{V},
\]
which ensures that~$I_i\in V^\star$ as desired. The above calculations also hold for $\varphi=1$, which leads to $\overline{I_i}=0$ since all expressions in the end involve derivatives of~$\varphi$. 

\smallskip

At this stage it has been established that~$\nabla_{xx}^2 u$ and~$\nabla_x^\star\nabla_x u$ both belong to~$\mathrm{L}^2(U_\tau \otimes \mu)$ (for the latter function, this follows from the fact that~$\zeta^2 \nabla_x^\star\nabla_x u \in \mathrm{L}^2(U_\tau \otimes \mu_W)$ and~$\zeta^2 \geq Z_W$). We next use Bochner's formula
\begin{equation}
  \label{eq:bochner}
  |\nabla_{xx}^2 u|^2=\nabla_x u\cdot\nabla_x (\nabla_x^\star\nabla_x u) -(\nabla_x u)^{\top}\nabla_x^2 \phi \nabla_xu -\nabla_x^\star\nabla_x \left(\dfrac{\abs{\nabla_x u}^2}{2}\right),
\end{equation}
and integrate it against $U_\tau\otimes\mu$ to obtain (since~$\zeta^2/Z_W \geq 1$)
\begin{align*}
  \|\nabla^2_{xx} u\|^2_{\mathrm{L}^2(U_\tau \otimes \mu)}
  & \stackrel{\eqref{eq:lowerbdphi}}{\le}  \|\nabla_x^\star\nabla_x u\|_{\rmL^2(U_\tau\otimes\mu)}^2 +  M\|\nabla_x u\|^2_{\rmL^2(U_\tau\otimes\mu)}\\
  & = \|\zeta \nabla_x^\star\nabla_x u\|_{\rmL^2(U_\tau\otimes\mu_W)}^2 +  M\|\nabla_x u\|^2_{\rmL^2(U_\tau\otimes\mu)}\\
  & \leq Z_W^{-1} \left\|\zeta^2 \nabla_x^\star\nabla_x u \right\|_{\rmL^2(U_\tau\otimes\mu_W)}^2 +  M\|\nabla_x u\|^2_{\rmL^2(U_\tau\otimes\mu)}\\
  &  \stackrel{\eqref{eq:elliptestL2},\eqref{eq:boundLu}}{\le} \left[ 4Z_W^{-1} +M\max \left(\frac{\tau^2}{\pi^2} ,  P_W \right)\right] \|g\|_{\rmL^2(U_\tau \otimes \mu_W)}^2.
\end{align*}
The inequality~\eqref{eq:elliptestH1} then follows from the latter inequality and~\eqref{eq:est2}.

\subsubsection*{Step 3: An orthogonal decomposition}

In order to follow the construction of~\cite{Brigati2023}, it is helpful to first prove that the operator~$\zeta^2 \nabla_x^\star \nabla_x$ is self-adjoint and positive over the Hilbert space
\[
\mathrm{L}^2_0(\mu_W) = \left\{ z \in \mathrm L^2( \mu_W) \ \middle| \ \int_{\R^d} z \dif \mu_W =0 \right\}.
\] 
It is clear that this operator is symmetric and positive on~$\rmL^2_0(\mu_W)$. For smooth functions~$f,g$ with compact support, consider the bilinear form
\[
Q(f,g) :=  \bangle{f,g}_{\rmL^2(\mu_W)} + \langle \nabla_x f, \nabla_x g \rangle_{\rmL^2(\mu)} =\bangle{f, \zeta^2 \nabla_x^\star \nabla_x g}_{\rmL^2(\mu_W)} + \bangle{f,g}_{\rmL^2(\mu_W)} .
\]
By density, we can extend $Q$ to the Hilbert space (note that this is not the usual Sobolev space)
\[
\mathrm{H}_1 := \left\{ f \in \rmL^2(\mu_W) \ \middle| \ \nabla_x f \in \rmL^2(\mu) \right\},
\]
and notice that $Q(f,f)<\infty$ if and only if $f\in \mathrm{H}_1$. This allows us to consider the Friedrichs extension (see for example \cite[Section 2.3]{teschl2014mathematical}) to define $\mathrm{Id}+\zeta^2 \nabla_x^\star \nabla_x$ as a self-adjoint operator on $\rmL^2(\mu_W)$ with maximal domain; and hence define $\zeta^2 \nabla_x^\star \nabla_x$ as well. This allows to construct
\[
L := \sqrt{\zeta^2 \nabla_x^\star \nabla_x}
\]
using functional calculus. Note moreover that $\zeta^2 \nabla_x^\star \nabla_x$ is coercive over~$\mathrm{L}^2_0(\mu_W)$ since, for any $f \in \mathrm{L}^2_0(\mu_W)$ in the domain of the operator, the weighted Poincar\'e inequality~\eqref{eq:wp} implies 
\begin{equation}
  \label{eq:coercivity_L}
  \|L f\|^2_{\rmL_0^2(\mu_W)} = \bangle{f,\zeta^2 \nabla_x^\star \nabla_x f}_{\rmL^2(\mu_W)} = \|\nabla_x f \|^2_{\rmL^2(\mu)} \geq \frac{1}{P_W}\|f \|^2_{\rmL^2(\mu_W)}.
\end{equation}
In particular, the operator~$L$ is invertible on~$\mathrm{L}^2_0(\mu_W)$ with
\begin{equation}
  \label{eq:inv_L}
  \|L^{-1}\|_{\rmL_0^2(\mu_W) \rightarrow \rmL_0^2(\mu_W)} \le \sqrt{P_W}.
\end{equation}

With the above preparations, we can consider the following subspace of~$\mathrm{L}^2_0(U_\tau \otimes \mu_W)$:
\begin{equation}\label{eq:spaceN}
  \mathcal{N} := \left\{ z \in \mathrm{L}^2_0(U_\tau \otimes \mu_W) \ \middle| \ \exists z_{\pm} \in \mathrm{L}^2_0(\mu_W) \, : \, z = \mathrm{e}^{-tL} z_+ + \mathrm{e}^{-(\tau-t)L} z_-  \right\}.
\end{equation}
Referring to the statement of Proposition~\ref{prop:divergence_equation}, we claim that, if $g \in \mathcal{N}^\perp$, namely
\[
\forall z\in \mathcal{N}, \qquad \iint_{[0,\tau] \times \R^d} gz \dif U_\tau \dif \mu_W =0,
\]
then~$(F_0,F_1):=(\partial_t u, u)$ with~$u$ the solution to~\eqref{ellipt} with right hand side~$g$, provides a solution to~\eqref{div}. This choice ensures that the estimates~\eqref{eq:elliptestL2} and~\eqref{eq:elliptestH1} hold true, and so that~\eqref{eq:divestL2} and~\eqref{eq:divestH1} are respectively satisfied.  Since~$F_0(0,\cdot) = F_0(\tau,\cdot) = 0$ by the boundary conditions in~\eqref{ellipt}, it therefore remains to prove that~$ \nabla_x F_1(0,\cdot) = \nabla_x F_1(\tau,\cdot) = 0$, \emph{i.e.} $\nabla_x u(0,\cdot) = \nabla_x u(\tau,\cdot) = 0$. 

\smallskip

To prove the latter property, note first that, for any~$z \in \mathcal{N}$, it holds $(-\partial^2_{tt} + \zeta^2 \nabla_x^\star \nabla_x) z = - L^2 z + L^2 z =0$. If~$g \in \mathcal{N}^\perp$ and~$u$ solves~\eqref{ellipt}, then, for all $z \in \mathcal{N}$,
\begin{align*}
  0
  & = \iint_{[0,\tau] \times \R^d} g \, z \dif U_\tau \dif\mu_W   \\
  & =\iint_{[0,\tau] \times \R^d} (-\partial^2_{tt} + \zeta^2 \nabla_x^\star \nabla_x) u \, z \dif U_\tau \dif\mu_W \\
  & = \iint_{[0,\tau] \times \R^d} u \, (-\partial_{tt}^2 + \zeta^2 \nabla_x^\star \nabla_x) z \dif U_\tau \dif\mu_W + \int_{\R^d} \left[u(\tau,x) \, \partial_t z(\tau,x) - u(0,x) \, \partial_t z(0,x)\right] \mu_W(\dif x)  \\
  & = \int_{\R^d} \left[u(\tau,x) \, \partial_t z(\tau,x) - u(0,x) \, \partial_t z(0,x)\right] \mu_W(\dif x)  \\
  & = \int_{\R^d} u(0,x) \, ( L z_+- L \mathrm{e}^{-\tau L} z_- )(x) - u(\tau,x) \, (L \mathrm{e}^{-\tau L} z_+ - L z_-)(x) \, \mu_W(\dif x),
\end{align*}
where we used in the third equality the boundary conditions of \eqref{ellipt}. By choosing~$z_+ = \mathrm{e}^{-\tau L} z_-$, the last equality implies that
\[
\forall z_- \in \mathrm{L}^2_0(\mu_W), \qquad \int_{\R^d} u(\tau,\cdot) \, L (\mathrm{e}^{-2\tau L} -1) z_- \dif \mu_W =0.
\]
As $L^2$ is coercive over~$\mathrm{L}^2_0(\mu_W)$, the operator $L (\mathrm{e}^{-2\tau L} -1)$ is invertible over the same space. Since~$z_-$ is arbitrary, we find that~$L (\mathrm{e}^{-2\tau L} -1)u(\tau,\cdot) =0$ in~$\mathrm{L}^2_0(\mu_W)$. Therefore, $u(\tau,\cdot)$ is a constant function. The same can be said for $u(0,\cdot)$, so that the boundary conditions in \eqref{div} are satisfied.

\smallskip

In order to complete the proof of Proposition~\ref{prop:divergence_equation}, we analyse what happens for~$g \in \mathcal{N}$, and construct an admissible solution to~\eqref{div} in this case, following the strategy of~\cite{Brigati2023,Cao2023}.
 
\subsubsection*{Step 4: The exceptional space}

Assume that $g$ is of the form $g = \mathrm{e}^{-t L} g_+,$ for some function $g_+ \in \mathrm{L}_0^2(\mu_W)$. Our goal is to solve \eqref{div} with $\rme^{-tL} g_+$ playing the role of $g$ and appropriate estimates on $F_0, F_1$. We repeat here the equation to solve for convenience:
\begin{equation}\label{eq:diveqL2}
    -\partial_t F_0 + L^2 F_1 = \mathrm{e}^{-t L} g_+.
\end{equation}
A solution to \eqref{eq:diveqL2} has already been computed in~\cite{Brigati2023}:
\begin{align*}
  F_0(t,\cdot) & = -2L^{-1}(1-\mathrm e^{-\tau L})^{-2} (1-\mathrm e^{-tL})(1-\mathrm e^{-(\tau-t)L})\left(\mathrm e^{-tL}-\frac{1}{2}(1+\mathrm e^{-\tau L})\right)\mathrm e^{-tL} g_+, \\
  F_1(t,\cdot) & = 6L^{-2}(1-\mathrm e^{-\tau L})^{-2} (1-\mathrm e^{-tL})(\mathrm e^{-tL}-\mathrm e^{-\tau L}) \mathrm e^{-tL} g_+.
\end{align*}
These functions indeed make sense since, for any~$t\in [0,\tau]$, the operators $(1-\rme^{-\tau L})^{-1}(1-\rme^{-tL})$, $(1-\rme^{-\tau L})^{-1}(1-\rme^{-(\tau-t)L})$, $(1-\mathrm e^{-\tau L})^{-1}(\mathrm e^{-tL}-\mathrm e^{-\tau L})$ and~$\mathrm{e}^{-tL}-\frac{1}{2}(1+\rme^{-\tau L})$ are all contractions on~$\rmL^2_0(\mu_W)$. Moreover, the operators~$L^{-1}$ and~$L^{-2}$ are bounded on~$\rmL^2_0(\mu_W)$, so that~$F_0(t,\cdot)$ and~$F_1(t,\cdot)$ are well defined and both belong to~$\rmL^2_0(\mu_W)$. Overall, $F_0$ and~$F_1$ belong to~$\mathrm{L}^2_{0,0}( U_\tau \otimes \mu_W) := \mathrm{L}^2( U_\tau,\rmL^2_0(\mu_W)) \subseteq \rmL^2_0(U_\tau \otimes \mu_W)$. The bound~\eqref{eq:inv_L} allows us to control~$F_0$ as
\begin{equation}
  \label{eq:L2F0}
  \|F_0\|_{\rmL^2(U_\tau  \otimes \mu_W)} \leq 2\sqrt{P_W} \|\rme^{-tL}g_+\|_{\mathrm{L}^2(U_\tau \otimes \mu_W)},
\end{equation}
and, in view of the equality~$\partial_t F_0 = L^2 F_1 - \mathrm{e}^{-t L} g_+$, 
\begin{equation}
  \label{eq:dtF0}
  \|\partial_t F_0\|_{\mathrm{L}^2( U_\tau \otimes \mu_W)} \leq \|\rme^{-tL} g_+\|_{\rmL^2(U_\tau\otimes\mu_W)} + \|L^2 F_1\|_{\rmL^2(U_\tau\otimes\mu_W)} \le 7\|\rme^{-tL} g_+\|_{\rmL^2(U_\tau\otimes\mu_W)}.
  \end{equation}

  \smallskip
  
To control~$\|\nabla_x  F_0\|_{\rmL^2(U_\tau\otimes\mu)}$, it suffices to bound the norm of the operator~$\nabla_x L^{-1}$ considered on~$\mathrm{L}^2_{0,0}( U_\tau \otimes \mu_W)$ since the remaining part of~$F_0$ is a contraction in~$\rmL^2_0(\mu_W)$ for any~$t \in [0,\tau]$. Now, for any~$f \in \mathrm{L}^2_{0,0}(U_{\tau} \otimes \mu_W)$ (note that~$L^{-1} f(t,\cdot)$ is a well-defined element of~$\rmL^2_0(\mu_W)$ for almost every~$t \in [0,\tau]$),   
\begin{align*}
  \|\nabla_x L^{-1}f\|^2_{\mathrm{L}^2(U_\tau \otimes \mu)}
  & = \iint_{[0,\tau] \times \R^d} \nabla_x L^{-1} f \cdot \nabla_x L^{-1} f \dif U_\tau \dif \mu  \\
  & = \iint_{[0,\tau] \times \R^d} \left(\nabla_x^\star \nabla_x L^{-1} f \right) L^{-1} f \dif U_\tau \dif \mu \\
  & = \iint_{[0,\tau] \times \R^d} \left( \zeta^2 \nabla_x^\star \nabla_x L^{-1} f \right) L^{-1} f  \dif U_\tau \dif \mu_W \\
  & = \iint_{[0,\tau] \times \R^d} \left( L^2 L^{-1} f\right) L^{-1} f \dif U_\tau \dif \mu_W  = \|f\|^2_{\mathrm{L}^2( U_\tau \otimes \mu_W)}.
  \stepcounter{equation} \tag{\theequation} \label{eq:dxLinvbd}
\end{align*}
Therefore, with $\rme^{-tL} g_+$ playing the role of $f$ in \eqref{eq:dxLinvbd}, we arrive at 
    \begin{equation}\label{eq:dxF0}
        \|\nabla_x F_0\|_{\mathrm{L}^2(U_\tau \otimes \mu)} \le 2\|\rme^{-tL}g_+\|_{\rmL^2(U_\tau\otimes\mu_W)}.
    \end{equation}

Let us next turn to~$\nabla_x F_1$ and its derivatives. To alleviate the notation, we introduce for~$t \in [0,\tau]$ the operator
\[
B(t) := 6(1-\mathrm e^{-\tau L})^{-2} (1-\mathrm e^{-tL})(\mathrm e^{-tL}-\mathrm e^{-\tau L}),
\]
which is bounded on~$\rmL^2_0(\mu_W)$ with~$\|B(t)\|_{\rmL^2_0(\mu_W) \rightarrow \rmL^2_0(\mu_W)} \le 6$. The bound on~$\nabla_x F_1$ is obtained similarly to the bound on~$\nabla_x F_0$ as
\begin{align*}
  \|\nabla_x F_1\|^2_{\rmL^2(U_\tau\otimes\mu)}
  & = \int_0^\tau \| \nabla_x L^{-2} B(t) \, \rme^{-tL}g_+\|^2_{\rmL^2(\mu)} U_\tau(\dif t) \\
  & \stackrel{\eqref{eq:dxLinvbd}}{= } \int_0^\tau \| L^{-1} B(t) \, \rme^{-tL}g_+\|^2_{\rmL^2(\mu_W)} U_\tau(\dif t) \\
  & \stackrel{\eqref{eq:inv_L}}{\le} P_W \int_0^\tau \| B(t) \, \rme^{-tL}g_+\|^2_{\rmL^2(\mu_W)} U_\tau(\dif t) \\
  & \le 36 P_W \| \rme^{-tL}g_+\|^2_{\rmL^2(U_\tau\otimes\mu_W)}. \stepcounter{equation} \tag{\theequation} \label{eq:L2F1}
\end{align*}
As for~$\nabla_{xx}^2 F_1(t) = \nabla_{xx}^2 L^{-2} B(t)\rme^{-tL} g_+$, note that for $g_+ \in \rmL^2_0(\mu_W)$, we have $B(t) \rme^{-t L} g_+ \in \rmL^2_{0,0}(U_\tau \otimes \mu_W)$ and consequently by elliptic regularity theory (identical to Step 2 but only in  the~$x$ variable) $L^{-2}B(t) \rme^{-t L} g_+ \in \mathrm{H}^2(U_\tau \otimes \mu)$. We can therefore again use Bochner's formula~\eqref{eq:bochner} (with $L^{-2} B(t)\rme^{-tL} g_+$ playing the role of~$u$), integrated against $\dif U_\tau \dif\mu$: 
\begin{align*}
  \|\nabla_{xx}^2 F_1\|_{\rmL^2(U_\tau\otimes\mu)}^2
  & \stackrel{\eqref{eq:lowerbdphi}}{\le} \|\nabla_{x}^\star \nabla_x L^{-2} B(t)\rme^{-tL} g_+\|_{\rmL^2(U_\tau\otimes\mu)}^2 +M \| \nabla_x L^{-2}B(t)\rme^{-tL} g_+\|^2_{\rmL^2(U_\tau\otimes\mu)} \\
  & \stackrel{\eqref{eq:dxLinvbd}}{=}  \|\zeta^{-1} B(t)\rme^{-tL} g_+\|_{\rmL^2(U_\tau\otimes\mu_W)}^2 + M\| L^{-1}B(t)\rme^{-tL}g_+\|^2_{\rmL^2(U_\tau\otimes\mu_W)}  \\
  & \stackrel{\eqref{eq:inv_L}}{\le} (Z_W^{-1}  + MP_W) \| B(t)\rme^{-tL} g_+\|^2_{\rmL^2(U_\tau\otimes\mu_W)} \\
  & \le 36(Z_W^{-1}  + MP_W) \|\rme^{-tL} g_+\|^2_{\rmL^2(U_\tau\otimes\mu_W)}. \stepcounter{equation} \tag{\theequation} \label{eq:dxF1}
\end{align*}

Finally, for the time derivative of~$\nabla_x F_1$, we write
\begin{align*}
  \nabla_x \partial_t F_1 & = -6\nabla_x L^{-1}(1-\rme^{-\tau L})^{-2} (1-\rme^{-tL})(\rme^{-tL}-\rme^{-\tau L}) \rme^{-tL} g_+ \\
  & \qquad + 6\nabla_x L^{-1}(1-\rme^{-\tau L})^{-2} \rme^{-tL}(\rme^{-tL}-\rme^{-\tau L}) \rme^{-tL} g_+ \\
  & \qquad - 6\nabla_x L^{-1}(1-\rme^{-\tau L})^{-2} (1-\rme^{-tL})\rme^{-tL} \rme^{-tL} g_+.
\end{align*}
The coercivity inequality~\eqref{eq:coercivity_L} implies by functional calculus that
\[
\|(1-\rme^{-\tau L})^{-1}\|_{\rmL^2_0(U_\tau \otimes\mu_W) \rightarrow \rmL^2_0(U_\tau \otimes\mu_W)} \le (1-\rme^{-\tau/\sqrt{P_W}})^{-1},
\]
which, combined with \eqref{eq:dxLinvbd} and operator inequalities such as~$0 \leq (1-\rme^{-\tau L})^{-1} (1-\rme^{-tL}) \leq 1$, allows us to conclude that
\begin{equation}
  \label{eq:dtF1}
  \|\nabla_x \partial_t F_1\|_{\mathrm{L}^2(U_\tau \otimes \mu)} \leq 6\left(1+\frac{2}{1-\rme^{-\tau/\sqrt{P_W}}}\right)\|\rme^{-tL} g_+\|_{\rmL^2(U_\tau\otimes\mu_W)}.
\end{equation}

\subsubsection*{Step 5: Conclusion}

We combine the estimates of Steps~1 to~4 to finish the proof. For any~$g\in \rmL^2(U_\tau\otimes \mu_W)$ with~$\overline{g}=0$, we perform an orthogonal decomposition with respect to the space~$\mathcal{N}$ defined in \eqref{eq:spaceN}:
\[
g = g^\perp + \rme^{-tL} g_+ + \rme^{-(\tau -t)L}g_-, \qquad g^\perp \in \mathcal{N}^\perp, \qquad g_\pm \in \rmL^2_0(\mu_W).
\]
Recall that~$\overline{z}=0$ whenever $z\in \mathcal{N}$ since the latter functional space is a subspace of~$L^2_0(U_\tau \otimes \mu_W)$. Therefore, $\overline{g^\perp}=0$ as well since $\overline{g}=0$. We can therefore construct a solution of~\eqref{div} by setting
\[
F_0 = \partial_t u + F_0^+ + F_0^-, \qquad F_1 = u + F_1^+ + F_1^-,
\]
where $u$ is constructed via Step~2 and solves~\eqref{ellipt} with $g^\perp$ on the right hand side instead of~$g$, and $F_0^\pm, \ F_1^\pm$ are constructed via Step~4 and solve \eqref{eq:diveqL2} with right hand-sides~$\rme^{-tL}g_+$ and~$\rme^{-(\tau-t)L}g_-$ respectively. Note that $F_0^-$ and~$F_1^-$ satisfy estimates similar to those satisfied by~$F_0^+$ and~$F_1^+$.

\smallskip

We claim that
\begin{equation} \label{eq:combinerhs} 
  \begin{split}
    \| \rme^{-t L} g_+\|_{\mathrm{L}^2(U_\tau\otimes \mu_W)}^2 & +\|  \rme^{-(\tau-t) L} g_-\|_{\mathrm{L}^2(U_\tau\otimes \mu_W)}^2 \\
    & \le \frac{1}{1-R_{W,\tau}}
    \left\|\rme^{-tL}g_+ + \rme^{-(\tau-t) L} g_- \right\|_{\rmL^2(U_\tau\otimes\mu_W)},
  \end{split}
\end{equation}
with $R_{W,\tau}$ defined in \eqref{eq:Rwtau}. To prove this inequality, we first note that, by functional calculus,
\begin{align*}
  \| \mathrm{e}^{-t L} g_+\|_{\mathrm{L}^2(U_\tau\otimes\mu_W)}^2 = \iint_{[0,\tau]\times \R^d} g_+ \rme^{-2tL}g_+  \dif U_\tau \dif \mu_W = \int_{\R^d} g_+\frac{\mathrm{Id}-\rme^{-2\tau L}}{2\tau}L^{-1} g_+ \dif \mu_W.
\end{align*}
We next expand
\begin{align*}
  & \left\| \mathrm{e}^{-t L} g_+  + \mathrm{e}^{-(\tau-t) L} g_- \right\|_{\mathrm{L}^2(U_\tau\otimes \mu_W)}^2 \\
  & = \int_{\R^d} g_+\frac{\mathrm{Id}-\rme^{-2\tau L}}{2\tau}L^{-1} g_+ \dif \mu_W  + \int_{\R^d} g_-\frac{\mathrm{Id}-\rme^{-2\tau L}}{2\tau}L^{-1} g_- \dif \mu_W + 2 \int_{\R^d} g_+ \rme^{-\tau L} g_-  \dif \mu_W \\
  & \geq \int_{\R^d} g_+\frac{\mathrm{Id}-\rme^{-2\tau L}}{2\tau}L^{-1} g_+\dif \mu_W  + \int_{\R^d} g_-\frac{\mathrm{Id}-\rme^{-2\tau L}}{2\tau}L^{-1} g_- \dif \mu_W  \\
  & \qquad  - \int_{\R^d} g_+\rme^{-\tau L}g_+\dif \mu_W  - \int_{\R^d} g_-\rme^{-\tau L} g_- \dif \mu_W,  \stepcounter{equation} \tag{\theequation} \label{eq:combexp1}
\end{align*} 
where we used a Cauchy--Schwarz inequality in the last step. Now, in view of the coercivity inequality~\eqref{eq:coercivity_L} and the fact that the function~$s \mapsto 2s \rme^{-s}/(1-\rme^{-2s})$ is nonincreasing when $s\ge 0$, it holds
\[
2 \tau L \rme^{-\tau L} (1-\rme^{-2\tau L})^{-1} \leq R_{W,\tau}, 
\]
so that
\[
\frac{\mathrm{Id}-\rme^{-2\tau L}}{2\tau}L^{-1} - \rme^{-\tau L} \geq (1-R_{W,\tau})\frac{\mathrm{Id}-\rme^{-2\tau L}}{2\tau}L^{-1}.
\]
Combining this inequality and~\eqref{eq:combexp1} then gives
\[
\begin{aligned}
  & \left\| \mathrm{e}^{-t L} g_+ + \mathrm{e}^{-(\tau-t) L} g_- \right\|_{\mathrm{L}^2(U_\tau\otimes \mu_W)}^2 \\
  & \qquad \geq (1-R_{W,\tau})\left(\| \rme^{-t L} g_+\|_{\mathrm{L}^2(U_\tau\otimes \mu_W)}^2+\|  \rme^{-(\tau-t) L} g_-\|_{\mathrm{L}^2(U_\tau\otimes \mu_W)}^2 \right),
\end{aligned}
\]
which leads to~\eqref{eq:combinerhs}. The estimate \eqref{eq:divestL2} is then a combination of \eqref{eq:elliptestL2}, \eqref{eq:L2F0}, \eqref{eq:L2F1} and~\eqref{eq:combinerhs}, as $g^\perp$ is orthogonal to $\rme^{-tL} g_\pm$ in $\rmL^2(U_\tau\otimes \mu_W)$, and
\begin{align*}
  \|F_0\|_{\rmL^2(U_\tau\otimes\mu_W)}^2  & + \|\nabla_x F_1\|_{\rmL^2(U_\tau\otimes\mu_W)}^2 \\
  & \le 3\Big[\|\partial_t u\|_{\rmL^2(U_\tau\otimes\mu_W)}^2 +\|F_0^+\|_{\rmL^2(U_\tau\otimes\mu_W)}^2  + \|F_0^-\|_{\rmL^2(U_\tau\otimes\mu_W)}^2 \\
  & \qquad + \|\nabla_x u\|_{\rmL^2(U_\tau\otimes\mu_W)}^2 +\|\nabla_x F_1^+\|_{\rmL^2(U_\tau\otimes\mu_W)}^2  + \|\nabla_x F_1^-\|_{\rmL^2(U_\tau\otimes\mu_W)}^2\Big]
  \\ & \le 3\max \left( \frac{\tau^2}{\pi^2}, P_W\right) \|g^\perp\|_{\rmL^2(U_\tau\otimes\mu_W)}^2 \\ & \qquad + 120P_W \left( \|\rme^{-tL} g_+\|_{\rmL^2(U_\tau\otimes\mu_W)}^2 +\|\rme^{-(\tau-t)L} g_- \|_{\rmL^2(U_\tau\otimes\mu_W)}^2 \right) \\
  & \le C_0^2 \|g\|_{\rmL^2(U_\tau \otimes \mu_W)}^2, 
\end{align*}
with
\begin{equation}
  \label{C_0}
  C_0 := \sqrt{3} \max \left\{ \frac{\tau}{\pi}, \sqrt{\frac{ 40 P_W}{1-R_{W,\tau}}} \right\}. 
\end{equation}

\smallskip

Likewise, the estimate~\eqref{eq:divestH1} is a combination of \eqref{eq:elliptestH1}-\eqref{eq:est2}, and~\eqref{eq:dtF0}-\eqref{eq:dxF0}-\eqref{eq:dxF1}-\eqref{eq:dtF1}. More precisely, 
\[
\begin{aligned}
  \|\partial_t F_0& \|_{\rmL^2(U_\tau \otimes \mu_W)}^2 + \|\nabla_x F_0\|_{\rmL^2(U_\tau \otimes \mu)}^2 + \|\partial_t \nabla_x F_1\|_{\rmL^2(U_\tau \otimes \mu)}^2 +  \|\nabla_{xx}^2 F_1\|_{\rmL^2(U_\tau \otimes \mu)}^2 \\
  & \leq 3\Big(\|\partial_{tt}^2 u \|_{\rmL^2(U_\tau \otimes \mu_W)}^2 + 2 \|\nabla_x \partial_t u\|_{\rmL^2(U_\tau \otimes \mu)}^2 + \|\nabla_{xx}^2 u\|_{\rmL^2(U_\tau \otimes \mu)}^2 \\
  & \qquad + \|\partial_t F_0^+\|_{\rmL^2(U_\tau \otimes \mu_W)}^2 + \|\nabla_x F_0^+\|_{\rmL^2(U_\tau \otimes \mu)}^2 + \|\nabla_x \partial_t F_1^+\|_{\rmL^2(U_\tau \otimes \mu)}^2 +  \|\nabla_{xx}^2 F_1^+\|_{\rmL^2(U_\tau \otimes \mu)}^2 \\
  & \qquad + \|\partial_t F_0^-\|_{\rmL^2(U_\tau \otimes \mu_W)}^2 + \|\nabla_x F_0^-\|_{\rmL^2(U_\tau \otimes \mu)}^2 + \|\nabla_x \partial_t F_1^-\|_{\rmL^2(U_\tau \otimes \mu)}^2 +  \|\nabla_{xx}^2 F_1^-\|_{\rmL^2(U_\tau \otimes \mu)}^2\Big) \\
  & \leq \left[6+12 Z_W^{-1} +3M\max \left(\frac{\tau^2}{\pi^2} ,  P_W \right) \right] \|g^\perp\|_{\rmL^2(U_\tau\otimes\mu_W)}^2 \\
  & \qquad + \left( 159 + 108 \left[Z_W^{-1}+MP_W + \left(1+\frac{2}{1-\rme^{-\tau/\sqrt{P_W}}}\right)^2\right]\right) \\ & \qquad \qquad \times \left( \|\rme^{-tL} g_+\|_{\rmL^2(U_\tau\otimes\mu_W)}^2 +\|\rme^{-(\tau-t)L} g_- \|_{\rmL^2(U_\tau\otimes\mu_W)}^2 \right) \\
  & \le C_1^2 \|g\|_{\rmL^2(U_\tau \otimes \mu_W)}^2,
\end{aligned} 
\] 
with
\begin{equation}
  \label{C_1}
  \begin{aligned}
  C_1 := \sqrt{3} & \max\left\{\sqrt{2+4 Z_W^{-1} + M\max \left(\frac{\tau^2}{\pi^2} ,  P_W \right)}, \right. \\
  & \qquad \qquad \left. \sqrt{\frac{1}{1-R_{W,\tau}}}\sqrt{ 53 + 36 \left[Z_W^{-1}+MP_W + \left(1+\frac{2}{1-\rme^{-\tau/\sqrt{P_W}}}\right)^2\right]}\right\}.
  \end{aligned}
\end{equation}
The result then holds via elementary estimates on $C_0,C_1$.

\subsection{Space-time-velocity averaging lemma}
\label{sec:Space-time-velocity_averaging}

In order to use Theorem~\ref{thm:wtpi} to prove the convergence result of Theorem~\ref{thm:weightedpiconv}, an intermediate step is a space-time-velocity averaging lemma, given below. The main idea, which is classical in kinetic theory and can be traced back to at least \cite{golse1988regularity}, is the following: averaging in velocity a solution to~\eqref{VOU} has a regularizing effect. The extra information we recover on the velocity average~$\Pi_v h$ then allows us to control the~$\mathrm{L}^2$ energy of the solution~$h$ by means of the terms appearing in the kinetic equation~\eqref{VOU}. For a solution of~\eqref{VOU}, this inequality directly gives a control on its $\rmL^2$ norm by its energy dissipation. Similar inequalities appeared already in~\cite{Albritton2019,Cao2023, Brigati2023, Dietert2023, eberle2024non}, and our Lemma~\ref{lm:avg} below is a slight adaptation of these estimates. 

\medskip

\begin{lemma}
  \label{lm:avg}
  Consider~$(t,x,v) \mapsto h(t,x,v) \in \rmL^2(U_\tau \otimes \Theta)$. Then, 
  \begin{equation}\label{eq:avg}
    \|\Pi_v h- \overline{\Pi_v h}\|_{\rmL^2(U_\tau\otimes \mu_W)}  \le  C_{0,\tau}\|h-\Pi_v h\|_{\rmL^2(U_\tau\otimes \Theta)} +C_{1,\tau}\|(\partial_t + \calT)h\|_{\rmL^2(U_\tau \otimes \mu; \mathrm{H}^{-1}(\nu))},
  \end{equation}
  where~$\Pi_v$ denotes the projection operator in the~$v$ variable:
  \[
  \left(\Pi_v f\right)(t,x) := \int_{\R^d} f(t,x,v) \, \nu(\dif v),
  \]
  and the constants $C_{0,\tau}, C_{1,\tau}$ have explicit expressions given in~\eqref{eq:c0tau} and~\eqref{eq:c1tau}.
\end{lemma}

\smallskip

We refer to~\cite[Section~7]{Brigati2023} for a discussion about the precise dimensional dependence of $C_{0,\tau},~C_{1,\tau}$.

\begin{proof}
  The proof closely follows the proof of~\cite[Lemma~2]{Brigati2023}. As a first step, we apply the weighted Lions inequality \eqref{eq:wtpl}, with $g = \Pi_v h$, to obtain
  \begin{equation}
    \label{eq:wtplpivh}
    \|\Pi_v h- \overline{\Pi_v h}\|_{\rmL^2(U_\tau\otimes \mu_W)}  \le \mathrm C_{\mathrm{Lions}} \left( \|\zeta^{-1}\partial_t \Pi_v h\|_{\mathrm{H}^{-1}(U_\tau\otimes\mu)} + \|\nabla_x \Pi_v h\|_{\mathrm{H}^{-1}(U_\tau\otimes\mu)} \right).
  \end{equation}
  We next separately bound the two terms on the right hand side of this inequality. We introduce to this end the space of~$H^1(U_\tau \otimes \mu)$ functions with Dirichlet boundary conditions in time:
  \[
  \mathrm{H}^1_\ZDCT(U_\tau \otimes \mu) = \left\{ g \in \mathrm{H}^1(U_\tau \otimes \mu) \, \middle| \, g(0,\cdot) = g(\tau,\cdot) = 0 \right\}.
  \]
  Note that the trace of functions in~$\mathrm{H}^1(U_\tau \otimes \mu)$ is well defined on~$\{0,\tau\} \times \R^d$.

  \smallskip
  
  To bound~$\|\zeta^{-1} \partial_t \Pi_v h\|_{\mathrm{H}^{-1}(U_\tau \otimes \mu)}$, we fix a test function~$z \in \mathrm{H}^1_\ZDCT(U_\tau \otimes \mu)$ with $\|z\|^2_{\mathrm{H}^1(U_\tau \otimes \mu)} \leq 1$; while the estimation of~$\|\nabla_x \Pi_v h\|_{\mathrm{H}^{-1}(U_\tau \otimes \mu)}$ can be performed by taking a test function~$Z =(Z_1,\dots,Z_d)\in \mathrm{H}^1_\ZDCT(U_\tau \otimes \mu)^d$ with $\|Z\|^2_{\mathrm{H}^1(U_\tau \otimes \mu)} \leq 1$. We remind the reader that the test functions~$z,~Z$ do not depend on the~$v$ variable.

  \medskip
  \noindent\textbf{Estimate of the time derivative.}
  For~$\zeta^{-1} \, \partial_t \Pi_v h$, we start by noticing that $\nabla_v \Pi_v h = 0$ and
  \[
  \int_{\R^d} \nabla_v \psi(v) \, \nu(\dif v) = 0,
  \]
  so that, by skew-symmetry of~$\partial_t + \calT$ and the boundary conditions for~$z$ in the integration by parts, 
  \[
  \begin{aligned}
    & \iint_{[0,\tau]\times\R^d} \zeta^{-1} (\partial_t \Pi_v h) \, z \dif U_\tau \dif \mu = \iiint_{[0,\tau]\times\R^d \times \R^d} \left[(\partial_t + \calT) \Pi_v  h \right] \zeta^{-1} z \dif U_\tau \dif \Theta \\
    & \qquad = \iiint_{[0,\tau]\times\R^d \times \R^d} [(\partial_t + \calT)  h]  \, \zeta^{-1}z \dif U_\tau \dif\Theta \\
    & \qquad\qquad + \iiint_{[0,\tau]\times\R^d \times \R^d} (h-\Pi_v h) \, (\partial_t + \calT)\left( \zeta^{-1} z\right)  \dif U_\tau \dif \Theta \\
    & \qquad\leq Z_W^{-1/2} \|(\partial_t + \calT) h\|_{\mathrm{L}^2(U_\tau \otimes \mu,\mathrm{H}^{-1}(\nu))} \|z\|_{\mathrm{L}^2(U_\tau \otimes \mu)} \|1\|_{\mathrm{H}^1(\nu)}
    \\ 
    & \qquad\qquad  +  \|h-\Pi_v h\|_{\mathrm{\mathrm{L}^2}(U_\tau \otimes \Theta)} \left\|(\partial_t + \calT)(\zeta^{-1}z)\right\|_{\mathrm{L}^2(U_\tau \otimes \Theta)}.
  \end{aligned}
  \]
  Since~$\|z\|_{\mathrm{H}^1(U_\tau \otimes\mu)}\le 1$, it remains to bound
  \begin{equation}
    \label{eq:term_to_bound_dt_avg}
    \left\|(\partial_t + \calT)(\zeta^{-1}z)\right\|_{\mathrm{L}^2(U_\tau \otimes \Theta)} \leq Z_W^{-1/2}\| \partial_t z \|_{\mathrm{L}^2(U_\tau \otimes \mu)} + \left\|\nabla_v \psi \cdot \nabla_x(\zeta^{-1} z) \right\|_{\mathrm{L}^2(U_\tau \otimes \Theta)}.
  \end{equation}
  The covariance matrix~$\mathscr{M}$, defined in~\eqref{eq:scrm}, is positive definite by~\cite[Appendix~A]{stoltz2018langevin}. Denoting by~$\rho(\mathscr{M)} > 0$ the largest eigenvalue of~$\mathscr{M}$, the last factor in~\eqref{eq:term_to_bound_dt_avg} can be bounded as
  \[
  \begin{aligned}
    \left\| \nabla \psi \cdot \nabla_x (\zeta^{-1}z) \right\|_{\mathrm{L}^2(U_\tau \otimes \Theta)}^2
    & = \int_0^\tau \int_{\R^d} \nabla_x (\zeta^{-1}z)^\top \mathscr{M} \nabla_x (\zeta^{-1}z) \dif U_\tau \dif \mu \\
    & \leq \rho(\mathscr{M}) \left\| \nabla_x (\zeta^{-1}z) \right\|_{\mathrm{L}^2(U_\tau \otimes \mu)}^2.
  \end{aligned}
  \]
  Finally, since~$\nabla_x(\zeta^{-1}z) = \zeta^{-1}\nabla_x z - \frac{\nabla_x \zeta}{\zeta}\zeta^{-1}z,$ the triangle inequality and~\eqref{eq:condW} imply
  \begin{equation*}
    \left\|\nabla_x (\zeta^{-1}z)\right\|_{\mathrm{L}^2(U_\tau \otimes \mu)}^2 \le 2Z_W^{-1} \left( \|\nabla_x z\|_{\mathrm{L}^2(U_\tau \otimes \mu)}^2 + \theta_W^2 \|z\|_{\mathrm{L}^2(U_\tau \otimes \mu)}^2 \right) \le 2Z_W^{-1}\max\{1,\theta_W^2\}.
\end{equation*}
  By taking the supremum over functions~$z \in \mathrm{H}^1_\ZDCT(U_\tau \otimes \mu)$ with~$\|z\|^2_{\mathrm{H}^1(U_\tau \otimes \mu)} \leq 1$, we finally obtain
  \begin{equation}
    \label{eq:dt_Pi_h_H-1}
    \begin{aligned}
      \| \zeta^{-1}\partial_t \Pi_v h\|_{\mathrm{H}^{-1}(U_\tau \otimes \mu)} & \leq Z_W^{-1/2} \, \|(\partial_t + \calT) h\|_{\mathrm{L}^2(U_\tau \otimes \mu,\mathrm{H}^{-1}(\nu))} \\
      & \ \ + Z_W^{-1/2} \, \left(1 + \sqrt{2\rho(\mathscr{M})}\max\{1,\theta_W\}\right)\|h-\Pi_v h\|_{\mathrm{\mathrm{L}^2}(U_\tau \otimes \Theta)} .
    \end{aligned}
  \end{equation}

  \medskip
  
  \noindent\textbf{Estimate of the spatial derivatives.}
  Turning to the estimation of~$\|\nabla_x \Pi_v h\|_{\mathrm{H}^{-1}(U_\tau \otimes \mu)}$, the result is shown with exactly the same computations as in~\cite[Lemma~2]{Brigati2023}, with the only difference that the role of~\cite[Assumption 4]{Brigati2023} is played here by Assumption~\ref{a:hessian}, which ensures that~$\|Z_i\nabla_x \phi\|_{\rmL^2(U_\tau\otimes\mu)} \le L\|Z_i\|_{\mathrm{H}^1(U_\tau\otimes\mu)}$. Introduce
  \[
  G = \frac{\nabla_v \psi}{\|\nabla_v \psi\|_{\mathrm{L}^2(\nu)}}, \qquad \mathcal{M} = \|\nabla_v \psi\|_{\mathrm{L}^2(\nu)}^2 \mathscr{M}^{-1},
  \qquad
  G^{\mathcal{M}} = \left(G^{\mathcal{M}}_1,\dots,G_d^\mathcal{M}\right) = \mathcal{M}G.
  \] 
  The arguments of \cite[Lemma~2]{Brigati2023} then yield
  \begin{align*}
    & \iint_{[0,\tau]\times\R^d} \nabla_x \Pi_v h  \cdot Z \dif U_\tau \dif \mu \\
    & \le \|\nabla_v \psi\|^{-1}_{\mathrm{L}^2(\nu)} \rho(\mathcal{M})\|G \|_{\mathrm{H}^1(\nu)} \|(\partial_t +\calT)h \|_{\mathrm{L}^2(U_\tau \otimes \mu, \mathrm{H}^{-1}(\nu) )} \\
    & \ \ + \|\nabla_v \psi\|^{-1}_{\mathrm{L}^2(\nu)}  \|h-\Pi_v h\|_{\mathrm{L}^2(U_\tau \otimes \Theta)}\left[\left( \sum_{i=1}^d \left\|G_i^{\mathcal{M}} \nabla_v \psi\right\|^2_{\mathrm{L}^2(\nu)}\right)^{1/2} + L \left(\sum_{i=1}^d \left\|\nabla_v G_i^{\mathcal{M}} \right\|^2_{\mathrm{L}^2(\nu)}\right)^{1/2}\right]  \\
    & \ \ + \rho(\mathcal{M}) \|\nabla_v \psi\|^{-1}_{\mathrm{L}^2(\nu)} \|h-\Pi_v h\|_{\mathrm{L}^2(U_\tau \otimes \Theta)}.
    \stepcounter{equation} \tag{\theequation}  \label{eq:dx_Pi_h_H-1}
  \end{align*}

  \medskip
  \noindent\textbf{Conclusion of the proof.}
  Combining \eqref{eq:wtplpivh}, \eqref{eq:dt_Pi_h_H-1} and taking the supremum over~$Z$ in~\eqref{eq:dx_Pi_h_H-1}, we conclude the proof of~\eqref{eq:avg} with 
  \begin{align}
    C_{0,\tau} = C_{\mathrm{Lions}} & \left( \|\nabla_v \psi\|^{-1}_{\mathrm{L}^2(\nu)} \left[\left( \sum_{i=1}^d \left\|G_i^{\mathcal{M}} \nabla_v \psi\right\|^2_{\mathrm{L}^2(\nu)}\right)^{1/2} + L \left(\sum_{i=1}^d \left\|\nabla_v G_i^{\mathcal{M}} \right\|^2_{\mathrm{L}^2(\nu)}\right)^{1/2}\right] \right. \nonumber \\
  & \left. \phantom{\left[\left( \sum_{i=1}^d \left\|G_i^{\mathcal{M}} \nabla_v \psi\right\|^2\right)^{1/2}\right]}\hspace{-3cm}+ Z_W^{-1/2}\left(1+\sqrt{2\rho(\mathscr{M})}\max\{1,\theta_W\}\right)+ \rho(\mathcal{M}) \|\nabla_v \psi\|^{-1}_{\mathrm{L}^2(\nu)} \right),
  \label{eq:c0tau} \\
  C_{1,\tau} &= C_{\mathrm{Lions}}\left( Z_W^{-1/2} +  \|\nabla_v \psi\|^{-1}_{\mathrm{L}^2(\nu)} \rho(\mathcal{M})\|G \|_{\mathrm{H}^1(\nu)} \right). \label{eq:c1tau} 
\end{align}
Note that all the quantities in~\eqref{eq:c0tau} and~\eqref{eq:c1tau} are finite thanks to Assumption~\ref{ass2}.
\end{proof}

\section{Proof of Theorem \ref{thm:weightedpiconv}: Convergence Result with Weighted Poincar\'e--Lions Inequality}
\label{sec:weightedpiproof}

It is well-known that~\eqref{VOU} is well posed and $h(t,\cdot,\cdot) \in \rmL^2(\Theta)$ for all~$t \geq 0$. In addition, the total mass is conserved:
\begin{equation*}
  \forall t \geq 0, \qquad \iint_{\R^d\times\R^d} h(t,x,v) \, \Theta(\dif x \dif v) = 0,
\end{equation*}
and the $\rmL^2$ norm of the solution is nonincreasing: 
\begin{equation}
  \label{ee}
  \frac{\dif}{\dif t} \iint_{\R^d\times\R^d} h(t,x,v)^2 \, \Theta (\dif x \dif v) = -2\gamma \iint_{\R^d\times\R^d} |\nabla_v h|^2 \, \dif \Theta \le 0.
\end{equation}
Solutions of \eqref{VOU} moreover satisfy the maximum principle:
\begin{equation}\label{eq:Linftybd}
  \forall t \geq 0, \qquad \|h(t,\cdot,\cdot)\|_{\rmL^\infty(\Theta)} \le \|h_0\|_{\rmL^\infty(\Theta)}.
\end{equation} 
As in~\cite{Cao2023, Brigati2023, eberle2024non}, after proving Theorem~\ref{thm:wtpi} and Lemma \ref{lm:avg}, one can prove the long-time convergence of solutions of~\eqref{VOU} using energy estimate on the time-averaged $\rmL^2$-energy
\begin{equation}\label{eq:calHt}
    \mathcal{H}_\tau(t) := \int_t^{t+\tau} \|h(s,\cdot,\cdot)\|^2_{\mathrm{L}^2(\Theta)} \, U_\tau(\dif s).
\end{equation}
An integration in time of~\eqref{ee} leads to the following estimate:
\begin{equation}
  \label{eefull}
  \frac{\dif}{\dif t}\mathcal{H}_\tau(t) = -\mathcal{D}_\tau(t),
\end{equation}
with the energy dissipation
\begin{equation}\label{eq:calDt}
  \mathcal D_\tau(t) := 2 \gamma \int_t^{t+\tau} \|\nabla_v h\|_{\rmL^2(\Theta)}^2 \,U_\tau(\dif s).
\end{equation}
The time-averaged functional inequalities obtained in Section~\ref{sec:weightedpi} allow to upper bound~$\mathcal{H}_\tau(t)$ by a function of~$\mathcal D_\tau(t)$, henceforth leading to a nonlinear differential inequality for~$\mathcal{H}_\tau(t)$, from which convergence is obtained by a Bihari--LaSalle argument. We make this stategy precise, by distinguishing two cases depending on whether~$\nu$ satisfies a Poincar\'e or a weighted Poincar\'e inequality.

\medskip

\begin{proof}[Proof of Theorem \ref{thm:weightedpiconv}]
  We fix~$t,\tau>0,$ and write 
  \begin{equation}\label{eq:L2decom}
    \mathcal{H}_\tau(t) = \int_t^{t+\tau} \|h-\Pi_v h\|^2_{\mathrm{L}^2(\Theta)} \dif U_\tau  + \int_t^{t+\tau} \|\Pi_v h\|^2_{\mathrm{L}^2(\mu)} \dif U_\tau,
\end{equation}
by orthogonal decomposition in $\mathrm{L}^2(U_\tau\otimes\Theta).$

\medskip\noindent\textbf{Case~(i).}
Let us first focus on the case when $\nu$ satisfies the Poincar\'e inequality~\eqref{eq:pinu}. The first contribution in~\eqref{eq:L2decom} is bounded from above by
\[
\frac{1}{C_{\mathrm{P},\nu}}\int_t^{t+\tau} \|\nabla_v h(s,\cdot,\cdot)\|^2_{\mathrm{L}^2(\Theta)} = \frac{1}{2\gamma C_{\mathrm{P},\nu}} \mathcal D_\tau(t).
\]
For the second contribution, we note that $\int_{\R^d} \Pi_v h(t,x) \, \mu(\dif x) =0$ for any~$t \geq 0$, so that, by standard properties of the average, we have, for all $a \in \R$,
\[
\int_t^{t+\tau} \|\Pi_v h(s,\cdot)\|^2_{\mathrm{L}^2(\mu)} \, U_\tau(\dif s) \leq \int_t^{t+\tau} \|\Pi_v h(s,\cdot)-a\|^2_{\mathrm{L}^2(\mu)} \, U_\tau(\dif s).
\]
In particular,
\begin{equation}
  \label{eq:minusavg}
  \int_t^{t+\tau} \|\Pi_v h(s,\cdot)\|^2_{\mathrm{L}^2(\mu)} \, U_\tau(\dif s) \leq \int_t^{t+\tau} \|\Pi_v h(s,\cdot)-\overline{\Pi_v h}\|^2_{\mathrm{L}^2(\mu)} \, U_\tau(\dif s).
\end{equation}
Upon rewriting, for $\sigma$ given by~\eqref{int}, 
\[
(\Pi_v h -\overline{\Pi_v h})^2 = \left[(\Pi_v h -\overline{\Pi_v h})^2 \zeta^{-2} \right]^{\sigma/(2+\sigma)} \left[(\Pi_v h -\overline{\Pi_v h})^2 \zeta^{\sigma} \right]^{2/(2+\sigma)}, 
\]
H\"older's inequality implies that the right hand side of~\eqref{eq:minusavg} is controlled by
\begin{equation}
  \label{eq:holder1}
  \left( \int_t^{t+\tau} \|\Pi_v h(s,\cdot)-\overline{\Pi_v h}\|^2_{\mathrm{L}^2(\mu_W)} \, U_\tau(\dif s) \right)^{\frac{\sigma}{\sigma+2}} \left( \int_t^{t+\tau} \int_{\R^d} (\Pi_v h -\overline{\Pi_v h})^2 Z_W^{\sigma/2} W^\sigma \dif U_\tau \dif \mu  \right)^{\frac{2}{\sigma+2}}.
\end{equation}
 The idea is to use Lemma~\ref{lm:avg} to control the first term of~\eqref{eq:holder1}, while the second term can be bounded by~$\|h_0\|_{\rmL^\infty(\Theta)}$, thanks to~\eqref{int}. For the first term of \eqref{eq:holder1},
\begin{align*}
  & \left( \int_t^{t+\tau} \|\Pi_v h(s,\cdot)-\overline{\Pi_v h}\|^2_{\mathrm{L}^2(\mu_W)} \, U_\tau(\dif s) \right)^{\frac{\sigma}{\sigma+2}} \\
  & \leq  2^{\frac{\sigma}{\sigma+2}} \left( C^2_{0,\tau}\int_t^{t+\tau} \|h-\Pi_v h\|^2_{\rmL^2( \Theta)}\dif U_\tau  + C^2_{1,\tau}\int_t^{t+\tau} \|(\partial_t + \calT)h\|^2_{\rmL^2(\mu;\mathrm{H}^{-1}(\nu))} \dif U_\tau  \right)^{\frac{\sigma}{\sigma+2}} \stepcounter{equation} \tag{\theequation}\label{eq:pivii} \\
  & \leq 2^{\frac{\sigma}{\sigma+2}} \left( C^2_{0,\tau}C_{\mathrm{P},\nu}^{-1}\int_t^{t+\tau} \|\nabla_v h\|^2_{\rmL^2( \Theta)}\dif U_\tau+ \gamma^2 \, C^2_{1,\tau}\int_t^{t+\tau}\| \nabla_v^\star \nabla_v h\|^2_{\rmL^2( \mu; \mathrm{H}^{-1}(\nu))} \dif U_\tau \right)^{\frac{\sigma}{\sigma+2}}   \\
  & \le 2^{\frac{\sigma}{\sigma+2}} \left( C^2_{0,\tau}C_{\mathrm{P},\nu}^{-1} +\gamma^2 C^2_{1,\tau} \right)^{\frac{\sigma}{\sigma+2}} \left(\frac{1}{2\gamma}\mathcal{D}_\tau(t)\right)^{\frac{\sigma}{\sigma+2}} = \left( \frac{C^2_{0,\tau}}{\gamma C_{\mathrm{P},\nu}} +\gamma C^2_{1,\tau} \right)^{\frac{\sigma}{\sigma+2}} \mathcal{D}_\tau(t)^{\frac{\sigma}{\sigma+2}},
\end{align*}
where we used \eqref{eq:avg}, the Poincar\'e inequality~\eqref{eq:pinu} for~$\nu$ and the fact that $h$ solves \eqref{VOU}; as well as the following inequality for~$g \in \rmL^2(\nu)^d$:
\begin{equation}
  \label{eq:nabla_v_star_bound}
  \begin{aligned}
    \| \nabla_v^\star g \|_{\mathrm{H}^{-1}(\nu)} & = \sup_{\| z \|_{\mathrm{H}^1(\nu)} \leq 1} \left\langle\nabla_v^*g,z\right\rangle_{\mathrm{H}^{-1}(\nu),\mathrm{H}^1(\nu)} \\
    & = \sup_{\| z \|_{\mathrm{H}^1(\nu)} \leq 1} \int_{\R^d} g^\top \nabla_v z \dif \nu \leq \|g\|_{\rmL^2(\nu)}.
  \end{aligned}
\end{equation}
On the other hand, since $\|\Pi_v h\|_{\rmL^\infty(\mu)} \le \| h\|_{\rmL^\infty(\Theta)}$, 
\[
\left( \int_t^{t+\tau} \int_{\R^d} (\Pi_v h -\overline{\Pi_v h})^2 Z_W^{\sigma/2} W^\sigma \dif U_\tau \dif \mu \right)^{\frac{2}{\sigma+2}} \leq 4^{\frac{2}{\sigma+2}} \,Z_W^{\frac{\sigma}{\sigma+2}} \, \|h_0\|^{\frac{4}{\sigma+2}}_{\mathrm{L}^\infty(\Theta)} \, \|W\|^{\frac{2\sigma }{\sigma+2}}_{\mathrm{L}^\sigma(\mu)} =: C_{2}.
\]
This allows finally to estimate~\eqref{eq:L2decom} as
\begin{equation}\label{ee2}
  \mathcal{H}_\tau(t) \leq \frac{1}{2\gamma C_{\mathrm{P},\nu}} \mathcal{D}_\tau(t) + C_2  \left( \frac{C^2_{0,\tau}}{\gamma C_{\mathrm{P},\nu}} + \gamma C^2_{1,\tau} \right)^{\frac{\sigma}{\sigma+2}} \mathcal{D}_\tau(t)^{\frac{\sigma}{\sigma+2}}.
\end{equation}
Let $\varphi: \mathbb R_+ \mapsto \mathbb R_+$ be the inverse of the increasing function~$y \mapsto A_1 y + A_2 y^{\frac{\sigma}{\sigma+2}}$ with~$A_1 = 1/(2\gamma C_{\mathrm{P},\nu})$ and~$A_2 = C_2 \left(\gamma^{-1} C^2_{0,\tau} C_{\mathrm{P},\nu}^{-1}+\gamma C^2_{1,\tau}\right)^{\frac{\sigma}{\sigma+2}}$. Then, \eqref{ee2} can be rewritten as
\begin{equation}\label{eq:BL1}
   \varphi\left( \mathcal{H}_\tau(t) \right) \leq  \mathcal{D}_\tau(t) .
\end{equation}
For $0 \leq y \le \varphi(\mathcal{H}_\tau(0))=:\varphi_0$, it holds~$\varphi^{-1}(y) \le (A_1 \varphi_0^{2/(\sigma+2)}+A_2)y^\frac{\sigma}{\sigma+2}$, and so~$\varphi(z) \geq (A_1 \varphi_0^{2/(\sigma+2)}+A_2)^{-\frac{\sigma+2}{\sigma}}z^\frac{\sigma+2}{\sigma}$ for $z = \varphi^{-1}(y) \le \mathcal{H}_\tau(0)$. Substituting this into \eqref{eq:BL1} (note that $\mathcal{H}_\tau(t)\le \mathcal{H}_\tau(0)$ by \eqref{ee}) gives
\[
\frac{\dif}{\dif t}\mathcal{H}_\tau(t) \leq -\left(A_1 \varphi_0^{2/(\sigma+2)}+A_2\right)^{-\frac{\sigma+2}{\sigma}} \mathcal{H}_\tau(t)^\frac{\sigma+2}{\sigma}.
\]
Solving the ODE, we end up 
with \begin{equation}\label{eq:BLcase1}
  \mathcal{H}_\tau(t) \le \left(\mathcal{H}_\tau(0)^{-2/\sigma} + \frac{2\left(A_1 \varphi_0^{2/(\sigma+2)}+A_2\right)^{-\frac{\sigma+2}{\sigma}}}{\sigma}t\right)^{-\frac{\sigma}{2}}.
\end{equation}

\medskip\noindent\textbf{Case~(ii).}
The main difference with case~(i) is the treatment of~$\int_t^{t+\tau} \|h-\Pi_v h\|^2_{\rmL^2( \Theta)} \dif U_\tau$ in~\eqref{eq:L2decom}: instead of directly using the Poincar\'e inequality~\eqref{eq:pinu}, we use H\"older's inequality in a way similar to~\eqref{eq:holder1} in conjunction with the weighted Poincar\'e inequality~\eqref{eq:wtpinu} in~$\nu$ to obtain 
\begin{align*}
  & \int_t^{t+\tau}\|h-\Pi_v h\|^2_{\rmL^2( \Theta)} \dif U_\tau \\
  & \le \left(\int_t^{t+\tau} \iint_{\R^d \times \R^d} \mathcal{G}^{-2} |h-\Pi_v h|^2 \dif U_\tau \dif \Theta\right)^{\frac{\delta}{2+\delta}}\left(\int_t^{t+\tau} \iint_{\R^d \times \R^d} \mathcal{G}^\delta |h-\Pi_v h|^2 \dif U_\tau \dif \Theta\right)^{\frac{2}{2+\delta}} \\
  & \le 4^\frac{2}{2+\delta}P_v^\frac{\delta}{2+\delta}\left(\int_{\R^d} \mathcal{G}^\delta \dif \nu\right)^\frac{2}{2+\delta}\|h_0\|_{\rmL^\infty(\Theta)}^{\frac{4}{2+\delta}}\left(\int_t^{t+\tau}\|\nabla_v h\|_{\rmL^2(\Theta)}^2 \dif U_\tau\right)^{\frac{\delta}{2+\delta}}  =: C_3\mathcal{D}_\tau(t)^{\frac{\delta}{2+\delta}},
\end{align*}
with
\[
C_3 = 2^{\frac{4-\delta}{2+\delta}} P_v^\frac{\delta}{2+\delta} \gamma^{-\frac{\delta}{2+\delta}}\|h_0\|_{\rmL^\infty(\Theta)}^{\frac{4}{2+\delta}}\|\mathcal{G}\|_{\rmL^\delta(\nu)}^{\frac{2\delta}{2+\delta}}.
\]
Substituting this into \eqref{eq:L2decom} and \eqref{eq:pivii} leads to the energy dissipation inequality
\begin{align*}
  \mathcal{H}_\tau(t) & \leq C_3 \mathcal{D}_\tau(t)^{\frac{\delta}{2+\delta}} + C_2\left[2C_{0,\tau}^2 C_3 \mathcal{D}_\tau(t)^{\frac{\delta}{2+\delta}} + \gamma C_{1,\tau}^2 \mathcal{D}_\tau(t) \right]^\frac{\sigma}{\sigma+2} .
\end{align*}
This also leads to an inequality of the same form as~\eqref{eq:BL1}, but with a different function~$\varphi$ given by the inverse of the increasing function~$y \mapsto C_3 y^{\frac{\delta}{2+\delta}} + C_2\left(2 C_{0,\tau}^2C_3 y^{\frac{\delta}{2+\delta}} + \gamma C_{1,\tau}^2 y \right)^\frac{\sigma}{\sigma+2}$. The final step is also similar to case~(i): if $y \le \varphi(\mathcal{H}_\tau(0))=:\varphi_0$, then $\varphi^{-1}(y)\le By^{\frac{\delta\sigma}{(\delta+2)(\sigma+2)}}$ with
\[
B = C_3 \varphi_0^{\frac{2\delta}{(\delta+2)(\sigma+2)}} + C_2 \left(2C_{0,\tau}^2C_3 + \gamma C_{1,\tau}^2\varphi_0^\frac{2}{\delta+2} \right)^\frac{\sigma}{\sigma+2},
\]
and consequently $\varphi(z) \geq B^{-\frac{(\delta+2)(\sigma+2)}{\delta\sigma}}z^{\frac{(\delta+2)(\sigma+2)}{\delta\sigma}}$ for $z = \varphi^{-1}(y) \le \mathcal{H}_\tau(0)$. Substituting this into the ODE, we finally obtain
\begin{equation}
  \label{eq:BLcase2}
  \mathcal{H}_\tau(t) \le \left(\mathcal{H}_\tau(0)^{-\frac{2(\sigma+\delta+2)}{\sigma\delta}} + \frac{2(\sigma+\delta+2)B^{-\frac{(\delta+2)(\sigma+2)}{\delta\sigma}}}{\sigma\delta}t\right)^{-\frac{\sigma\delta}{2(\sigma+\delta+2)}}.
\end{equation}
This concludes the proof as $\|h(t)\|_{\rmL^2(\Theta)}^2 \le \mathcal{H}_\tau(t-\tau)$ for~$t\geq \tau$, and $\calH(0) \le \|h_0\|_{\rmL^2(\Theta)}^2$ in view of~\eqref{ee}.
\end{proof}

\section{Proof of Theorem \ref{thm:weakPIconv}}
\label{sec:weakpi}

Our goal in this section is give another proof of long-time convergence rate for solutions of~\eqref{VOU}, which allows in particular to obtain stretched exponential rates. The weighted Poincar\'e--Lions inequality in Theorem~\ref{thm:wtpi} and the averaging result of Lemma~\ref{lm:avg} are still important building blocks for our proof. In this section, we will use the Dirichlet form~$\calD$ to control~$\calH$ in a slightly different way: since in general~$\calD$ only controls a part of~$\calH$, we consider an additional term involving the functional~$\Phi=\|\cdot\|^2_\osc$.

\subsection{Weak dissipation inequality for solutions of~\eqref{VOU}}
The aim of this section is to obtain a control of the time averaged~$\rmL^2$ norm of solutions of~\eqref{VOU} in terms of the dissipation functional~\eqref{eq:calDt} and an additional term~$\Phi(h_0)$ (recall that~$\Phi$ is introduced in Definition~\ref{def:WPI}). We distinguish the case when~$\nu$ satisfies a Poincar\'e inequality, and when~$\nu$ satisfies only a weak Poincar\'e inequality, as given by the following assumption. 

\medskip

\begin{assumption}
  \label{assm:wpi_v}
  There is a function $\beta_v:(0,\infty)\to [0,\infty)$ satisfying the standard conditions given in Definition~\ref{def:WPI}. In particular, for any $s>0$,  
  \begin{equation*}
    \forall f \in \mathrm{H}^1(\nu) \cap \rmL^\infty(\R^d), \qquad \|f-\nu(f)\|^2_{\rmL^2(\nu)} \le s \|\nabla_v f\|^2_{\rmL^2(\nu)} + \beta_v(s) \Phi(f).
  \end{equation*}
\end{assumption}

The following result provides a control of~$\mathcal{H}_\tau(t)$ in terms of~$\mathcal{D}_\tau(t)$ and~$\Phi(h_0)$ (recall that~$\mathcal{H}_\tau(t)$ and~$\mathcal{D}_\tau(t)$ are defined in \eqref{eq:calHt} and \eqref{eq:calDt} respectively).

\medskip

\begin{lemma}
  \label{lem:weak_Poincare_solutions_VOU}
  Fix~$\tau>0$, and an initial condition $h_0 \in L^\infty(\R^d)$ with $\int_{\R^d \times \R^d} h_0 \dif \Theta = 0$. Suppose that Assumptions~\ref{a:wp},~\ref{a:hessian},~\ref{ass2} hold true, and that~$\nu$ satisfies either the Poincar\'e inequality~\eqref{eq:pinu} or the weak Poincar\'e inequality from Assumption~\ref{assm:wpi_v}. Then, we have the weak dissipation inequality: 
  \begin{equation}
    \forall s>0, \qquad
    \mathcal{H}_\tau(t)\le s\mathcal{D}_\tau(t) + \beta_{\mathrm{kin}}(s) \Phi (h_0),
    \label{eq:abstract_wpi}
  \end{equation}
  with the following function~$\beta_{\mathrm{kin}}$ depending on the functional inequality satisfied by~$\nu$:
  \begin{itemize}
      \item when~$\nu$ satisfies the Poincar\'e inequality~\eqref{eq:pinu}, 
      \begin{equation}
          \beta_{\mathrm{kin}}(s) = \mu\left( s\le \widetilde C_{\tau,\nu,\gamma} W^2 + \frac{1}{2\gamma C_{\mathrm{P},\nu}}\right),
          \label{eq:beta_pi_mod}
      \end{equation}
      where~$\widetilde C_{\tau,\nu,\gamma} := Z_W(C^2_{0,\tau}/(\gamma C_{\mathrm{P},\nu})+ C_{1,\tau}^2 \gamma)$.
    \item when~$\nu$ satisfies the weak Poincar\'e inequality in Assumption~\ref{assm:wpi_v} with $\beta_v$ strictly positive, 
      \begin{equation}
        \label{eq:beta_related_overline_beta}
        \beta_{\mathrm{kin}}(s) = {\overline{C}\, \overline{\beta} (2\gamma \,s)}
      \end{equation}
      for a constant $\overline C>0$, explicitly constructed in the proof, and
      \begin{equation}
        \overline{\beta}(s):=\inf\left\{ s_1 \beta_v(s_2-c) + \beta_x(s_1) \, \middle | \, s_1>0, \, s_2>0, \, s_1 s_2 = s\right\},
        \label{eq:beta_form_mod}
      \end{equation}
      where
      \begin{equation}
        \label{eq:beta_x}
        \beta_x(s_1) := \mu\left( s_1 \le 2Z_W {C^2_{0,\tau}} W^2{+1}  \right)
      \end{equation}
      is assumed to be positive and $c=\gamma^2 C^2_{1,\tau}/C^2_{0,\tau}$.
  \end{itemize}
\end{lemma}

\begin{proof}
  The argument we present is inspired by~\cite[Theorem~4.2]{cattiaux2010functional}. Given a function $h(t,x,v) \in \rmL^2(U_\tau \otimes \Theta)$ and $(\partial_t + \calT)h \in \rmL^2(U_\tau \otimes \mu; \mathrm{H}^{-1}(\nu))$ with $\int_{[0,\tau] \times \R^d \times \R^d} h \dif U_\tau \dif \Theta =0$, we have, for any~$a \in (0,\infty)$,
  \begin{align*}
    & \|h\|^2_{\rmL^2(U_\tau \otimes \Theta)}
    \stackrel{\eqref{eq:L2decom}, \eqref{eq:minusavg}}{\le} \|h-\Pi_v h\|^2_{\rmL^2(U_\tau\otimes \Theta)} + \left\| \Pi_v h- \overline{\Pi_v h} \right\|^2_{\rmL^2(U_\tau\otimes \mu)}  \\ 
    &\le  \|h-\Pi_v h\|^2_{\rmL^2(U_\tau\otimes \Theta)} + \int_0^\tau \int_{\{W < a\}} \left(\Pi_v h- \overline{\Pi_v h}\right)^2 \frac{a^2}{W^2} \dif U_\tau \dif \mu \\
    &\quad + \int_0^\tau \int_{ \{W \geq a\} } \left(\Pi_v h- \overline{\Pi_v h}\right)^2 \dif U_\tau \dif \mu\\
    &\le \|h-\Pi_v h\|^2_{\rmL^2(U_\tau\otimes \Theta)} + a^2 Z_W\int_{[0,\tau]\times \R^d} \left(\Pi_v h- \overline{\Pi_v h}\right)^2  \dif U_\tau \dif \mu_W \\
    &\quad+ \mu(a\le W)\Phi(\Pi_v h) \\
    &\le (2a^2  Z_W C^2_{0,\tau}+1) \|h-\Pi_v h\|^2_{\rmL^2(U_\tau\otimes \Theta)} + 2 a^2 Z_W C^2_{1,\tau}\|(\partial_t + \calT)h\|^2_{\rmL^2(U_\tau \otimes \mu; \mathrm{H}^{-1}(\nu))}\\
        &\quad+ \mu(a\le W) \Phi(h), 
 \stepcounter{equation} \tag{\theequation}\label{eq:weakplh}
\end{align*}
where we have used Lemma~\ref{lm:avg}. Here we also denote, with a slight abuse of notation, 
\[
\Phi(h) = \sup_{t\in [0,\tau]}\|h(t,\cdot,\cdot)\|^2_\osc.
\]
By definition of the oscillation norm, $\Phi(\Pi_v h) \le \Phi(h)$. 

\smallskip

Now, for~$h$ a solution of~\eqref{VOU}, and in view of~\eqref{eq:nabla_v_star_bound},
\begin{align*}
  \int_t^{t+\tau} \|(\partial_t + \calT)h(s,\cdot,\cdot)\|^2_{\rmL^2(\mu; \mathrm{H}^{-1}(\nu))} \dif U_\tau & = \gamma^2 \int_t^{t+\tau}\|\nabla_v^\star \nabla_v h(s,\cdot,\cdot)\|^2_{\rmL^2(\mu; \mathrm{H}^{-1}(\nu))} \dif U_\tau\\ & \le \gamma^2 \int_t^{t+\tau}\| \nabla_v h(s,\cdot,\cdot)\|^2_{\rmL^2( \Theta)} \dif U_\tau. 
\end{align*} 
By the maximum principle, we also have that $\Phi(h)\le \Phi(h_0)$, and we therefore obtain from~\eqref{eq:weakplh} that
\begin{equation}
  \begin{aligned}
    & \int_t^{t+\tau}\|h(s,\cdot,\cdot)\|^2_{\rmL^2(\Theta)}\dif U_\tau \\
    & \quad \le (2a^2  Z_W C^2_{0,\tau}+1) \int_t^{t+\tau}\|(h-\Pi_v h)(s,\cdot,\cdot)\|^2_{\rmL^2(\Theta)}\dif U_\tau \\
    & \quad \qquad + 2 a^2 Z_W C^2_{1,\tau}\gamma^2 \int_t^{t+\tau}\| \nabla_v h(s,\cdot,\cdot)\|^2_{\rmL^2( \Theta)} \dif U_\tau+ \mu(a\le W) \Phi(h_0).
  \end{aligned}
  \label{eq:h_WPI1}
\end{equation}

We first consider the case when $\nu$ satisfies a weak Poincar\'e inequality (Assumption~\ref{assm:wpi_v}), for a strictly positive function $\beta_v>0$. 
In view of Assumption~\ref{assm:wpi_v}, we then have, for any $a>0$ and~$s_2>0$,
\begin{align*}
    & \int_t^{t+\tau}\|h(s,\cdot,\cdot)\|^2_{\rmL^2(\Theta)}\dif U_\tau \\
    & \quad \le \left(  (2a^2 Z_W C^2_{0,\tau}+1)s_2   + 2 a^2 Z_W C^2_{1,\tau}\gamma^2 \right) \int_t^{t+\tau}\| \nabla_v h(s,\cdot,\cdot)\|^2_{\rmL^2(\Theta)} \dif U_\tau\\
    & \quad \qquad +\left[ \mu(a\le W) +(2a^2  Z_W C^2_{0,\tau}+1)\beta_v(s_2)\right]\Phi(h_0) \\
    & \quad \le \frac{1}{2\gamma}\left( {(s_1+1 )s_2   +  s_1\gamma^2 C^2_{1,\tau}/C^2_{0,\tau}  }\right) \mathcal{D}_\tau(t) \\ & \quad \qquad + \left[ \mu\left(s_1 \le 2 Z_W {C^2_{0,\tau} }W^2\right) + {{(s_1+1)}}\beta_v(s_2)\right]\Phi(h_0)\\
    &\quad \le \frac{1}{2\gamma}{(s_1+1)(s_2+c)} \mathcal{D}_\tau(t) + \left[ \mu\left(s_1 \le 2 Z_W {C^2_{0,\tau} }W^2\right) + {{(s_1+1)}}\beta_v(s_2)\right]\Phi(h_0),
    \end{align*}
where we introduced $s_1:=2Z_W{ C^2_{0,\tau}} a^2$ and {$c:=\gamma^2 C^2_{1,\tau}/C^2_{0,\tau}$}, and used that the condition~$a \leq W$ is equivalent to~$s_1 \le 2Z_W C^2_{0,\tau} W^2$. Now as in  \textit{chaining} of weak Poincar\'e inequalities~\cite[Theorem~33]{Andrieu2022}, we set $s={(s_1+1)(s_2+c) }/(2\gamma)$. 
Minimizing over $s_1,s_2>0$ with $s$ fixed, we arrive at~\eqref{eq:abstract_wpi} with~$\beta_{\mathrm{kin}}$ replaced by
\begin{equation}
  \widetilde{\beta}(s)=\inf\left\{{ (s_1+1)} \beta_v(s_2) + \mu\left( s_1 \le 2Z_W  {C^2_{0,\tau}} W^2  \right) \, \middle | \, s_1>0, \, s_2>0, \, s = S(s_1,s_2)\right\},
  \label{eq:beta_form_mod_2}
\end{equation}
where $S(s_1,s_2)={(s_1+1)(s_2+c)}/(2\gamma)$. It remains to prove that $\beta_{\mathrm{kin}}(s) \geq \widetilde{\beta}(s)$ for~$\beta_{\mathrm{kin}}(s)$ defined in~\eqref{eq:beta_related_overline_beta}. 

\smallskip

{By the change of variables $s_1+1 \mapsto s_1, \, s_2+c \mapsto s_2$, we can rewrite \eqref{eq:beta_form_mod_2} as
\begin{equation*}
    \begin{split}
        \widetilde \beta(s) = \inf\left\{ s_1 \beta_v(s_2-c) + \beta_x(s_1) \ \middle | \ s_1 > 1, s_2 > c, s_1 s_2 = 2\gamma \, s  \right\},
    \end{split}
\end{equation*}
where we recall that~$\beta_x$ is defined in \eqref{eq:beta_x}, and we define $\beta_v(s):=\beta_v(0)$ for $s<0$.
Since we have a weak Poincar\'e inequality with functions~$\beta_v$ and~$\beta_x$ with~$\beta_x(s_1) \to 0$ as~$s_1 \to +\infty$ and also for~$\beta_v$, it holds $\widetilde{\beta}(s)\to 0$ as $s\to\infty$ (see the proof of \cite[Theorem~33]{Andrieu2022} for an explicit argument). In the limit~$s\to \infty$, the values of $s_1,s_2$ which saturate the infimum in \eqref{eq:beta_form_mod_2}-\eqref{eq:beta_form_mod} must also diverge to $+\infty$ (to see this, note first that~$s_1$ has to diverge otherwise the second term in the argument of the infimum would be positive; and next since $\beta_v$ is positive, $s_2$ needs to diverge otherwise the first term in the argument of the infimum would diverge). Thus there is a threshold $M>0$ such that, provided $s\ge M$, we need only consider values $s_1\ge 1$ and $s_2\ge c $ (using that~$\beta_x,\beta_v$ have positive values). In other words, provided that $s>M$, we further have that
\[
\widetilde \beta(s) = \inf\left\{ s_1 \beta_v(s_2-c) + \beta_x(s_1) \ \middle | \ s_1 > 0, s_2 > 0, s_1 s_2 = 2\gamma \, s \right\} = \overline \beta(2\gamma\, s).
\]
We have thus established that~$\widetilde{\beta}(s) \leq \overline{\beta}(2 \gamma s)$ for~$s > M$. Upon defining 
\[
\overline{C} := \max\left\{ 1, \sup_{0 \leq s \leq M} \frac{\widetilde{\beta}(s)}{\overline {\beta}(2\gamma s)} \right\} < +\infty,
\]
we can conclude that~$\widetilde{\beta}(s) \leq \beta_{\mathrm{kin}}(s)$ for any~$s \geq 0$, as desired.
}

\smallskip

When $\nu$ satisfies the Poincar\'e inequality \eqref{eq:pinu}, which includes the case when $\nu$ is a standard Gaussian distribution, the result follows by going back to~\eqref{eq:h_WPI1} and performing similar estimates with~$s_2=1$, namely
  \[
  \begin{aligned}
    &   \int_t^{t+\tau}\|h(s,\cdot,\cdot)\|^2_{\rmL^2(\Theta)}\dif U_\tau \\
    & \le \left( \frac{2a^2  Z_W C^2_{0,\tau}+1}{C_{\mathrm{P},\nu}} + 2 a^2 Z_W C^2_{1,\tau}\gamma^2 \right) \int_t^{t+\tau}\| \nabla_v h(s,\cdot,\cdot)\|^2_{\rmL^2( \Theta)} \dif U_\tau+ \mu(a\le W) \Phi(h_0) \\
    & \le \left( Z_W a^2 \frac{C^2_{0,\tau}/ {C_{\mathrm{P},\nu}} + C^2_{1,\tau}\gamma^2}{\gamma} + \frac{1}{2\gamma C_{\mathrm{P},\nu}} \right) \mathcal{D}_\tau(t) + \mu(a\le W) \Phi(h_0),
  \end{aligned}
  \]
  which indeed leads to~\eqref{eq:beta_pi_mod}.
\end{proof}

  \subsection{General case}
  \label{sec:weak_Poincare_general_case}

We now state a general and abstract theorem, which will be used to deduce convergence from \eqref{eq:abstract_wpi}. We follow the approach of~\cite{Andrieu2022}, and present a series of analogous definitions and lemmas.

\medskip

\begin{definition}
  Consider a function~$\beta$ appearing in a weak Poincar\'e inequality~\eqref{eq:WPI} or $\beta=\beta_{\mathrm{kin}}$ appearing in a weak dissipation inequality \eqref{eq:abstract_wpi}, and let $K:[0,\infty)\to [0,\infty)$ be given by~$K(0)=0$ and, for $u>0$,
  \begin{equation*}
    K(u) := u \,\beta(u^{-1}).
  \end{equation*}
  The convex conjugate (or Legendre transform) of this function, $K^*:[0,\infty)\to [0,\infty]$, is
    \begin{equation*}
      K^*(w) := \sup_{u\geq 0} \{uw-K(u)\}.
    \end{equation*}
    Finally, define $F_a:(0,a)\to [0,\infty)$ as
      \begin{equation*}
        F_a(z) := \int_z^a \frac{\dif w}{K^*(w)},
      \end{equation*}
      and introduce the constant
      \begin{equation}\label{eq:cala}
        \mathfrak{a}:= \sup \left\{ \frac{\|f\|^2_{\rmL^2(U_\tau\otimes\Theta)}}{\Phi(f)} \ : \ f\in \mathrm{L}^2(U_\tau\otimes\Theta), \int_{[0,\tau]\times\R^d\times\R^d} f \dif U_\tau\dif\Theta=0 \right\}\le \frac14,
      \end{equation}
      the latter bound following by Popoviciu's inequality (see, for instance, \cite{Bhatia2000}).
      \label{def:WPI_bits}
\end{definition}

\medskip

The following lemma gathers gathers some properties of the function introduced in the previous definition (see Lemma~1 in~\cite[Supplementary Material]{Andrieu2022} and~\cite[Lemma~7]{Andrieu2022}).

\medskip

\begin{lemma}
  It holds~$K^*(0)=0$ and~$K^*(w)>0$ for~$w>0$. The function~$K^*$ is convex, continuous and strictly increasing on its domain. For $w\in[0,\mathfrak a]$, it holds~$K^*(w)\le w$; moreover, the function~$w\mapsto w^{-1}K^*(w)$ is increasing.

  The function $F_\mathfrak{a}$ is well-defined, convex, continuous and decreasing, with $\lim_{z\downarrow 0}F_\mathfrak{a}(z)=\infty$, and possesses a well-defined decreasing inverse function $F^{-1}_\mathfrak{a}:(0,\infty)\to (0,\mathfrak a)$ with $F^{-1}_\mathfrak{a}(z)\to 0$ as $z\to\infty$.
  \label{lemma:K_prop}
\end{lemma}

\medskip

\begin{lemma}
  \label{lemma:opt_WPI}
  Under the same hypotheses as Lemma~\ref{lem:weak_Poincare_solutions_VOU}, we have, for any non-zero initial condition~$h_0 \in \rmL^\infty$,
  \begin{equation}
    \frac{\calD(t)}{\Phi(h_0)}\geq K^*\left(\frac{\mathcal H_\tau(t)}{\Phi(h_0)}\right),
    \label{eq:opt_WPI}
  \end{equation}
  where~$K^*$ is defined as in Definition~\ref{def:WPI_bits} based on the function~$\beta_\mathrm{kin}$ from Lemma~\ref{lem:weak_Poincare_solutions_VOU}.
\end{lemma}

\begin{proof}
  This proof is inspired by the discrete-time argument in the proof of~\cite[Theorem~8]{Andrieu2022}. We start by reformulating~\eqref{eq:abstract_wpi} as 
  \begin{equation*}
    \frac{\calD(t)}{\Phi(h_0)} \geq \frac{\calH(t)}{s\,\Phi(h_0)}-\frac{\beta_\mathrm{kin}(s)}{s}.
  \end{equation*}
  Rewriting the inequality in terms of~$K$, and reparameterizing $u:=s^{-1}$, we obtain
  \begin{equation*}
    \frac{\calD(t)}{\Phi(h_0)} \geq u\, \frac{\calH(t)}{\Phi(h_0)} -K(u).
  \end{equation*}
  Since $u$ is arbitrary, we can take a supremum over~$u\geq 0$ on the right hand side of the previous inequality to conclude that~\eqref{eq:opt_WPI} holds.
\end{proof}

Lemma~\ref{lemma:opt_WPI} allows us to use a Bihari--LaSalle argument to deduce (potentially subgeometric) convergence of the semigroup.

\medskip

\begin{theorem}
  Under the assumptions of Lemma~\ref{lem:weak_Poincare_solutions_VOU}, it holds
  \begin{equation*}
    \mathcal{H}_\tau(t) \le \Phi(h_0)  F_\mathfrak{a}^{-1}(t).
  \end{equation*}
  \label{thm:conv}
\end{theorem}

Theorem~\ref{thm:weakPIconv} is then a direct corollary of Theorem~\ref{thm:conv} since, for any $t\geq \tau$, it holds~$\|h(t)\|_{\rmL^2(\Theta)}^2 \le \mathcal{H}_\tau (t-\tau)$ in view of~\eqref{ee}: the rate in Theorem~\ref{thm:weakPIconv} is given by
\begin{equation}
    \mathsf F(t)=    F_{\mathfrak a}^{-1}(t-\tau)\mathbf{1}_{\{t > \tau\}} + \tau \mathbf{1}_{\{t \leq \tau\}}.
    \label{eq:sfF_def}
\end{equation}

\begin{proof}
  By~\eqref{eefull} and~\eqref{eq:opt_WPI},
  \begin{equation*}
    \frac{\dif }{\dif t}\mathcal{H}_\tau(t)
    = -\mathcal D_\tau(t)\le - \Phi(h_0)  K^*\left(\frac{\calH(t)}{\Phi(h_0)}\right),
  \end{equation*}
  and so
  \[
  \frac{\dif }{\dif t}\left( \frac{\calH(t)}{\Phi(h_0)} \right) \leq -K^*\left(\frac{\calH(t)}{\Phi(h_0)}\right).
  \]
  Thus, from a Bihari--LaSalle argument, we obtain that
  \[
  F_\mathfrak{a}\left( \frac{\calH(t)}{\Phi(h_0)} \right) - F_\mathfrak{a}\left( \frac{\calH(0)}{\Phi(h_0)} \right) \geq t, 
  \]
  so that 
  \begin{equation*}
    \frac{\calH(t)}{\Phi(h_0)}\le F_\mathfrak{a}^{-1}\left(  F_\mathfrak{a}\left( \frac{\|h_0\|^2_{\rmL^2(\Theta)}}{\Phi(h_0)}\right) +t\right) \le F_\mathfrak{a}^{-1}\left( t\right),
  \end{equation*}
  where we used that~$F_\mathfrak{a}^{-1}$ is decreasing.
\end{proof}

\subsection{Specific cases}
\label{subsec:specific_cases}

We specify in this section the general convergence bound of Theorem~\ref{thm:conv} to concrete situations of interest. This allows us to come up with Table~\ref{tab:weakpirates}. For each situation, we make precise what the functions~$K^*$ and~$F_\mathfrak{a}$ are, relying on various results of~\cite{Andrieu2022}. We start by considering in Section~\ref{sec:strong_vel_confinement} the case when the velocity distribution~$\nu$ satisfies a Poincar\'e inequality, in which case the convergence rate is dictated by the asymptotic behavior of the potential~$\phi$, then turn in Section~\ref{sec:weak_velocity_confinement} to situations when the velocity distribution also has heavy tails, in which case the convergence rate is obtained by a chaining of weak Poincar\'e inequalities.

\subsubsection{Strong velocity confinement}
\label{sec:strong_vel_confinement}

We assume in this section that~$\nu$ satisfies the Poincar\'e inequality~\eqref{eq:pinu}. This includes the case when $\nu$ is Gaussian. We distinguish two cases: $\phi(x) = \langle x\rangle^\alpha$ for $\alpha \in (0,1)$ (sublinear case) or $\phi(x) = (d+p)\log\langle x\rangle$ for some $p>0$ (logarithmic case).

\paragraph{Sublinear potential.} The key abstract result to obtain the convergence rate is~\cite[Lemma~15]{Andrieu2022}.

\begin{lemma}
  Suppose that~$\beta(s)=\eta_0\exp(-\eta_1 s^{\eta_2})$ for some $\eta_0,\eta_1,\eta_2>0$. Then, there exists~$c>0$ such that, for $w>0$ sufficiently small,
  \begin{equation}
    \label{eq:K_star_sublinear}
    K^*(w)\geq c w \left(\log \frac{1}{w}\right)^{-1/\eta_2}.
  \end{equation} 
  The resulting convergence rate can be bounded as
  \begin{equation*}
    F_\mathfrak{a}^{-1} (t) \le C \exp\left( -C' t^{\frac{\eta_2}{1+\eta_2}} \right),
  \end{equation*}
  for some constants~$C,C'>0$.
  \label{lem:sub_exp_beta}
\end{lemma}

\medskip

\begin{example}\label{ex1}
  Consider the case when $\phi(x) = \langle x\rangle^\alpha$ for $\alpha \in (0,1)$. Since $W = \bangle{x}^{1-\alpha}$, the condition~$u \le W(x)$ is equivalent to~$\bangle{x} \geq u^{\frac{1}{1-\alpha}}$, and hence, using a change of variables,
  \begin{equation}
    \mu(u\le W) = \int_{\bangle{x} \geq u^{\frac{1}{1-\alpha}} } \exp(-\bangle{x}^\alpha) \dif x \lesssim \int_{u^\frac{1}{1-\alpha}}^\infty r^{d-1} \exp(-r^\alpha) \dif r \lesssim \exp(-cu^{\frac{\alpha}{1-\alpha}}),
    \label{eq:x_subexp}
  \end{equation}
  where the inequalities omit constants which do not depend on $u$. Thus, referring to~\eqref{eq:beta_pi_mod}, we see that the weak dissipation inequality~\eqref{eq:abstract_wpi}  holds in this case with
  \begin{equation*}
    \beta_\mathrm{kin}(s) = \eta_0 \exp\left (-\eta_1 s^{\frac{\alpha}{2(1-\alpha)}}\right).
  \end{equation*}
  Lemma~\ref{lem:sub_exp_beta} then leads to the following convergence bound, for some $C,C'>0$:
  \begin{equation*}
    \mathcal{H}_\tau(t) \le C \Phi(h_0) \exp\left( -C't^{\frac{\alpha}{2-\alpha}} \right).
  \end{equation*}
  Note that this offers an improvement upon the corresponding convergence rate obtained in~\cite[Example~1.4(c)]{rockner2001weak}, where the power in the stretched exponential is~$\alpha/(4-3\alpha)$, which is smaller than~$\alpha/(2-\alpha)$ since~$2-\alpha < 4-3\alpha$.
\end{example} 

\paragraph{Logarithmic potential.} The key abstract result to obtain the convergence rate is~\cite[Lemma~14]{Andrieu2022}.

\begin{lemma}
  Suppose that~$\beta(s)=\eta_0 s^{-\eta_1}$ for some $\eta_0,\eta_1>0$. Then,
  \[
  K^*(w)= \frac{\eta_0 \eta_1}{(\eta_0(1+\eta_1))^{1+\eta_1^{-1}}}w^{1+\eta_1^{-1}},
  \]
  and the resulting convergence rate can be bounded as
  \begin{equation*}
    F_\mathfrak{a}^{-1} (t) \le \eta_0 (1+\eta_1)^{1+\eta_1}\, t^{-\eta_1}.
  \end{equation*}
  \label{lem:poly_beta}
\end{lemma}

\begin{example}
  \label{ex:poly_gaus}
  Consider the case when $\phi(x) = (d+p)\log\langle x\rangle$ for some $p>0$. Since~$W(x)=\langle x\rangle$, 
  \begin{equation*}
    \mu(u\le W) \lesssim \int_{\langle x\rangle \ge u} \langle x \rangle^{-(d+p)}\lesssim u^{-p}.
  \end{equation*}
  Referring again to \eqref{eq:beta_pi_mod}, we see that the weak dissipation inequality~\eqref{eq:abstract_wpi} holds with 
  \begin{equation*}
    \beta_\mathrm{kin}(s) = \eta_0 s^{-p/2}.
  \end{equation*}
  Applying Lemma~\ref{lem:poly_beta}, we conclude to the convergence bound
  \begin{equation*}
    \mathcal{H}_\tau(t) \le C \Phi(h_0)\, t^{-p/2}.
  \end{equation*}
\end{example}

Let us conclude this section by commenting that the scalings of~$\beta_\mathrm{kin}(s)$ in both Examples~\ref{ex1} and~\ref{ex:poly_gaus} match these of weak Poincar\'e inequalities for~$\mu$ obtained in respectively~\cite[Propositions~4.11 and~4.9]{cattiaux2010functional}, which improves upon \cite{rockner2001weak}. This can be seen by applying the results of~\cite{cattiaux2010functional} with~$V(x) = \langle x\rangle$ and weak Poincar\'e inequalities in the form
\[
\int_{\R^d} (f - \mu(f))^2 \dif \mu \leq \beta^{-1}(s) \int_{\R^d}|\nabla_x f|^2 \dif \mu + s \Phi(f).
\]

\subsubsection{Weak velocity confinement}
\label{sec:weak_velocity_confinement}
We consider in this section cases when the velocity distribution is weakly confining and does not satisfy a Poincar\'e inequality. These examples are largely inspired by the previous work~\cite{grothaus2019weak}. A key technical tool in this context is chaining (\emph{i.e.} composition) of weak Poincar\'e inequalities; see the proof of~\cite[Theorem~33]{Andrieu2022}.

\medskip

\begin{proposition}
  Given a weak Poincar\'e inequality with $\beta$ of the form 
  \begin{equation*}
    \beta(s)=\inf\left\{s_1 \beta_2(s_2)+\beta_1 (s_1) \, \middle| \, s_1>0, \, s_2>0, \, s_1s_2 = s \right\},
    \end{equation*}
  where each function~$\beta_1,\beta_2$ has a corresponding function~$K^*_1, K^*_2$ respectively (as in Definition~\ref{def:WPI_bits}), the resulting function~$K^*$ for $\beta$ is
  \begin{equation*}
    K^* = K_2^* \circ K_1^*.
  \end{equation*}
  \label{prop:chaining}
\end{proposition}
We also note from~\eqref{eq:beta_form_mod} that we are also interested in handling $\beta$ functions of the form $\widetilde \beta(s)=\beta(s-c)$ for some fixed $c>0$. We prove in the lemma below this fixed shift does not have any asymptotic impact. 

\medskip

\begin{lemma}\label{lem:shiftbetas}
    For a fixed function~$\beta$ and a constant~$c>0$, define $\widetilde \beta(s)=\beta(s-c)$ for $s\ge c$, and $\widetilde \beta(s)=\beta(0)$ for $0\le s< c$. Then the corresponding function $\widetilde K^*$ for $\widetilde \beta$ satisfies the following lower bound for an appropriate constant $\widetilde c$:
    \begin{equation*}
        \forall u \in (0,1/8), \qquad \widetilde K^*(u) \ge \widetilde c K^*(u).
    \end{equation*}
\end{lemma}

\begin{remark}
    In particular, since we have the following inequality for $x<1/8$:
    $$\int_x^\mathfrak{a}\frac{\dif v}{\widetilde K^*(v)}=a+\int_x^{1/8}\frac{\dif v}{\widetilde K^*(v)}\le  a+ \tilde c^{-1} \int_x^{1/8}\frac{\dif v}{ K^*(v)}, 
    \qquad
    a = \int_{1/8}^\mathfrak{a}\frac{1}{\widetilde K^*},$$
    the resulting rate of convergence arising from $\tilde K^*$ can be bounded by the same rate as that of $K^*$ up to constant factors.
\end{remark}

\begin{proof}
    Simple algebra shows that we can write, with the change of variable~$\widetilde{u} = u/(1+cu)$,
    \begin{align*} 
        \widetilde K^*(v) & = \sup_{\widetilde{u} > 0} \left\{ \widetilde{u} \left[ v - \widetilde{\beta}\left(\frac{1}{\widetilde{u}}\right) \right] \right\} = \sup_{\widetilde{u} > 0} \left\{ \widetilde{u} \left[ v - \beta\left(\frac{1}{\widetilde{u}}-c\right) \right] \right\}\\
        & = \sup_{u>0} \left\{ \frac{u}{1+cu}\left[ v-\beta\left(\frac1u\right)\right] \right\}.
    \end{align*}
    Now, we can find a threshold $w:= \sup\left\{u>0: \beta(1/u)\le 1/8\right\} $, since $\beta(s)\to 0$ as $s\to\infty$ monotonically. Thus, for $v<1/8$, it holds that $v-\beta(1/u) \leq 0$ for $u \geq w$, and so the maximization in~$\widetilde{K}^*$ can be restricted to~$u < w$:
        \begin{align*}
            \widetilde K^*(v)= \sup_{0<u<w} \left\{ \frac{u}{1+cu}\left[ v-\beta\left(\frac1u\right)\right] \right\} \ge \sup_{0<u<w} \left\{ \frac{u}{1+cw}\left[ v-\beta\left(\frac1u\right)\right] \right\}
            = \widetilde c K^*(v),
        \end{align*}
    with $\widetilde c := (1+cw)^{-1}$.
\end{proof}

Once the function~$K^*$ has been identified, a rate of convergence can be then be derived following Theorem~\ref{thm:conv}. We apply the above result to various situations of interest, with sublinear or logarithmic potential energies~$\phi$ and kinetic energies~$\psi$, by identifying in each case the function~$\beta_v$ associated with the weak Poincar\'e inequality for the velocity distribution, as well as the function~$\beta_x$ defined in~\eqref{eq:beta_x}, in order to deduce a convergence rate from~\eqref{eq:beta_form_mod} and Proposition~\ref{prop:chaining}. Note at this stage that the functions associated with~$\beta_\mathrm{kin}$ in~\eqref{eq:beta_related_overline_beta} are related to the functions associated with~$\overline{\beta}$ in~\eqref{eq:beta_form_mod} as
\[
K(w) = 2\gamma \overline{C} \, \overline{K}\left(\frac{w}{2\gamma}\right),
\qquad
K^*(w) = 2\gamma \overline{C} \, \overline{K}^*\left(\frac{w}{\overline{C}}\right),
\]
so that
\[
F_\mathfrak{a}(z) = \frac{1}{2\gamma} \overline{F}_{\mathfrak{a}/\overline{C}}\left(\frac{z}{\overline{C}}\right).
\]
Therefore, the convergence rate~$F_\mathfrak{a}^{-1}(t)$ associated with~$\beta_\mathrm{kin}$ is related to the convergence rate~$\overline{F}_\mathfrak{a}^{-1}(t)$ associated with~$\overline{\beta}$ as
\[
F_\mathfrak{a}^{-1}(t) = \overline{C} \, \overline{F}_{\mathfrak{a}/\overline{C}}^{-1}(2\gamma t).
\]
For the examples below, it is therefore sufficient to understand the scaling of~$\overline{\beta}$ as~$F_\mathfrak{a}^{-1}$ and~$\overline{F}_\mathfrak{a}^{-1}$ have similar asymptotic behaviors.

\medskip

\begin{example}
  \label{exam:vel_poly_x_poly}
  Suppose that the kinetic energy is~$\psi(v)=(d+q)\log\langle v \rangle$ and that the potential energy is~$\phi(x) = (d+p)\log\langle x\rangle$ for some $q,p>0$. In view of to~\cite[Proposition~4.9]{cattiaux2010functional}, a weak Poincar\'e inequality holds for~$\nu$ with
  \begin{equation}
    \beta_v(s) = \eta_{0,v} s^{-q/2},
    \label{eq:wpi_v_poly}
  \end{equation}
  while~$\beta_x(s) = \eta_{0,x} s^{-p/2}$ from the computations in Example~\ref{ex:poly_gaus}. Then, in view of~\eqref{eq:beta_form_mod} and~\cite[Example~34]{Andrieu2022}, 
  \begin{equation*}
    \overline{\beta}(s) \le \eta_0 s^{-\frac{pq}{2(2+p+q)}}. 
  \end{equation*}
  Lemma~\ref{lem:poly_beta} therefore allows to conclude to the convergence rate
  \begin{equation*}
    \mathcal{H}_\tau(t) \le C \Phi(h_0) t^{-\frac{pq}{2(2+p+q)}}.
  \end{equation*}
  The rate is symmetric in $p$ and $q$, and is consistent with the one obtained in Example~\ref{ex:poly_gaus} in the limit $q\to\infty$, and with Example~\ref{ex:v_poly_x_std} in Appendix \ref{app:stdspace} in the limit $p\to\infty$.
\end{example}

\medskip

\begin{example}
  \label{example:v_poly_x_subexp}
  Suppose that the kinetic energy is~$\psi(v)=(d+q)\log\langle v \rangle$ for $q>0$ and that the potential energy is~$\phi(x)=\langle x\rangle^\alpha$ for~$\alpha \in (0,1)$. Then~$\beta_v$ is still given by~\eqref{eq:wpi_v_poly} as in Example~\ref{exam:vel_poly_x_poly}, while~$\beta_x$ is obtained from~\eqref{eq:x_subexp} in Example~\ref{ex1}, namely
  \[
  \beta_x(s) = \eta_0 \exp\left (-\eta_1 s^{\frac{\alpha}{2(1-\alpha)}}\right).
  \]
  Applying Proposition~\ref{prop:chaining}, and using the expressions of~$K_x^*,K_v^*$ given by Lemmas~\ref{lem:sub_exp_beta} and~\ref{lem:poly_beta} respectively, we find that
  \[
  \begin{aligned}
    \overline{K}^*(w) & = (K_v^* \circ K_x^*)(w) = C K_x^*(w)^{1+2/q} = C' w^{1+2/q} \left( \log \frac1w \right)^{-2(\alpha^{-1}-1)(1+2q^{-1})} \\
    & \geq \mathscr{C}_\varepsilon w^{1+2/q-\varepsilon}
  \end{aligned}
  \]
  for an arbitrarily small~$\varepsilon \in (0,2/q)$, and therefore
  \begin{equation*}
    \begin{split}
      \overline{F}_\mathfrak{a}(z) = \int_z^\mathfrak{a} \frac{\dif w}{\overline{K}^*(w)} \leq \mathscr{C}_\varepsilon^{-1} \int_z^\mathfrak{a} w^{-1-2/q+\varepsilon} \, dw \leq \frac{1}{(2/q-\varepsilon)\mathscr{C}_\varepsilon} z^{-2/q + \varepsilon}.
    \end{split}
  \end{equation*}
  This allows to conclude that
  \[
  \mathcal{H}_\tau(t) \le R_{\varepsilon'} \Phi(h_0)\, t^{-q/2+\varepsilon'}
  \]
  for~$\varepsilon'>0$ arbitrarily small and some constant~$R_{\varepsilon'}>0$ depending on $\varepsilon'$.
\end{example}

\medskip

\begin{example}
  \label{example:v_subexp_x_subexp}
  Suppose that the kinetic energy is~$\psi(v) = \langle v\rangle^\delta$ for $\delta \in (0,1)$ and that the potential energy is~$\phi(x) = (d+p)\log\langle x\rangle$ for some $p>0$. Then, by an argument similar to the one of Example~\ref{example:v_poly_x_subexp} (where the expressions of~$K_v^*$ and~$K_x^*$ are exchanged), we find that, for any arbitrarily small~$\varepsilon>0$, there exists a constant~$R_\varepsilon>0$ such that 
  \[
  \mathcal{H}_\tau(t) \le R_\varepsilon \Phi(h_0)\, t^{-p/2+\varepsilon}.
  \]
  \label{example:v_subexp_x_poly}
\end{example}

\medskip

\begin{example}
  \label{ex6}
  Suppose that the kinetic energy is~$\psi(v) = \langle v\rangle^\delta$ and that the potential energy is~$\phi(x)=\langle x\rangle^\alpha$ for $\delta,\alpha\in(0,1)$. Since the associated functions~$\beta_x,\beta_v$ are obtained from~\eqref{eq:x_subexp} and~\eqref{eq:wpi_v_poly} respectively, the associated functions~$K_x^*,K_v^*$ are given by~\eqref{eq:K_star_sublinear} with~$\eta_2$ set to~$\alpha/[2(1-\alpha)]$ and~$\delta/[2(1-\delta)]$ respectively. Therefore, in view of Proposition~\ref{prop:chaining}, we obtain, for~$w \in (0,\mathfrak{a}]$,
  \[
  \begin{aligned}
    \overline{K}^*(w) = (K_v^* \circ K_x^*)(w)
    & \geq C_v K_x^*(w) \left(\log \frac{1}{K_x^*(w)} \right)^{-2(1-\delta)/\delta} \\
    & \geq C_v C_x w \left(\log \frac{1}{w} \right)^{-2(1-\alpha)/\alpha} \left(\log \frac{1}{K_x^*(w)} \right)^{-2(1-\delta)/\delta} \\
    & \geq \mathscr{C} w \left(\log \frac{1}{w} \right)^{-1/\eta},
  \end{aligned}
  \]
  where~$\eta>0$ satisfies
  \[
  \frac{1}{\eta} = 2\left( \frac{1-\alpha}{\alpha} + \frac{1-\delta}{\delta} \right),
  \]
  \emph{i.e.}
  \[
  \eta = \frac{\alpha \delta}{2(\alpha+\delta-2\alpha\delta)}.
  \]
  In view of Lemma~\ref{lem:sub_exp_beta}, this leads to the convergence rate 
  \begin{equation*}
    \mathcal{H}_\tau(t) \le C \Phi(h_0) \exp\left( -C't^{\frac{\alpha \delta }{2\alpha+2\delta -3\alpha \delta}} \right),
  \end{equation*}
  for some constants $C,C'>0$. As in Example~\ref{exam:vel_poly_x_poly}, the convergence rate is symmetric with respect to~$\alpha$ and~$\delta$. Moreover, the result is consistent with Example~\ref{ex1} when taking the limit~$\delta\to 1$, and with Example~\ref{ex:v_subex_x_std} in Appendix~\ref{app:stdspace} when taking the limit~$\alpha \to 1$.
\end{example}

\bmhead{Acknowledgements}
The idea of this work first originated when GS, AQW and LW met at the 14th International Conference on Monte Carlo Methods and Applications at Sorbonne University in Paris in June 2023. 

\noindent We would like to thank Jean Dolbeault for helpful suggestions. 

The work of GB has received funding from the European Union’s Horizon 2020 research and innovation programme under the Marie Skłodowska-Curie grant agreement No 101034413.

The work of GS was funded by the European Research Council (ERC) under the European Union's Horizon 2020 research and innovation programme (project EMC2, grant agreement No 810367), and by Agence Nationale de la Recherche, under grants ANR-19-CE40-0010-01 (QuAMProcs) and ANR-21-CE40-0006 (SINEQ).

GS and AQW would like to thank the Isaac Newton Institute for Mathematical Sciences, Cambridge, for support and hospitality during the programme Stochastic Systems for Anomalous Diffusions where some work on this paper was undertaken. This work was supported by EPSRC grant no EP/R014604/1.

LW is supported by the National Science Foundation via grant DMS-2407166. Part of this work was completed when LW was visiting Institute of Science and Technology, Klosterneuburg, Austria and Max Planck Institute for Mathematics in the Sciences in Leipzig, Germany. LW would like to thank ISTA and MPI MiS for their hospitality.

\begin{appendix}
  
\section{Convergence rate with strong spatial confinement using weak Poincar\'e inequality}
\label{app:stdspace}

We present in this section how to derive convergence rates of~\eqref{eq:langevin} for $\rmL^\infty$ initial data and various velocity equilibria~$\psi$ when $\mu$ satisfies the standard Poincar\'e inequality~\eqref{eq:pimu}. In this case, we know from~\cite[Theorem~2]{Cao2023} or~\cite[Lemma~3]{Brigati2023} that there exists a constant $C_{\rm PL}$ such that for any $h(t,x,v) \in \rmL^2(U_\tau \otimes \Theta)$ with $(\partial_t + \calT)h \in \rmL^2(U_\tau \otimes \mu; \mathrm{H}^{-1}(\nu))$ and $\int_{[0,\tau] \times \R^d \times \R^d} h \dif U_\tau \dif \Theta =0$,
\begin{equation}\label{eq:stdpl}
  \|h\|_{\rmL^2(U_\tau \otimes \Theta)}^2 \le C_{\rm PL}\left(\|h-\Pi_v h\|^2_{\rmL^2(U_\tau \otimes \Theta)}  + \|(\partial_t + \calT) h\|_{\rmL^2(U_\tau \otimes \mu; \mathrm{H}^{-1}(\nu))}^2\right).
\end{equation}
Now, consider an initial condition~$h_0 \in \rmL^\infty$, and let $h$ be the solution of~\eqref{VOU}. The second term on the right-hand side of \eqref{eq:stdpl} is directly controlled by $\gamma \calD(t)/2$ by manipulations similar to the ones used for case~(i) in the proof of Theorem~\ref{thm:weightedpiconv}; while the other term can be treated using Assumption~\ref{assm:wpi_v} (namely that~$\nu$ satisfies a weak Poincar\'e inequality). This leads to
\[
\calH(t) \le C_{\rm PL} \left(\left(s+\frac{\gamma}{2}\right)\calD(t) + \beta_v(s) \Phi(h_0) \right).
\]
After a change of variables $s+\frac{\gamma}{2} \mapsto s$, this inequality is exactly the weak dissipation inequality~\eqref{eq:abstract_wpi}, with $\beta_v(s-\frac{\gamma}{2})$ playing the role of $\beta_\mathrm{kin}$, which by Lemma \ref{lem:shiftbetas} has the same asymptotic behavior as for $\beta_v$. Hence, we can  apply Theorem \ref{thm:conv} to obtain the following convergence bounds, similarly to previous examples. 

\medskip

\begin{example}\label{ex:v_subex_x_std}
  Suppose that the kinetic energy is $\psi(v) = \bangle{v}^\delta$ for $\delta \in (0,1)$. Then, $\beta_v$ is obtained from~\eqref{eq:x_subexp} in Example~\ref{ex1}, and we follow a similar calculation to that example and derive
  \[
  \mathcal{H}_\tau(t) \le C \Phi(h_0) \exp\left( -C't^{\frac{\delta}{2-\delta}} \right).
  \]
\end{example}

\begin{example}\label{ex:v_poly_x_std}
  Suppose that the kinetic energy is $\psi(v) = (d+q)\log \bangle{v}$ for $q>0$. Then $\beta_v$ has the form \eqref{eq:wpi_v_poly}, and we obtain, with a calculation similar to Example \ref{ex:poly_gaus},
  \[
  \mathcal{H}_\tau(t) \le C \Phi(h_0)\, t^{-q/2}.
  \]
\end{example}

\end{appendix}

\bibliography{kFP-wPI}


\begin{thebibliography}{67}
\ifx \bisbn   \undefined \def \bisbn  #1{ISBN #1}\fi
\ifx \binits  \undefined \def \binits#1{#1}\fi
\ifx \bauthor  \undefined \def \bauthor#1{#1}\fi
\ifx \batitle  \undefined \def \batitle#1{#1}\fi
\ifx \bjtitle  \undefined \def \bjtitle#1{#1}\fi
\ifx \bvolume  \undefined \def \bvolume#1{\textbf{#1}}\fi
\ifx \byear  \undefined \def \byear#1{#1}\fi
\ifx \bissue  \undefined \def \bissue#1{#1}\fi
\ifx \bfpage  \undefined \def \bfpage#1{#1}\fi
\ifx \blpage  \undefined \def \blpage #1{#1}\fi
\ifx \burl  \undefined \def \burl#1{\textsf{#1}}\fi
\ifx \doiurl  \undefined \def \doiurl#1{\url{https://doi.org/#1}}\fi
\ifx \betal  \undefined \def \betal{\textit{et al.}}\fi
\ifx \binstitute  \undefined \def \binstitute#1{#1}\fi
\ifx \binstitutionaled  \undefined \def \binstitutionaled#1{#1}\fi
\ifx \bctitle  \undefined \def \bctitle#1{#1}\fi
\ifx \beditor  \undefined \def \beditor#1{#1}\fi
\ifx \bpublisher  \undefined \def \bpublisher#1{#1}\fi
\ifx \bbtitle  \undefined \def \bbtitle#1{#1}\fi
\ifx \bedition  \undefined \def \bedition#1{#1}\fi
\ifx \bseriesno  \undefined \def \bseriesno#1{#1}\fi
\ifx \blocation  \undefined \def \blocation#1{#1}\fi
\ifx \bsertitle  \undefined \def \bsertitle#1{#1}\fi
\ifx \bsnm \undefined \def \bsnm#1{#1}\fi
\ifx \bsuffix \undefined \def \bsuffix#1{#1}\fi
\ifx \bparticle \undefined \def \bparticle#1{#1}\fi
\ifx \barticle \undefined \def \barticle#1{#1}\fi
\bibcommenthead
\ifx \bconfdate \undefined \def \bconfdate #1{#1}\fi
\ifx \botherref \undefined \def \botherref #1{#1}\fi
\ifx \url \undefined \def \url#1{\textsf{#1}}\fi
\ifx \bchapter \undefined \def \bchapter#1{#1}\fi
\ifx \bbook \undefined \def \bbook#1{#1}\fi
\ifx \bcomment \undefined \def \bcomment#1{#1}\fi
\ifx \oauthor \undefined \def \oauthor#1{#1}\fi
\ifx \citeauthoryear \undefined \def \citeauthoryear#1{#1}\fi
\ifx \endbibitem  \undefined \def \endbibitem {}\fi
\ifx \bconflocation  \undefined \def \bconflocation#1{#1}\fi
\ifx \arxivurl  \undefined \def \arxivurl#1{\textsf{#1}}\fi
\csname PreBibitemsHook\endcsname

\bibitem[\protect\citeauthoryear{Achleitner
  et~al.}{2017}]{achleitner2017optimal}
\begin{bchapter}
\bauthor{\bsnm{Achleitner}, \binits{F.}},
\bauthor{\bsnm{Arnold}, \binits{A.}},
\bauthor{\bsnm{Signorello}, \binits{B.}}:
\bctitle{On optimal decay estimates for {ODEs} and {PDEs} with modal
  decomposition}.
In: \bbtitle{International Workshop on Stochastic Dynamics Out of Equilibrium},
pp. \bfpage{241}--\blpage{264}
(\byear{2017}).
\bcomment{Springer}
\end{bchapter}
\endbibitem

\bibitem[\protect\citeauthoryear{Albritton et~al.}{2024}]{Albritton2019}
\begin{barticle}
\bauthor{\bsnm{Albritton}, \binits{D.}},
\bauthor{\bsnm{Armstrong}, \binits{S.}},
\bauthor{\bsnm{Mourrat}, \binits{J.-C.}},
\bauthor{\bsnm{Novack}, \binits{M.}}:
\batitle{{Variational methods for the kinetic Fokker--Planck equation}}.
\bjtitle{Analysis \& PDE}
\bvolume{17}(\bissue{6}),
\bfpage{1953}--\blpage{2010}
(\byear{2024})
\end{barticle}
\endbibitem

\bibitem[\protect\citeauthoryear{Allen and Tildesley}{2017}]{AT17}
\begin{bbook}
\bauthor{\bsnm{Allen}, \binits{M.P.}},
\bauthor{\bsnm{Tildesley}, \binits{D.J.}}:
\bbtitle{Computer Simulation of Liquids},
\bedition{2nd} edn.
\bpublisher{Oxford University Press, Inc.},
\blocation{Oxford}
(\byear{2017})
\end{bbook}
\endbibitem

\bibitem[\protect\citeauthoryear{Amrouche et~al.}{2015}]{amrouche2015lemma}
\begin{barticle}
\bauthor{\bsnm{Amrouche}, \binits{C.}},
\bauthor{\bsnm{Ciarlet}, \binits{P.G.}},
\bauthor{\bsnm{Mardare}, \binits{C.}}:
\batitle{On a lemma of {J}acques-{L}ouis {L}ions and its relation to other
  fundamental results}.
\bjtitle{Journal de Math{\'e}matiques Pures et Appliqu{\'e}es}
\bvolume{104}(\bissue{2}),
\bfpage{207}--\blpage{226}
(\byear{2015})
\end{barticle}
\endbibitem

\bibitem[\protect\citeauthoryear{Andrieu et~al.}{2021}]{Andrieu2021}
\begin{barticle}
\bauthor{\bsnm{Andrieu}, \binits{C.}},
\bauthor{\bsnm{Dobson}, \binits{P.}},
\bauthor{\bsnm{Wang}, \binits{A.Q.}}:
\batitle{{Subgeometric hypocoercivity for piecewise-deterministic Markov
  process Monte Carlo methods}}.
\bjtitle{Electronic Journal of Probability}
\bvolume{26},
\bfpage{1}--\blpage{26}
(\byear{2021})
\doiurl{10.1214/21-EJP643}
\end{barticle}
\endbibitem

\bibitem[\protect\citeauthoryear{Andrieu et~al.}{2022}]{Andrieu2022}
\begin{barticle}
\bauthor{\bsnm{Andrieu}, \binits{C.}},
\bauthor{\bsnm{Lee}, \binits{A.}},
\bauthor{\bsnm{Power}, \binits{S.}},
\bauthor{\bsnm{Wang}, \binits{A.Q.}}:
\batitle{{Comparison of {M}arkov chains via weak {P}oincar{\'{e}} inequalities
  with application to pseudo-marginal MCMC}}.
\bjtitle{The Annals of Statistics}
\bvolume{50}(\bissue{6}),
\bfpage{3592}--\blpage{3618}
(\byear{2022})
\doiurl{10.1214/22-AOS2241}
\end{barticle}
\endbibitem

\bibitem[\protect\citeauthoryear{Andrieu et~al.}{2023}]{Andrieu2023}
\begin{botherref}
\oauthor{\bsnm{Andrieu}, \binits{C.}},
\oauthor{\bsnm{Lee}, \binits{A.}},
\oauthor{\bsnm{Power}, \binits{S.}},
\oauthor{\bsnm{Wang}, \binits{A.Q.}}:
Weak {P}oincar\'e inequalities for {M}arkov chains: {T}heory and applications.
arXiv preprint
\textbf{2312.11689}
(2023)
\end{botherref}
\endbibitem

\bibitem[\protect\citeauthoryear{Arnold and
  Toshpulatov}{2024}]{arnold2024exponential}
\begin{barticle}
\bauthor{\bsnm{Arnold}, \binits{A.}},
\bauthor{\bsnm{Toshpulatov}, \binits{G.}}:
\batitle{Exponential stability and hypoelliptic regularization for the kinetic
  {F}okker--{P}lanck equation with confining potential}.
\bjtitle{Journal of Statistical Physics}
\bvolume{191}(\bissue{5}),
\bfpage{51}
(\byear{2024})
\end{barticle}
\endbibitem

\bibitem[\protect\citeauthoryear{Bakry et~al.}{2014}]{bakry_analysis_2014}
\begin{bbook}
\bauthor{\bsnm{Bakry}, \binits{D.}},
\bauthor{\bsnm{Gentil}, \binits{I.}},
\bauthor{\bsnm{Ledoux}, \binits{M.}}:
\bbtitle{Analysis and Geometry of {Markov} Diffusion Operators}.
\bpublisher{Springer},
\blocation{Cham; New York}
(\byear{2014})
\end{bbook}
\endbibitem

\bibitem[\protect\citeauthoryear{Barthe et~al.}{2005}]{Barthe2005}
\begin{barticle}
\bauthor{\bsnm{Barthe}, \binits{F.}},
\bauthor{\bsnm{Cattiaux}, \binits{P.}},
\bauthor{\bsnm{Roberto}, \binits{C.}}:
\batitle{{Concentration for independent random variables with heavy tails}}.
\bjtitle{Applied Mathematics Research eXpress}
\bvolume{2005}(\bissue{2}),
\bfpage{39}--\blpage{60}
(\byear{2005})
\doiurl{10.1155/AMRX.2005.39}
\end{barticle}
\endbibitem

\bibitem[\protect\citeauthoryear{Baudoin}{2017}]{baudoin_bakry-emery_2013}
\begin{barticle}
\bauthor{\bsnm{Baudoin}, \binits{F.}}:
\batitle{Bakry--{{\'E}}mery meet {V}illani}.
\bjtitle{Journal of Functional Analysis}
\bvolume{273}(\bissue{7}),
\bfpage{2275}--\blpage{2291}
(\byear{2017})
\end{barticle}
\endbibitem

\bibitem[\protect\citeauthoryear{Bernard
  et~al.}{2022}]{bernard2022hypocoercivity}
\begin{barticle}
\bauthor{\bsnm{Bernard}, \binits{{\'E}.}},
\bauthor{\bsnm{Fathi}, \binits{M.}},
\bauthor{\bsnm{Levitt}, \binits{A.}},
\bauthor{\bsnm{Stoltz}, \binits{G.}}:
\batitle{Hypocoercivity with {S}chur complements}.
\bjtitle{Annales Henri Lebesgue}
\bvolume{5},
\bfpage{523}--\blpage{557}
(\byear{2022})
\end{barticle}
\endbibitem

\bibitem[\protect\citeauthoryear{Bhatia and Davis}{2000}]{Bhatia2000}
\begin{barticle}
\bauthor{\bsnm{Bhatia}, \binits{R.}},
\bauthor{\bsnm{Davis}, \binits{C.}}:
\batitle{{A better bound on the variance}}.
\bjtitle{American Mathematical Monthly}
\bvolume{107}(\bissue{4}),
\bfpage{353}--\blpage{357}
(\byear{2000})
\doiurl{10.2307/2589180}
\end{barticle}
\endbibitem

\bibitem[\protect\citeauthoryear{Blanchet et~al.}{2007}]{blanchet2007hardy}
\begin{barticle}
\bauthor{\bsnm{Blanchet}, \binits{A.}},
\bauthor{\bsnm{Bonforte}, \binits{M.}},
\bauthor{\bsnm{Dolbeault}, \binits{J.}},
\bauthor{\bsnm{Grillo}, \binits{G.}},
\bauthor{\bsnm{V{\'a}zquez}, \binits{J.-L.}}:
\batitle{{H}ardy--{P}oincar{\'e} inequalities and applications to nonlinear
  diffusions}.
\bjtitle{Comptes rendus. Math{\'e}matique}
\bvolume{344}(\bissue{7}),
\bfpage{431}--\blpage{436}
(\byear{2007})
\end{barticle}
\endbibitem

\bibitem[\protect\citeauthoryear{Bobkov and Ledoux}{2009}]{bobkov2009weighted}
\begin{barticle}
\bauthor{\bsnm{Bobkov}, \binits{S.G.}},
\bauthor{\bsnm{Ledoux}, \binits{M.}}:
\batitle{Weighted {P}oincar{\'e}-type inequalities for {C}auchy and other
  convex measures}.
\bjtitle{The Annals of Probability}
\bvolume{37}(\bissue{2}),
\bfpage{403}--\blpage{427}
(\byear{2009})
\end{barticle}
\endbibitem

\bibitem[\protect\citeauthoryear{Boltzmann}{1872}]{boltzmann1872weitere}
\begin{bbook}
\bauthor{\bsnm{Boltzmann}, \binits{L.}}:
\bbtitle{Weitere Studien {\"U}ber Das W{\"a}rmegleichgewicht Unter
  Gasmolek{\"u}len}
vol. \bseriesno{66}.
\bpublisher{Aus der kk Hot-und Staatsdruckerei}, \blocation{???}
(\byear{1872})
\end{bbook}
\endbibitem

\bibitem[\protect\citeauthoryear{Bouin et~al.}{2022}]{bouin2023mathrm}
\begin{bchapter}
\bauthor{\bsnm{Bouin}, \binits{E.}},
\bauthor{\bsnm{Dolbeault}, \binits{J.}},
\bauthor{\bsnm{Ziviani}, \binits{L.}}:
\bctitle{${L}^2$ hypocoercivity methods for kinetic {F}okker-{P}lanck equations
  with factorised {G}ibbs states}.
In: \bbtitle{INdAM Meeting: Kolmogorov Operators and Their Applications
  Workshop},
pp. \bfpage{23}--\blpage{56}
(\byear{2022}).
\bcomment{Springer}
\end{bchapter}
\endbibitem

\bibitem[\protect\citeauthoryear{Bouin et~al.}{2021}]{bouin2021hypocoercivity}
\begin{barticle}
\bauthor{\bsnm{Bouin}, \binits{E.}},
\bauthor{\bsnm{Dolbeault}, \binits{J.}},
\bauthor{\bsnm{Lafleche}, \binits{L.}},
\bauthor{\bsnm{Schmeiser}, \binits{C.}}:
\batitle{Hypocoercivity and sub-exponential local equilibria}.
\bjtitle{Monatshefte f{\"u}r Mathematik}
\bvolume{194},
\bfpage{41}--\blpage{65}
(\byear{2021})
\end{barticle}
\endbibitem

\bibitem[\protect\citeauthoryear{Bre{\v{s}}ar and
  Mijatovi{\'c}}{2024}]{brevsar2024subexponential}
\begin{barticle}
\bauthor{\bsnm{Bre{\v{s}}ar}, \binits{M.}},
\bauthor{\bsnm{Mijatovi{\'c}}, \binits{A.}}:
\batitle{Subexponential lower bounds for $f$-ergodic {M}arkov processes}.
\bjtitle{Probability Theory and Related Fields}
(\byear{2024})
\doiurl{10.1007/s00440-024-01306-z}
\end{barticle}
\endbibitem

\bibitem[\protect\citeauthoryear{Brigati}{2023}]{brigati2023time}
\begin{barticle}
\bauthor{\bsnm{Brigati}, \binits{G.}}:
\batitle{Time averages for kinetic {F}okker-{P}lanck equations}.
\bjtitle{Kinetic \& Related Models}
\bvolume{16}(\bissue{4}),
\bfpage{524}--\blpage{539}
(\byear{2023})
\end{barticle}
\endbibitem

\bibitem[\protect\citeauthoryear{Brigati and Stoltz}{2023}]{Brigati2023}
\begin{botherref}
\oauthor{\bsnm{Brigati}, \binits{G.}},
\oauthor{\bsnm{Stoltz}, \binits{G.}}:
{How to construct decay rates for kinetic Fokker--Planck equations?}
arXiv preprint
\textbf{2302.14506}
(2023)
\end{botherref}
\endbibitem

\bibitem[\protect\citeauthoryear{Brigati
  et~al.}{2024}]{brigati2024hypocoercivity}
\begin{botherref}
\oauthor{\bsnm{Brigati}, \binits{G.}},
\oauthor{\bsnm{L{\"o}rler}, \binits{F.}},
\oauthor{\bsnm{Wang}, \binits{L.}}:
Hypocoercivity meets lifts.
arXiv preprint arXiv:2412.10890
(2024)
\end{botherref}
\endbibitem

\bibitem[\protect\citeauthoryear{Cao}{2019}]{cao2020kinetic}
\begin{barticle}
\bauthor{\bsnm{Cao}, \binits{C.}}:
\batitle{The kinetic {F}okker--{P}lanck equation with weak confinement force}.
\bjtitle{Communications in Mathematical Sciences}
\bvolume{17}(\bissue{8}),
\bfpage{2281}--\blpage{2308}
(\byear{2019})
\end{barticle}
\endbibitem

\bibitem[\protect\citeauthoryear{Cao et~al.}{2023}]{Cao2023}
\begin{barticle}
\bauthor{\bsnm{Cao}, \binits{Y.}},
\bauthor{\bsnm{Lu}, \binits{J.}},
\bauthor{\bsnm{Wang}, \binits{L.}}:
\batitle{{On Explicit $L^2$-Convergence Rate Estimate for Underdamped Langevin
  Dynamics}}.
\bjtitle{Archive for Rational Mechanics and Analysis}
\bvolume{247}(\bissue{5}),
\bfpage{1}--\blpage{34}
(\byear{2023})
\end{barticle}
\endbibitem

\bibitem[\protect\citeauthoryear{Cattiaux
  et~al.}{2010}]{cattiaux2010functional}
\begin{barticle}
\bauthor{\bsnm{Cattiaux}, \binits{P.}},
\bauthor{\bsnm{Gozlan}, \binits{N.}},
\bauthor{\bsnm{Guillin}, \binits{A.}},
\bauthor{\bsnm{Roberto}, \binits{C.}}:
\batitle{Functional inequalities for heavy tailed distributions and application
  to isoperimetry}.
\bjtitle{Electronic Journal of Probability}
\bvolume{15},
\bfpage{346}--\blpage{385}
(\byear{2010})
\end{barticle}
\endbibitem

\bibitem[\protect\citeauthoryear{Cheng et~al.}{2018}]{cheng2018underdamped}
\begin{bchapter}
\bauthor{\bsnm{Cheng}, \binits{X.}},
\bauthor{\bsnm{Chatterji}, \binits{N.S.}},
\bauthor{\bsnm{Bartlett}, \binits{P.L.}},
\bauthor{\bsnm{Jordan}, \binits{M.I.}}:
\bctitle{Underdamped {Langevin MCMC}: A non-asymptotic analysis}.
In: \bbtitle{Conference on Learning Theory},
pp. \bfpage{300}--\blpage{323}
(\byear{2018}).
\bcomment{PMLR}
\end{bchapter}
\endbibitem

\bibitem[\protect\citeauthoryear{Dalalyan and
  Riou-Durand}{2020}]{dalalyan2020sampling}
\begin{barticle}
\bauthor{\bsnm{Dalalyan}, \binits{A.S.}},
\bauthor{\bsnm{Riou-Durand}, \binits{L.}}:
\batitle{On sampling from a log-concave density using kinetic {L}angevin
  diffusions}.
\bjtitle{Bernoulli}
\bvolume{26}(\bissue{3}),
\bfpage{1956}--\blpage{1988}
(\byear{2020})
\end{barticle}
\endbibitem

\bibitem[\protect\citeauthoryear{Degond}{1986}]{degond1986global}
\begin{bchapter}
\bauthor{\bsnm{Degond}, \binits{P.}}:
\bctitle{Global existence of smooth solutions for the
  {V}lasov-{F}okker-{P}lanck equation in $1 $ and $2 $ space dimensions}.
In: \bbtitle{Annales Scientifiques de l'{\'E}cole Normale Sup{\'e}rieure},
vol. \bseriesno{19},
pp. \bfpage{519}--\blpage{542}
(\byear{1986})
\end{bchapter}
\endbibitem

\bibitem[\protect\citeauthoryear{Dietert}{2023}]{Dietert2023}
\begin{botherref}
\oauthor{\bsnm{Dietert}, \binits{H.}}:
{$L^2$-Hypocoercivity for non-equilibrium kinetic equations}.
arXiv preprint
\textbf{2310.13456}
(2023)
\end{botherref}
\endbibitem

\bibitem[\protect\citeauthoryear{Dietert
  et~al.}{2022}]{dietert2022quantitative}
\begin{botherref}
\oauthor{\bsnm{Dietert}, \binits{H.}},
\oauthor{\bsnm{H{\'e}rau}, \binits{F.}},
\oauthor{\bsnm{Hutridurga}, \binits{H.}},
\oauthor{\bsnm{Mouhot}, \binits{C.}}:
Quantitative geometric control in linear kinetic theory.
arXiv preprint
\textbf{2209.09340}
(2022)
\end{botherref}
\endbibitem

\bibitem[\protect\citeauthoryear{Dolbeault
  et~al.}{2009}]{dolbeault_hypocoercivity_2009}
\begin{barticle}
\bauthor{\bsnm{Dolbeault}, \binits{J.}},
\bauthor{\bsnm{Mouhot}, \binits{C.}},
\bauthor{\bsnm{Schmeiser}, \binits{C.}}:
\batitle{Hypocoercivity for kinetic equations with linear relaxation terms}.
\bjtitle{Comptes Rendus Math\'ematique}
\bvolume{347}(\bissue{9}),
\bfpage{511}--\blpage{516}
(\byear{2009})
\doiurl{10.1016/j.crma.2009.02.025}
\end{barticle}
\endbibitem

\bibitem[\protect\citeauthoryear{Dolbeault
  et~al.}{2015}]{dolbeault2015hypocoercivity}
\begin{barticle}
\bauthor{\bsnm{Dolbeault}, \binits{J.}},
\bauthor{\bsnm{Mouhot}, \binits{C.}},
\bauthor{\bsnm{Schmeiser}, \binits{C.}}:
\batitle{Hypocoercivity for linear kinetic equations conserving mass}.
\bjtitle{Transactions of the American Mathematical Society}
\bvolume{367}(\bissue{6}),
\bfpage{3807}--\blpage{3828}
(\byear{2015})
\end{barticle}
\endbibitem

\bibitem[\protect\citeauthoryear{Dolbeault
  et~al.}{2013}]{dolbeault2013exponential}
\begin{barticle}
\bauthor{\bsnm{Dolbeault}, \binits{J.}},
\bauthor{\bsnm{Klar}, \binits{A.}},
\bauthor{\bsnm{Mouhot}, \binits{C.}},
\bauthor{\bsnm{Schmeiser}, \binits{C.}}:
\batitle{Exponential rate of convergence to equilibrium for a model describing
  fiber lay-down processes}.
\bjtitle{Applied Mathematics Research eXpress}
\bvolume{2013}(\bissue{2}),
\bfpage{165}--\blpage{175}
(\byear{2013})
\end{barticle}
\endbibitem

\bibitem[\protect\citeauthoryear{Eberle and L{\"{o}}rler}{2024}]{eberle2024non}
\begin{barticle}
\bauthor{\bsnm{Eberle}, \binits{A.}},
\bauthor{\bsnm{L{\"{o}}rler}, \binits{F.}}:
\batitle{{Non-reversible lifts of reversible diffusion processes and relaxation
  times}}.
\bjtitle{Probability Theory and Related Fields}
(\byear{2024})
\doiurl{10.1007/s00440-024-01308-x}
\end{barticle}
\endbibitem

\bibitem[\protect\citeauthoryear{Eberle and
  L{\"{o}}rler}{2024}]{Eberle2024spacetime}
\begin{botherref}
\oauthor{\bsnm{Eberle}, \binits{A.}},
\oauthor{\bsnm{L{\"{o}}rler}, \binits{F.}}:
{Space-time divergence lemmas and optimal non-reversible lifts of diffusions on
  Riemannian manifolds with boundary}.
arXiv preprint
\textbf{2412.16710}
(2024)
\end{botherref}
\endbibitem

\bibitem[\protect\citeauthoryear{Eberle et~al.}{2019}]{eberle_couplings_2019}
\begin{barticle}
\bauthor{\bsnm{Eberle}, \binits{A.}},
\bauthor{\bsnm{Guillin}, \binits{A.}},
\bauthor{\bsnm{Zimmer}, \binits{R.}}:
\batitle{Couplings and quantitative contraction rates for {Langevin} dynamics}.
\bjtitle{The Annals of Probability}
\bvolume{47}(\bissue{4}),
\bfpage{1982}--\blpage{2010}
(\byear{2019})
\doiurl{10.1214/18-AOP1299}
\end{barticle}
\endbibitem

\bibitem[\protect\citeauthoryear{Eckmann and
  Hairer}{2003}]{eckmann_spectral_2003}
\begin{barticle}
\bauthor{\bsnm{Eckmann}, \binits{J.-P.}},
\bauthor{\bsnm{Hairer}, \binits{M.}}:
\batitle{Spectral properties of hypoelliptic operators}.
\bjtitle{Communications in Mathematical Physics}
\bvolume{235}(\bissue{2}),
\bfpage{233}--\blpage{253}
(\byear{2003})
\doiurl{10.1007/s00220-003-0805-9}
\end{barticle}
\endbibitem

\bibitem[\protect\citeauthoryear{Golse et~al.}{1988}]{golse1988regularity}
\begin{barticle}
\bauthor{\bsnm{Golse}, \binits{F.}},
\bauthor{\bsnm{Lions}, \binits{P.-L.}},
\bauthor{\bsnm{Perthame}, \binits{B.}},
\bauthor{\bsnm{Sentis}, \binits{R.}}:
\batitle{Regularity of the moments of the solution of a transport equation}.
\bjtitle{Journal of Functional Analysis}
\bvolume{76}(\bissue{1}),
\bfpage{110}--\blpage{125}
(\byear{1988})
\end{barticle}
\endbibitem

\bibitem[\protect\citeauthoryear{Grad}{1949}]{grad1949kinetic}
\begin{barticle}
\bauthor{\bsnm{Grad}, \binits{H.}}:
\batitle{On the kinetic theory of rarefied gases}.
\bjtitle{Communications on pure and applied mathematics}
\bvolume{2}(\bissue{4}),
\bfpage{331}--\blpage{407}
(\byear{1949})
\end{barticle}
\endbibitem

\bibitem[\protect\citeauthoryear{Grothaus and
  Stilgenbauer}{2014}]{grothaus_hypocoercivity_2014}
\begin{barticle}
\bauthor{\bsnm{Grothaus}, \binits{M.}},
\bauthor{\bsnm{Stilgenbauer}, \binits{P.}}:
\batitle{Hypocoercivity for {Kolmogorov} backward evolution equations and
  applications}.
\bjtitle{Journal of Functional Analysis}
\bvolume{267}(\bissue{10}),
\bfpage{3515}--\blpage{3556}
(\byear{2014})
\doiurl{10.1016/j.jfa.2014.08.019}
\end{barticle}
\endbibitem

\bibitem[\protect\citeauthoryear{Grothaus and Wang}{2019}]{grothaus2019weak}
\begin{barticle}
\bauthor{\bsnm{Grothaus}, \binits{M.}},
\bauthor{\bsnm{Wang}, \binits{F.-Y.}}:
\batitle{Weak {P}oincar{\'e} inequalities for convergence rate of degenerate
  diffusion processes}.
\bjtitle{The Annals of Probability}
\bvolume{47}(\bissue{5}),
\bfpage{2930}--\blpage{2952}
(\byear{2019})
\end{barticle}
\endbibitem

\bibitem[\protect\citeauthoryear{Helffer and
  Nier}{2005}]{helffer2005hypoelliptic}
\begin{bbook}
\bauthor{\bsnm{Helffer}, \binits{B.}},
\bauthor{\bsnm{Nier}, \binits{F.}}:
\bbtitle{Hypoelliptic Estimates and Spectral Theory for {F}okker-{P}lanck
  Operators and {W}itten {L}aplacians}.
\bsertitle{Lecture Notes in Mathematics},
vol. \bseriesno{1862}.
\bpublisher{Springer},
\blocation{Berlin, Heidelberg}
(\byear{2005})
\end{bbook}
\endbibitem

\bibitem[\protect\citeauthoryear{H{\'e}rau}{2006}]{herau_hypocoercivity_2006}
\begin{barticle}
\bauthor{\bsnm{H{\'e}rau}, \binits{F.}}:
\batitle{Hypocoercivity and exponential time decay for the linear inhomogeneous
  relaxation {Boltzmann} equation}.
\bjtitle{Asymptotic Analysis}
\bvolume{46}(\bissue{3-4}),
\bfpage{349}--\blpage{359}
(\byear{2006})
\end{barticle}
\endbibitem

\bibitem[\protect\citeauthoryear{H{\'e}rau and
  Nier}{2004}]{herau_isotropic_2004}
\begin{barticle}
\bauthor{\bsnm{H{\'e}rau}, \binits{F.}},
\bauthor{\bsnm{Nier}, \binits{F.}}:
\batitle{Isotropic hypoellipticity and trend to equilibrium for the
  {Fokker}-{Planck} equation with a high-degree potential}.
\bjtitle{Archive for Rational Mechanics and Analysis}
\bvolume{171}(\bissue{2}),
\bfpage{151}--\blpage{218}
(\byear{2004})
\doiurl{10.1007/s00205-003-0276-3}
\end{barticle}
\endbibitem

\bibitem[\protect\citeauthoryear{H{\"o}rmander}{1967}]{hormander1967hypoelliptic}
\begin{barticle}
\bauthor{\bsnm{H{\"o}rmander}, \binits{L.}}:
\batitle{Hypoelliptic second order differential equations}.
\bjtitle{Acta Mathematica}
\bvolume{119},
\bfpage{147}--\blpage{171}
(\byear{1967})
\end{barticle}
\endbibitem

\bibitem[\protect\citeauthoryear{Hu and Wang}{2019}]{hu2019subexponential}
\begin{barticle}
\bauthor{\bsnm{Hu}, \binits{S.}},
\bauthor{\bsnm{Wang}, \binits{X.}}:
\batitle{Subexponential decay in kinetic {F}okker--{P}lanck equation: Weak
  hypocoercivity}.
\bjtitle{Bernoulli}
\bvolume{25}(\bissue{1}),
\bfpage{174}--\blpage{188}
(\byear{2019})
\end{barticle}
\endbibitem

\bibitem[\protect\citeauthoryear{{Iacobucci} et~al.}{2019}]{IOS19}
\begin{barticle}
\bauthor{\bsnm{{Iacobucci}}, \binits{A.}},
\bauthor{\bsnm{{Olla}}, \binits{S.}},
\bauthor{\bsnm{{Stoltz}}, \binits{G.}}:
\batitle{{Convergence rates for nonequilibrium {Langevin} dynamics}}.
\bjtitle{Ann. Math. Qu\'ebec}
\bvolume{43}(\bissue{1}),
\bfpage{73}--\blpage{98}
(\byear{2019})
\end{barticle}
\endbibitem

\bibitem[\protect\citeauthoryear{Kolmogorov}{1934}]{kolmogoroff1934zufallige}
\begin{barticle}
\bauthor{\bsnm{Kolmogorov}, \binits{A.}}:
\batitle{Zuf\"allige {B}ewegungen (zur {T}heorie der {B}rownschen {B}ewegung)}.
\bjtitle{Annals of Mathematics}
\bvolume{35},
\bfpage{116}--\blpage{117}
(\byear{1934})
\end{barticle}
\endbibitem

\bibitem[\protect\citeauthoryear{Leimkuhler and Matthews}{2015}]{LM15}
\begin{bbook}
\bauthor{\bsnm{Leimkuhler}, \binits{B.}},
\bauthor{\bsnm{Matthews}, \binits{C.}}:
\bbtitle{Molecular Dynamics: With Deterministic and Stochastic Numerical
  Methods}.
\bpublisher{Springer},
\blocation{Cham}
(\byear{2015})
\end{bbook}
\endbibitem

\bibitem[\protect\citeauthoryear{Leimkuhler and
  Matthews}{2013}]{leimkuhler2013rational}
\begin{barticle}
\bauthor{\bsnm{Leimkuhler}, \binits{B.}},
\bauthor{\bsnm{Matthews}, \binits{C.}}:
\batitle{Rational construction of stochastic numerical methods for molecular
  sampling}.
\bjtitle{Applied Mathematics Research eXpress}
\bvolume{2013}(\bissue{1}),
\bfpage{34}--\blpage{56}
(\byear{2013})
\end{barticle}
\endbibitem

\bibitem[\protect\citeauthoryear{Leli\`evre and Stoltz}{2016}]{LS16}
\begin{barticle}
\bauthor{\bsnm{Leli\`evre}, \binits{T.}},
\bauthor{\bsnm{Stoltz}, \binits{G.}}:
\batitle{Partial differential equations and stochastic methods in molecular
  dynamics}.
\bjtitle{Acta Numerica}
\bvolume{25},
\bfpage{681}--\blpage{880}
(\byear{2016})
\end{barticle}
\endbibitem

\bibitem[\protect\citeauthoryear{Leli{\`e}vre et~al.}{2010}]{LRS10}
\begin{bbook}
\bauthor{\bsnm{Leli{\`e}vre}, \binits{T.}},
\bauthor{\bsnm{Rousset}, \binits{M.}},
\bauthor{\bsnm{Stoltz}, \binits{G.}}:
\bbtitle{Free Energy Computations: A Mathematical Perspective}.
\bpublisher{Imperial College Press},
\blocation{London}
(\byear{2010})
\end{bbook}
\endbibitem

\bibitem[\protect\citeauthoryear{Li and Lu}{2024}]{li2024quantum}
\begin{botherref}
\oauthor{\bsnm{Li}, \binits{B.}},
\oauthor{\bsnm{Lu}, \binits{J.}}:
Quantum space-time {P}oincar\'e inequality for {L}indblad dynamics.
arXiv preprint
\textbf{2406.09115}
(2024)
\end{botherref}
\endbibitem

\bibitem[\protect\citeauthoryear{Liggett}{1991}]{Liggett1991}
\begin{barticle}
\bauthor{\bsnm{Liggett}, \binits{T.M.}}:
\batitle{{$L_2$} rates of convergence for attractive reversible nearest
  particle systems: {T}he critical case}.
\bjtitle{Ann. Probab.}
\bvolume{19}(\bissue{3}),
\bfpage{935}--\blpage{959}
(\byear{1991})
\doiurl{10.1214/AOP/1176990330}
\end{barticle}
\endbibitem

\bibitem[\protect\citeauthoryear{Lu and Wang}{2022}]{lu2022explicit}
\begin{barticle}
\bauthor{\bsnm{Lu}, \binits{J.}},
\bauthor{\bsnm{Wang}, \binits{L.}}:
\batitle{On explicit {$L^2$}-convergence rate estimate for piecewise
  deterministic {M}arkov processes in {MCMC} algorithms}.
\bjtitle{The Annals of Applied Probability}
\bvolume{32}(\bissue{2}),
\bfpage{1333}--\blpage{1361}
(\byear{2022})
\end{barticle}
\endbibitem

\bibitem[\protect\citeauthoryear{MA et~al.}{2021}]{ma2021there}
\begin{barticle}
\bauthor{\bsnm{MA}, \binits{Y.-A.}},
\bauthor{\bsnm{CHATTERJI}, \binits{N.S.}},
\bauthor{\bsnm{CHENG}, \binits{X.}},
\bauthor{\bsnm{FLAMMARION}, \binits{N.}},
\bauthor{\bsnm{BARTLETT}, \binits{P.L.}},
\bauthor{\bsnm{JORDAN}, \binits{M.I.}}:
\batitle{Is there an analog of {Nesterov acceleration for gradient-based
  MCMC}?}
\bjtitle{Bernoulli}
\bvolume{27}(\bissue{3}),
\bfpage{1942}--\blpage{1992}
(\byear{2021})
\end{barticle}
\endbibitem

\bibitem[\protect\citeauthoryear{Mattingly
  et~al.}{2002}]{mattingly_ergodicity_2002}
\begin{barticle}
\bauthor{\bsnm{Mattingly}, \binits{J.C.}},
\bauthor{\bsnm{Stuart}, \binits{A.M.}},
\bauthor{\bsnm{Higham}, \binits{D.J.}}:
\batitle{Ergodicity for {SDEs} and approximations: locally {Lipschitz} vector
  fields and degenerate noise}.
\bjtitle{Stochastic Processes and their Applications}
\bvolume{101}(\bissue{2}),
\bfpage{185}--\blpage{232}
(\byear{2002})
\doiurl{10.1016/S0304-4149(02)00150-3}
\end{barticle}
\endbibitem

\bibitem[\protect\citeauthoryear{Mouhot and
  Neumann}{2006}]{mouhot2006quantitative}
\begin{barticle}
\bauthor{\bsnm{Mouhot}, \binits{C.}},
\bauthor{\bsnm{Neumann}, \binits{L.}}:
\batitle{Quantitative perturbative study of convergence to equilibrium for
  collisional kinetic models in the torus}.
\bjtitle{Nonlinearity}
\bvolume{19}(\bissue{4}),
\bfpage{969}
(\byear{2006})
\end{barticle}
\endbibitem

\bibitem[\protect\citeauthoryear{Muckenhoupt}{1972}]{Muckenhoupt1972}
\begin{barticle}
\bauthor{\bsnm{Muckenhoupt}, \binits{B.}}:
\batitle{Hardy's inequality with weights}.
\bjtitle{Studia Mathematica}
\bvolume{44}(\bissue{1}),
\bfpage{31}--\blpage{38}
(\byear{1972})
\end{barticle}
\endbibitem

\bibitem[\protect\citeauthoryear{Rey-Bellet}{2006}]{bellet2006ergodic}
\begin{bchapter}
\bauthor{\bsnm{Rey-Bellet}, \binits{L.}}:
\bctitle{Ergodic properties of {M}arkov processes}.
In: \beditor{\bsnm{Attal}, \binits{S.}},
\beditor{\bsnm{Joye}, \binits{A.}},
\beditor{\bsnm{Pillet}, \binits{C.-A.}} (eds.)
\bbtitle{Open Quantum Systems II}.
\bsertitle{Lecture Notes in Mathematics},
vol. \bseriesno{1881},
pp. \bfpage{1}--\blpage{39}.
\bpublisher{Springer},
\blocation{Berlin}
(\byear{2006})
\end{bchapter}
\endbibitem

\bibitem[\protect\citeauthoryear{R{\"o}ckner and Wang}{2001}]{rockner2001weak}
\begin{barticle}
\bauthor{\bsnm{R{\"o}ckner}, \binits{M.}},
\bauthor{\bsnm{Wang}, \binits{F.-Y.}}:
\batitle{Weak {P}oincar{\'e} inequalities and ${L^2}$-convergence rates of
  {M}arkov semigroups}.
\bjtitle{Journal of Functional Analysis}
\bvolume{185}(\bissue{2}),
\bfpage{564}--\blpage{603}
(\byear{2001})
\end{barticle}
\endbibitem

\bibitem[\protect\citeauthoryear{Shen and Lee}{2019}]{shen2019randomized}
\begin{botherref}
\oauthor{\bsnm{Shen}, \binits{R.}},
\oauthor{\bsnm{Lee}, \binits{Y.T.}}:
The randomized midpoint method for log-concave sampling.
Advances in Neural Information Processing Systems
\textbf{32}
(2019)
\end{botherref}
\endbibitem

\bibitem[\protect\citeauthoryear{Stoltz and
  Trstanova}{2018}]{stoltz2018langevin}
\begin{barticle}
\bauthor{\bsnm{Stoltz}, \binits{G.}},
\bauthor{\bsnm{Trstanova}, \binits{Z.}}:
\batitle{Langevin dynamics with general kinetic energies}.
\bjtitle{Multiscale Modeling \& Simulation}
\bvolume{16}(\bissue{2}),
\bfpage{777}--\blpage{806}
(\byear{2018})
\end{barticle}
\endbibitem

\bibitem[\protect\citeauthoryear{Talay}{2002}]{talay2002stochastic}
\begin{barticle}
\bauthor{\bsnm{Talay}, \binits{D.}}:
\batitle{Stochastic {H}amiltonian systems: exponential convergence to the
  invariant measure, and discretization by the implicit {E}uler scheme}.
\bjtitle{Markov Processes and Related Fields}
\bvolume{8}(\bissue{2}),
\bfpage{163}--\blpage{198}
(\byear{2002})
\end{barticle}
\endbibitem

\bibitem[\protect\citeauthoryear{Teschl}{2014}]{teschl2014mathematical}
\begin{bbook}
\bauthor{\bsnm{Teschl}, \binits{G.}}:
\bbtitle{Mathematical Methods in Quantum Mechanics}.
\bsertitle{Graduate Studies in Mathematics},
vol. \bseriesno{157}.
\bpublisher{American Mathematical Society},
\blocation{Providence, RI}
(\byear{2014})
\end{bbook}
\endbibitem

\bibitem[\protect\citeauthoryear{Villani}{2009}]{villani_hypocoercivity_2009}
\begin{botherref}
\oauthor{\bsnm{Villani}, \binits{C.}}:
Hypocoercivity.
Memoirs of the American Mathematical Society
\textbf{202}(950)
(2009)
\end{botherref}
\endbibitem

\bibitem[\protect\citeauthoryear{Wu}{2001}]{wu2001large}
\begin{barticle}
\bauthor{\bsnm{Wu}, \binits{L.}}:
\batitle{Large and moderate deviations and exponential convergence for
  stochastic damping {H}amiltonian systems}.
\bjtitle{Stochastic processes and their applications}
\bvolume{91}(\bissue{2}),
\bfpage{205}--\blpage{238}
(\byear{2001})
\end{barticle}
\endbibitem

\end{thebibliography}

\end{document}